\documentclass[final,reqno]{amsart}

\usepackage[margin=1.5in,bottom=1.15in]{geometry}	

\usepackage{bm}

\usepackage{amsmath}	
\usepackage{amssymb}	
\usepackage{amsfonts}	
\usepackage{amsthm}	
\usepackage[foot]{amsaddr}	
\usepackage{stmaryrd}
\usepackage[centercolon=true]{mathtools}	
\usepackage{breqn}
\usepackage[titletoc,toc]{appendix}
\usepackage{chngcntr}
\usepackage{apptools}
\usepackage{threeparttable}
\AtAppendix{\counterwithin{lemma}{section}}
\AtAppendix{\counterwithin{theorem}{section}}
\usepackage[normalem]{ulem}
\mathtoolsset{%
}

\usepackage[utf8]{inputenc}	
\usepackage[T1]{fontenc}	

\usepackage[
cal=cm,
]
{mathalfa}

\usepackage{dsfont}	

\usepackage[light,scaled=.95]{AlegreyaSans}

\usepackage[proportional,tabular,lining,sf,mono=false]{libertine}

\usepackage{acronym}	

\usepackage[labelfont={bf,small},labelsep=colon,font=small]{caption}	

\usepackage{subcaption}	
 \usepackage{float}

\usepackage[svgnames]{xcolor}	
\colorlet{MyRed}{Crimson!60!DarkRed}
\colorlet{MyBlue}{DodgerBlue!75!black}
\colorlet{MyGreen}{DarkGreen}
\colorlet{MyViolet}{DarkMagenta}

\colorlet{MyLightBlue}{DodgerBlue!20}
\colorlet{MyLightGreen}{MyGreen!20}

\colorlet{PrimalColor}{MyBlue}
\colorlet{PrimalFill}{MyLightBlue}
\colorlet{DualColor}{MyRed}

\colorlet{AlertColor}{MyRed}	
\colorlet{BadColor}{MyRed}	
\colorlet{GoodColor}{MyGreen}	
\colorlet{LinkColor}{MediumBlue}	
\colorlet{MacroColor}{MyViolet}
\colorlet{RevColor}{MediumBlue}	

\usepackage{latexsym}	
\usepackage{fontawesome}	
\usepackage{pifont}	

\usepackage{tikz}	
\usetikzlibrary{calc,patterns}	
\usepackage{pgfplots}
\usetikzlibrary{datavisualization} 
\usetikzlibrary{datavisualization.formats.functions} 
\usetikzlibrary{positioning}
\usepackage{array}	
\usepackage{booktabs}	
\usepackage[inline,shortlabels]{enumitem}	
\setlist[1]{topsep=\smallskipamount,itemsep=\smallskipamount,left=\parindent}
\setlist[2]{left=0pt}

\usepackage[kerning=true]{microtype}	
\usepackage{ifluatex}
\ifluatex
  \usepackage{microtype}
  \DisableLigatures{encoding = *, family = *}
\else
  \usepackage{microtype}
\fi

\usepackage{tabto}	
\usepackage{xspace}	

\usepackage[sort&compress]{natbib}	

\bibpunct[, ]{[}{]}{,}{}{,}{,}
\setcitestyle{numbers,square}

\usepackage{hyperref}
\hypersetup{
final,
colorlinks=true,
linktocpage=true,
pdfstartview=FitH,
breaklinks=true,
pdfpagemode=UseNone,
pageanchor=true,
pdfpagemode=UseOutlines,
plainpages=false,
bookmarksnumbered,
bookmarksopen=false,
bookmarksopenlevel=1,
hypertexnames=true,
pdfhighlight=/O,
hyperfootnotes=false,
urlcolor=LinkColor,linkcolor=LinkColor,citecolor=LinkColor,	
pdftitle={},
pdfauthor={},
pdfsubject={},
pdfkeywords={},
pdfcreator={pdfLaTeX},
pdfproducer={LaTeX with hyperref}
}

\usepackage[sort&compress,capitalize,nameinlink]{cleveref}	
\crefname{algo}{Algorithm}{Algorithms}
\crefname{assumption}{Assumption}{Assumptions}
\crefname{case}{Case}{Cases}
\crefname{cond}{Condition}{Conditions}
\crefname{noref}{}{}
\crefname{problem}{Problem}{Problems}

\makeatletter	
\DeclareRobustCommand{\crefnosort}[1]{%
  \begingroup\@cref@sortfalse\cref{#1}\endgroup
}
\makeatother

\usepackage{algorithm}	
\usepackage{algpseudocode}	

\theoremstyle{plain}
\newtheorem{theorem}{Theorem}	
\newtheorem{lemma}{Lemma}	
\newtheorem{proposition}{Proposition}	

\newtheorem*{theorem*}{Theorem}	
\newtheorem*{corollary*}{Corollary}	

\theoremstyle{definition}
\newtheorem{assumption}{Assumption}	
\newtheorem{definition}{Definition}	

\newtheorem*{assumption*}{Assumptions}	
\newtheorem*{definition*}{Definition}	
\newtheorem*{example*}{\raisebox{\depth}{$\blacktriangleright$}~Example}	

\theoremstyle{remark}
\newtheorem{remark}{Remark}	

\newtheorem*{remark*}{Remark}	
\newtheorem*{notation*}{Notation}	

\newcounter{proofstep}

\numberwithin{remark}{section}	
\numberwithin{example}{section}	

\usepackage[showdeletions,suppress]{color-edits}	
\newcommand{\draft}[1]{#1}	

\newcommand{\newmacro}[2]{\newcommand{#1}{\draft{#2}}}	
\newcommand{\newop}[2]{\DeclareMathOperator{#1}{\draft{#2}}}	
\newcommand{\newoplims}[2]{\DeclareMathOperator*{#1}{\draft{#2}}}	
\newcommand{\newsmartmacro}[2]{
	\NewDocumentCommand{#1}{
		E{_}{{}}
	}{
		\draft{#2_{##1}}
	}
}

\newcommand{\eps}{\varepsilon}	

\DeclarePairedDelimiterX{\setdef}[2]{\{}{\}}{#1:#2}	
\DeclarePairedDelimiterXPP{\exclude}[1]{\mathopen{}\setminus}{\{}{\}}{}{#1}

\DeclarePairedDelimiterX{\braket}[2]{\langle}{\rangle}{#1,#2}	
\DeclarePairedDelimiterX{\internalInner}[1]{\langle}{\rangle}{#1}
\usepackage{xparse}
\NewDocumentCommand \inner {s m g}{
	\IfBooleanTF{#1}
	{
	\internalInner*{#2\IfValueT{#3}{,#3}}
	}
	{
	\internalInner{#2\IfValueT{#3}{,#3}}
	}
}

\DeclarePairedDelimiterXPP{\dnorm}[1]{}{\lVert}{\rVert}{_{\ast}}{#1}	
\DeclarePairedDelimiterXPP{\onenorm}[1]{}{\lVert}{\rVert}{_{1}}{#1}	
\DeclarePairedDelimiterXPP{\twonorm}[1]{}{\lVert}{\rVert}{_{2}}{#1}	
\DeclarePairedDelimiterXPP{\supnorm}[1]{}{\lVert}{\rVert}{_{\infty}}{#1}	

\newmacro{\nat}{i}	
\newmacro{\natA}{i}	
\newmacro{\natB}{j}	
\newmacro{\natC}{k}	
\newmacro{\nats}{\mathbb{N}}	
\newmacro{\N}{\nats}	

\newmacro{\integer}{a}	
\newmacro{\intA}{a}	
\newmacro{\intB}{b}	
\newmacro{\intC}{c}	
\newmacro{\integers}{\mathbb{Z}}	
\newmacro{\Z}{\integers}	

\newmacro{\rational}{r}	
\newmacro{\ratA}{r}	
\newmacro{\ratB}{s}	
\newmacro{\ratC}{t}	
\newmacro{\rationals}{\mathbb{Q}}	
\newmacro{\Q}{\mathcal Q}	

\newmacro{\real}{x}	
\newmacro{\realA}{x}	
\newmacro{\realB}{y}	
\newmacro{\realC}{z}	
\newmacro{\reals}{\mathbb{R}}	
\newmacro{\R}{\reals}	

\newmacro{\complex}{z}	
\newmacro{\complexA}{z}	
\newmacro{\complexB}{w}	
\newmacro{\complexC}{z}	
\newmacro{\complexes}{\mathbb{C}}	
\newmacro{\C}{\mathbb{C}}	

\newoplims{\argmax}{arg\,max}	
\newoplims{\argmin}{arg\,min}	
\newoplims{\intersect}{\bigcap}	
\newoplims{\union}{\bigcup}	

\newop{\aff}{aff}	
\newop{\bd}{bd}	
\newop{\bigoh}{\mathcal{O}}	
\newop{\card}{card}	
\newop{\cl}{cl}	
\newop{\conv}{conv}	
\newop{\clconv}{\overline{conv}}	
\newop{\crit}{crit}	
\newop{\diag}{diag}	
\newop{\diam}{diam}	
\newop{\dist}{dist}	
\newop{\dom}{dom}	
\newop{\gra}{\textup{gra}\:}
\newop{\zer}{\textup{zer}}
\newop{\Fix}{\textup{Fix}}
\newop{\eig}{eig}	
\newop{\ess}{ess}	
\newop{\Hess}{Hess}	
\newop{\ind}{ind}	
\newop{\im}{Im}	
\newop{\intr}{int}	
\newop{\Jac}{Jac}	
\newop{\one}{\mathds{1}}	
\newop{\proj}{pr}	
\newop{\prox}{prox}	
\newop{\rank}{rank}	
\newop{\relint}{ri}	
\newop{\sign}{sgn}	
\newop{\supp}{supp}	
\newop{\Sym}{Sym}	
\newop{\tr}{tr}	
\newop{\unif}{unif}	
\newop{\vol}{vol}	

\newop{\Tx}{\mathcal{T}_{\mm_x}}	
\newop{\Txb}{\mathcal{T}_{\mm_{\bar x}}}	
\DeclareMathOperator{\dive}{div}
\newcommand{\de}{\mathbf{D}}
\newcommand{\he}{\mathbf{H}}
\newcommand{\D}{\mathcal D}
\newcommand{\K}{\mathcal K}
\newcommand{\Om}{\Omega}
\newcommand{\CC}{\mathcal C}

\newcommand{\T}{\mathcal T}
\newcommand{\A}{\mathcal A}
\newcommand{\B}{\mathcal B}

\newcommand{\pp}{a.e.\,}
\newcommand{\M}{\textup{(M)}}

\newcommand{\cf}{cf.\xspace}	
\newcommand{\eg}{e.g.,\xspace}	
\newcommand{\ie}{i.e.,\xspace}	

\newcommand{\alt}[1]{#1'}	
\newcommand{\altalt}[1]{#1''}	

\newmacro{\ball}{\mathbb{B}}	
\newmacro{\clball}{\overline{\mathbb{B}}}	
\newmacro{\sphere}{\mathbb{S}}	

\newmacro{\argdot}{\kern.5pt\boldsymbol{\cdot}\kern.5pt}	
\newmacro{\ddt}{\frac{d}{dt}}	
\newmacro{\del}{\partial}	

\newmacro{\const}{c}	
\newmacro{\constA}{a}	
\newmacro{\constB}{b}	
\newmacro{\Const}{C}	

\newmacro{\param}{\theta}	
\newmacro{\params}{\Theta}	

\newmacro{\coef}{\alpha}	
\newmacro{\coefA}{\lambda}	
\newmacro{\coefB}{\mu}	
\newmacro{\coefC}{\nu}	

\newmacro{\expA}{p}	
\newmacro{\expB}{q}	
\newmacro{\expC}{r}	

\newmacro{\precs}{\eps}		
\newmacro{\precsalt}{\delta}	

\newmacro{\asympteq}{\asymp}

\newmacro{\func}{g} 

\newmacro{\realspace}{\R^{\vdim}}	
\newmacro{\vecspace}{\mathcal{X}}	
\newmacro{\dspace}{\R^{\vdim}}	
\newmacro{\subspace}{\mathcal{W}}	

\newmacro{\coord}{i}	
\newmacro{\coordA}{i}	
\newmacro{\coordB}{j}	
\newmacro{\coordC}{k}	
\newmacro{\nCoords}{d}	
\newmacro{\dims}{\nCoords}	
\newmacro{\vdim}{\nCoords}	

\newmacro{\bvec}{e}	
\newmacro{\uvec}{u}	
\newmacro{\bvecs}{\mathcal{E}}	

\newmacro{\point}{x}	
\newmacro{\pointA}{\point}	
\newmacro{\pointB}{\point'}	
\newmacro{\pointalt}{\pointB}
\newmacro{\pointC}{\point''}	
\newmacro{\pointaltalt}{\pointC}
\newmacro{\points}{\mathcal{X}}	
\newmacro{\intpoints}{\relint\points}	

\newmacro{\base}{p}	
\newmacro{\baseA}{q}	
\newmacro{\baseB}{q}	
\newmacro{\baseC}{u}	

\newmacro{\set}{\mathcal{K}}	
\newmacro{\setA}{\set}	
\newmacro{\setB}{\alt\set}	
\newmacro{\setC}{\altalt\set}	

\newmacro{\idx}{i}
\newmacro{\idxalt}{j}
\newmacro{\idxaltalt}{l}
\newmacro{\idxaltaltalt}{k}
\newmacro{\indices}{I}
\newmacro{\indicesalt}{J}

\newmacro{\closed}{\mathcal{C}}	
\newmacro{\cpt}{\mathcal{K}}	
\newmacro{\cptalt}{\alt\cpt}	
\newmacro{\nhd}{\mathcal{U}}	
\newmacro{\nhdalt}{\W}	
\newmacro{\nhdaltalt}{\V}	
\newmacro{\nbd}{\nhd}
\newmacro{\nbdalt}{\nhdalt}
\newmacro{\nbdaltalt}{\nhdaltalt}
\newmacro{\U}{\mathcal{U}}	
\newmacro{\V}{\mathcal{V}}	
\newmacro{\W}{\mathcal{W}}	
\newmacro{\open}{\mathcal{U}}	
\newmacro{\openA}{\mathcal{U}}	
\newmacro{\openB}{\mathcal{V}}	

\newmacro{\mfld}{\mathcal{M}}	
\newmacro{\gmat}{g}	
\newmacro{\gdist}{\dist_{\gmat}}	

\newmacro{\tvec}{z}	
\newmacro{\form}{\omega}	

\newmacro{\radius}{r}
\newmacro{\Radius}{R}
\newmacro{\Radiusalt}{\widetilde{\Radius}}
\newmacro{\radiusalt}{\alt\radius}
\newmacro{\margin}{\delta}
\newmacro{\marginalt}{\alt\margin}
\newmacro{\Margin}{\Delta}
\newmacro{\connectedcomp}{\mathcal{K}}
\newmacro{\plainset}{S} 
\newmacro{\interv}{A} 
\newmacro{\domain}{D}
\newmacro{\bigcpt}{\mathcal{D}}
\newsmartmacro{\bigcptalt}{\alt\bigcpt}
\newsmartmacro{\bigcptaltalt}{\alt\alt\bigcpt}

\newmacro{\cvx}{\mathcal{C}}	
\newmacro{\subd}{\partial}	

\newop{\tspace}{T}	
\newop{\tcone}{TC}	
\newop{\dcone}{\tcone^{\ast}}	
\newop{\ncone}{NC}	
\newop{\pcone}{PC}	
\newop{\hull}{\Delta}	

\newop{\minimize}{minimize}	
\newop{\Opt}{Opt}	
\newop{\Sol}{Sol}	
\newop{\gap}{Gap}	
\newop{\orcl}{\mathsf{G}}	
\newop{\err}{\mathsf{Z}}	

\newmacro{\obj}{f}	
\newmacro{\objalt}{g}	
\newmacro{\objA}{f}	
\newmacro{\objB}{g}	
\newmacro{\sobj}{F}	

\newcommand{\sol}[1][\point]{#1^{\ast}}	

\newmacro{\gvec}{g}	
\newmacro{\gbound}{G}	

\newmacro{\oper}{A}	
\newmacro{\vecfield}{v}	
\newmacro{\vbound}{V}	

\newmacro{\lips}{L}	
\newmacro{\strong}{\alpha}	
\newmacro{\smooth}{\beta}	
\newmacro{\tmplips}{L}	
\newmacro{\tmpbound}{B}	
\newmacro{\growth}{M}	
\newmacro{\regparam}{\lambda}	

\newop{\ex}{\mathbb{E}}	
\newop{\prob}{\mathbb{P}}	
\newop{\probalt}{\mathbb{Q}}	
\newop{\var}{\mathbb{V}}	
\newop{\cov}{cov}	
\newop{\simplex}{\hull}	

\DeclarePairedDelimiterXPP{\exof}[1]{\ex}{[}{]}{}{
 #1}

\DeclarePairedDelimiterXPP{\exwrt}[2]{\ex_{#1}}{[}{]}{}{
 #2}

\DeclarePairedDelimiterXPP{\probof}[1]{\prob}{(}{)}{}{
 #1}

\DeclarePairedDelimiterXPP{\probwrt}[2]{\prob_{\!#1}}{(}{)}{}{
 #2}

\DeclarePairedDelimiterXPP{\oneof}[1]{\one}{\{}{\}}{}{#1}	

\DeclarePairedDelimiterXPP{\varof}[1]{\var}{[}{]}{}{
 #1}

\DeclarePairedDelimiterXPP{\covof}[1]{\cov}{(}{)}{}{
 #1}

\newmacro{\event}{E}	
\newmacro{\eventA}{E}	
\newmacro{\eventB}{H}	

\newmacro{\seed}{\theta}	
\newmacro{\seeds}{\Theta}	
\newmacro{\pdist}{P}	
\newmacro{\history}{\mathcal{H}}	

\newmacro{\sample}{\omega}	
\newmacro{\samples}{\Omega}	
\newmacro{\sspace}{\R^{m}}	

\newmacro{\filter}{\mathcal{F}}	
\newmacro{\probspace}{(\samples,\filter,\prob)}	

\newmacro{\mean}{\mu}	
\newmacro{\sdev}{\sigma}	
\newmacro{\variance}{\sdev^{2}}	
\newmacro{\variancealt}{s^2}

\newmacro{\covmat}{\Sigma}
\newmacro{\hessmat}{H}

\newmacro{\rv}{X}
\newmacro{\trv}{X_\Radius}
\newmacro{\gaussian}{\mathcal{N}}

\newmacro{\partition}{Z}

\newmacro{\nSamples}{n}	
\newmacro{\datapoint}{\xi}

\newmacro{\beforestart}{-1}	
\newmacro{\start}{0}	
\newmacro{\afterstart}{1}	
\newmacro{\running}{\start,\afterstart,\dotsc}	

\newmacro{\run}{n}	
\newmacro{\runA}{n}	
\newmacro{\runB}{k}	
\newmacro{\runC}{\ell}	
\newmacro{\nRuns}{N}	
\newmacro{\nRunsalt}{\alt\nRuns}	
\newmacro{\runs}{\mathcal{\nRuns}}	
\newmacro{\runalt}{\runB}	

\newmacro{\tstart}{0}	
\newmacro{\timeA}{t}	
\newmacro{\timeB}{s}	
\newmacro{\timealt}{\timeB}	
\newmacro{\timealtalt}{u}
\newmacro{\timeC}{\tau}	
\newmacro{\timeD}{\lambda}	
\newmacro{\horizon}{T}	
\newmacro{\horizonalt}{S}	

\newmacro{\seq}{a}	
\newmacro{\seqA}{a}	
\newmacro{\seqB}{b}	
\newmacro{\seqC}{c}	

\newmacro{\state}{x}	
\newsmartmacro{\accstate}{x^{\step}}	
\newmacro{\stateA}{x}	
\newmacro{\stateB}{z}	
\newmacro{\statealt}{\stateB}
\newmacro{\stateC}{y}	
\newmacro{\stateD}{p}	
\newmacro{\statealtalt}{\stateC}
\newmacro{\statealtaltalt}{\stateD}

\newmacro{\cstate}{X}
\newmacro{\cstateA}{X}
\newmacro{\cstateB}{Z}
\newmacro{\cstatealt}{\cstateB}

\newmacro{\startingpoint}{\point_\start}	

\newcommand{\curr}[1][\state]{\draft{#1_{\run}}}	

\newmacro{\mat}{M}	
\newmacro{\hmat}{H}	

\newmacro{\ones}{\mathbf{1}}	
\newmacro{\eye}{I}	
\newmacro{\identity}{\eye}	

\newmacro{\eigval}{\lambda}	

\newop{\Nash}{NE}	
\newop{\CE}{CE}	
\newop{\CCE}{CCE}	
\newop{\NI}{NI}	

\newop{\brep}{br}	
\newop{\reg}{Reg}	
\newop{\preg}{\overline{Reg}}	
\newop{\val}{val}	

\newmacro{\play}{i}	
\newmacro{\playA}{i}	
\newmacro{\playB}{j}	
\newmacro{\playC}{k}	
\newmacro{\nPlayers}{N}	
\newmacro{\players}{\mathcal{\nPlayers}}	

\newmacro{\pure}{\alpha}	
\newmacro{\pureA}{\alpha}	
\newmacro{\pureB}{\beta}	
\newmacro{\pureC}{\gamma}	
\newmacro{\nPures}{A}	
\newmacro{\pures}{\mathcal{\nPures}}	

\newmacro{\strat}{x}	
\newmacro{\stratA}{x}	
\newmacro{\stratB}{\stratA'}	
\newmacro{\stratC}{\stratA''}	
\newmacro{\strats}{\mathcal{X}}	
\newmacro{\intstrats}{\strats^{\circle}}	

\newmacro{\pay}{u}	
\newmacro{\payv}{v}	
\newmacro{\payfield}{v}	
\newmacro{\loss}{\ell}	

\newmacro{\game}{\mathcal{G}}	
\newmacro{\gameFull}{\game(\players,\points,\pay)}	

\newmacro{\fingame}{\Gamma}	
\newmacro{\fingameFull}{\Gamma(\players,\pures,\pay)}	

\newmacro{\minmax}{L}	
\newmacro{\minvar}{\point_{1}}	
\newmacro{\minvarA}{\point_{1}}	
\newmacro{\minvarB}{\minvarA'}	
\newmacro{\minvars}{\points_{1}}	
\newmacro{\maxvar}{\point_{2}}	
\newmacro{\maxvarA}{\point_{2}}	
\newmacro{\maxvarB}{\maxvarA'}	
\newmacro{\maxvars}{\points_{2}}	

\newmacro{\pot}{U}	

\newmacro{\gradientbound}{16 \dims \log 6}	

\newmacro{\hreg}{h}	
\newmacro{\breg}{D}	
\newmacro{\mprox}{P}	
\newmacro{\mirror}{Q}	
\newmacro{\fench}{F}	
\newmacro{\hstr}{K}	
\newmacro{\hrange}{H}	
\newmacro{\proxdom}{\points^{\hreg}}	

\DeclarePairedDelimiterXPP{\bregof}[2]{\breg}{(}{)}{}{#1,#2}	
\DeclarePairedDelimiterXPP{\proxof}[2]{\mprox_{#1}}{(}{)}{}{#2}	

\newmacro{\dpoint}{y}	
\newmacro{\dpointA}{y}	
\newmacro{\dpointB}{\dpointA'}	
\newmacro{\dpointC}{\dpointA''}	
\newmacro{\dpoints}{\mathcal{Y}}	

\newmacro{\dstate}{Y}	
\newmacro{\dvec}{w}	

\newmacro{\zone}{\mathbb{D}}	

\newop{\Eucl}{\Pi}	
\newop{\logit}{softmax}	
\newop{\dkl}{KL}	

\newmacro{\flowmap}{\Theta}	
\DeclarePairedDelimiterXPP{\flowof}[2]{\flowmap_{#1}}{(}{)}{}{#2}	
\DeclarePairedDelimiterXPP{\dotflowof}[2]{\dot\flowmap_{#1}}{(}{)}{}{#2}	

\newmacro{\traj}{x}	
\DeclarePairedDelimiterXPP{\trajof}[1]{\traj}{(}{)}{}{#1}	
\DeclarePairedDelimiterXPP{\difftrajof}[1]{\dot\traj}{(}{)}{}{#1}	

\newcommand{\est}[1]{\hat #1}	

\newmacro{\signal}{\est\gvec}	
\newmacro{\step}{\eta}	
\newmacro{\learn}{\eta}	
\newmacro{\tempinv}{\beta}	
\newmacro{\batch}{B}
\newmacro{\batchidx}{b}

\newmacro{\efftime}{\tau}	

\newmacro{\error}{Z}	

\newmacro{\noise}{\mathsf{U}}	
\newmacro{\snoise}{\xi}	
\newmacro{\noisepar}{\sdev}	
\newmacro{\noisevar}{\variance}	
\newmacro{\aggnoise}{\mathrm{\uppercase\expandafter{\romannumeral1}}}	
\newmacro{\supnoise}{\aggnoise_{\infty}}	
\newmacro{\maxnoise}{\aggnoise^{\ast}}	

\newmacro{\bias}{b}	
\newmacro{\drift}{b}	
\newmacro{\bbound}{B}	
\newmacro{\sbias}{\chi}	
\newmacro{\aggbias}{\mathrm{\uppercase\expandafter{\romannumeral2}}}	
\newmacro{\supbias}{\aggbias_{\infty}}	
\newmacro{\maxbias}{\aggbias^{\ast}}	

\newmacro{\sbound}{M}	
\newmacro{\aggsecond}{\mathrm{\uppercase\expandafter{\romannumeral3}}}	
\newmacro{\supsecond}{\aggsecond_{\infty}}	
\newmacro{\maxsecond}{\aggsecond^{\ast}}	

\newmacro{\mgf}{M}	
\DeclarePairedDelimiterXPP{\mgfof}[2]{\mgf_{#1}}{(}{)}{}{#2}	

\newmacro{\cgf}{K}	
\DeclarePairedDelimiterXPP{\cgfof}[2]{\cgf_{#1}}{(}{)}{}{#2}	

\newmacro{\ham}{\mathcal{H}}	
\DeclarePairedDelimiterXPP{\hamof}[2]{\ham_{#1}}{(}{)}{}{#2}	

\newmacro{\lag}{\mathcal{L}}	
\DeclarePairedDelimiterXPP{\lagof}[2]{\lag_{#1}}{(}{)}{}{#2}	

\newmacro{\mom}{p}
\newmacro{\pos}{q}
\newmacro{\vel}{v}
\newmacro{\velalt}{w}

\newmacro{\hamilt}{\mathcal{H}}
\newmacro{\hamiltalt}{\bar\hamilt}

\newmacro{\lagrangian}{\mathcal{L}}
\newmacro{\lagrangianalt}{\bar\lagrangian}
\newmacro{\Sign}{\textup{Sign}}
\newmacro{\Signp}{\textup{Sign}^+}

\newmacro{\curve}{\gamma}	
\DeclarePairedDelimiterXPP{\curveat}[1]{\curve}{(}{)}{}{#1}	
\DeclarePairedDelimiterXPP{\diffcurveat}[1]{\dot\curve}{(}{)}{}{#1}	

\newmacro{\curves}{\Gamma}
\DeclarePairedDelimiterXPP{\curvesat}[3]{\curves_{#1}}{(}{)}{}{#2;#3}	

\newmacro{\contcurves}{\contfuncs}
\DeclarePairedDelimiterXPP{\contcurvesat}[2]{\contcurves_{#1}}{(}{)}{}{#2}	

\newmacro{\lint}{\cstate}	
\DeclarePairedDelimiterXPP{\lintat}[1]{\lint}{(}{)}{}{#1}	

\newmacro{\qpot}{B}	
\newmacro{\qmat}{B}	
\newmacro{\energy}{E}	

\newmacro{\action}{\mathcal{S}}
\newmacro{\act}{\mathcal{S}}
\DeclarePairedDelimiterXPP{\actof}[2]{\act_{#1}}{[}{]}{}{#2}	

\newmacro{\pth}{\curve}
\newmacro{\pthalt}{\varphi}
\newmacro{\pths}{\curves}
\newmacro{\dpth}{\xi}
\newmacro{\dpthalt}{\zeta}
\newmacro{\daction}{\mathcal{A}}
\newmacro{\quasipot}{V}
\newmacro{\dquasipot}{B}
\newmacro{\symdquasipot}{C}
\newmacro{\dquasipotalt}{\widetilde{\dquasipot}}
\newmacro{\invpot}{W}
\newmacro{\dinvpot}{\energy}
\newmacro{\logmgf}{H}
\newmacro{\contfuncs}{\mathcal{C}}
\newmacro{\map}{F}
\newmacro{\rate}{\rho}
\newmacro{\level}{s}

\newmacro{\eqcl}{\mathcal{K}}
\newmacro{\neqcl}{\nComps}
\newmacro{\primvar}{\alpha}
\newmacro{\bdprimvar}{\primvar_\infty}
\newmacro{\bdvar}{\noisepar_{\infty}^{2}}
\newmacro{\exponent}{s}
\newmacro{\bdpot}{\pot_\infty}
\newmacro{\potgbound}{C}

\newmacro{\ratenhd}{\mathcal{N}}

\newmacro{\diffcurr}{\delta \curr}

\newmacro{\comp}{\mathcal{K}}
\newmacro{\iComp}{i}
\newmacro{\jComp}{j}
\newmacro{\kComp}{k}
\newmacro{\nComps}{K}

\newmacro{\iGround}{0}
\newmacro{\ground}{\comp_{\iGround}}
\newmacro{\groundstates}{\sol[\indices]}
\newmacro{\asymptstables}{AS}
\newmacro{\nontrivialattract}{NTA}

\addauthor[Pan]{PM}{MediumBlue}

\newmacro{\meas}{\mu}	
\newmacro{\altmeas}{\nu}	
\newmacro{\occmeas}{\meas}	
\newmacro{\borel}{\mathcal{B}}	
\newmacro{\size}{\delta}	
\newmacro{\toler}{\eps}	

\newcommand{\Lip}{\textup{Lip}}
\newcommand{\kk}{\mathsf{k}}

\newcommand{\dd}{\mathrm{d}}
\newcommand{\mm}{\mathsf{m}}

\newcommand{\FF}{\mathsf{F}}
\newcommand{\GG}{\mathsf{G}}
\newcommand{\BB}{\mathsf{B}}

\newcommand{\bnu}{\boldsymbol{\nu}}
\newcommand{\X}{\mathcal{X}}
\newcommand{\LL}{\textup{{L}}}

\newcommand{\I}{\iota}

\makeatletter
\newcommand{\leqnomode}{\tagsleft@true\let\veqno\@@leqno}
\newcommand{\reqnomode}{\tagsleft@false\let\veqno\@@eqno}
\makeatother
\begin{document}

 

\title[A prediction-correction Hughes' model]{A coupled prediction-correction Hughes' model\\ for congested crowd motion}

\author
[H.~Ennaji]
{Hamza Ennaji$^{c,\ast}$}
\address{$^{c}$\,%
Corresponding author.}
\address{$^{\ast}$\,%
Univ. Grenoble Alpes, CNRS, Grenoble INP*, LJK, 38000 Grenoble, France.}
\email{hamza.ennaji@univ-grenoble-alpes.fr}
\author
[N.~Igbida]
{Noureddine Igbida$^{\sharp}$}
\address{$^{\sharp}$\,%
Institut de recherche XLIM, UMR-CNRS 7252, Faculté des Sciences et Techniques, Université de Limoges, 87100 Limoges, France.}
\email{noureddine.igbida@unilim.fr}

\author
[G.~Jradi]
{Ghadir Jradi $^{\sharp}$}
\email{ghadir.jradi@gmail.com}

\author[J.M. Urbano]{Jos\'{e} Miguel Urbano$^{\diamond}$}
\address{$^{\diamond}$\,Applied Mathematics and Computational Sciences (AMCS), Computer, Electrical and Mathematical Sciences and Engineering Division (CEMSE), King Abdullah University of Science and Technology (KAUST), Thuwal, 23955-6900, Kingdom of Saudi Arabia and CMUC, Department of Mathematics, University of Coimbra, 3000-143 Coimbra, Portugal.}
\email{miguel.urbano@kaust.edu.sa}








\makeatletter	
\newcommand{\thmtag}[1]{	
  \let\oldthetheorem\thetheorem	
  \renewcommand{\thetheorem}{#1}	
  \g@addto@macro\endtheorem{	
    \addtocounter{theorem}{0}	
    \global\let\thetheorem\oldthetheorem}	
  }
\makeatother

\makeatletter	
\newcommand{\asmtag}[1]{	
  \let\oldtheassumption\theassumption	
  \renewcommand{\theassumption}{#1}	
  \g@addto@macro\endassumption{	
    \addtocounter{assumption}{0}	
    \global\let\theassumption\oldtheassumption}	
  }
\makeatother

\begin{abstract}
In this work, we introduce a new macroscopic model for crowd motion inspired by the celebrated Hughes' model \cite{Hughes2002, Hughes2003}, which couples a nonlinear conservation law for the pedestrian density with an Eikonal equation describing the shortest path to the target. Our approach can be viewed both as a modification of Hughes' original formulation and as a refinement of the prediction–correction framework proposed in the recent work \cite{ennaji2023prediction}. The resulting model incorporates anticipatory behavior and dynamic route adjustment, offering a more realistic representation of crowd dynamics in complex environments. We present the mathematical formulation of the model, discuss its well-posedness properties, and illustrate its qualitative behavior through numerical simulations. Ultimately, we show, at least from a numerical perspective, that this variant provides a promising avenue towards establishing the well-posedness of the classical Hughes' model, which has remained a challenging open problem for a long time.

\end{abstract}


\allowdisplaybreaks	
\acresetall	
\acused{iid}
\acused{LHS}
\acused{RHS}

\maketitle

\setcounter{tocdepth}{1}


\section{Introduction}\label{sec:intro}
Over the past two decades, the analysis, modeling, and simulation of pedestrian behavior have received considerable attention, driven by the need to better understand and manage human crowds in a variety of contexts, including sports events, cultural gatherings, and religious assemblies. One can distinguish three major families of mathematical models used to describe pedestrian dynamics:

\begin{enumerate}

\item {\textit{Microscopic models.}} These are based on differential equations in which each pedestrian is described by its position $x(t)$ and velocity $v(t)$, or on discrete random walk approaches such as cellular automata (see, \textit{e.g.}, \cite{Maury&Venel, maury2007modele, blue2000modeling, blue1998emergent, BURSTEDDE2001507}).

\item {\textit{Macroscopic models.}} In this framework, the number of pedestrians is assumed to be large enough to be represented by a density $\rho(x,t)$. Such models typically take the form of conservation laws (first-order models) or couple a conservation law for the density with additional equations for the velocity field (second-order models). Within this class, one may also include control-theoretic and mean-field approaches (see, \textit{e.g.}, \cite{SLHughs, Hughs2DSIM, Silvasim, carrillo2016improved, AMADORI2012259, Hughs2, Hughs33, Hughes2002, Hughes2003,gibelli2024macroscopic, Herzog&al}).

\item {\textit{Mesoscopic models.}} These rely on a statistical representation of the crowd. Each pedestrian is characterized by a position–velocity pair $(x,v)$, and one studies a kinetic equation satisfied by the probability distribution $f(t,x,v)$ of individuals with state $(x,v)$ at time $t$ (see, \textit{e.g.}, \cite{degond2013vision, henderson1974fluid, dogbe2012modelling, fermo2013fully, bellomo2013microscale, MRS1, MRSV}).
\end{enumerate}

In pedestrian dynamics, modeling the congestion that arises during population evolution is a major challenge. In the literature (see \eg \cite{Maury-EDP-Normandie}), congestion is often handled through two main paradigms: \emph{soft congestion} and \emph{hard congestion}. In the soft congestion setting, the density is continuously penalized, meaning that as local density increases, the crowd's speed decreases, or equivalently, a pressure term disperses the crowd toward unoccupied areas. This is usually done using density-dependent constitutive laws. As for the hard congestion approach, one enforces a strict maximum density constraint $(\rho \leq \rho_{\max})$. This induces a singular pressure that is only activated when the maximal density is reached. Mathematically, this is often modeled using a maximal monotone graph such as the $\Signp$ function.

Various macroscopic models have been proposed to describe collective motion, ranging from congestion-based approaches to route-choice strategies driven by global optimization principles. For an extensive review of mathematical models for crowd dynamics, we refer the reader to \cite{bellomo2011modeling} and the monograph \cite{cristiani2014multiscale}. Among these, prediction-correction models provide an efficient framework to describe hard congestion effects (cf. \cite{MRS1, MRSV, ennaji2023prediction, IgUrbano}). The dynamics is decomposed into a prediction step corresponding to the desired pedestrian motion, followed by a correction step that ensures density constraints are satisfied. While such models naturally reproduce congestion phenomena and guarantee the admissibility condition on the density, numerical simulations reported in \cite{ennaji2023prediction} have identified limitations in evacuation scenarios.  Although pedestrians tend to follow the shortest paths toward the exits, the resulting dynamics fail to sufficiently redistribute the crowd toward underutilized areas. Consequently, pedestrians remain concentrated along preferred trajectories, leaving other regions almost empty and resulting in suboptimal use of available space.

Alternatively, Hughes' model introduces a global route-choice mechanism based on a potential satisfying an Eikonal equation coupled with the crowd density. This framework naturally accounts for the pedestrians' tendency to avoid congested areas, generating self-organized population redistributions. From a modeling perspective, it therefore provides a more realistic description of evacuation strategies. Nevertheless, the strong coupling between the transport equation and the Eikonal relation leads to significant theoretical difficulties (\cf \cite{Amadori2023, Hughs2DSIM}). In particular, the low regularity of the density, the nonlinear dependence of the cost function, and the potential emergence of congestion make the well-posedness analysis delicate at both the continuous and numerical levels.

The objective of this work is to combine the advantages of both approaches by coupling the prediction-correction framework with Hughes' strategy. The idea is to exploit Hughes' mechanism to dynamically reorient pedestrians toward less congested regions, thereby improving the spatial redistribution of the crowd, while using the correction step to enforce density constraints and provide additional regularity and stability. The correction mechanism also supplies congestion information required to define meaningful route-choice strategies in highly saturated regimes. The resulting coupled model therefore aims at combining realistic evacuation behaviors, efficient space utilization, and improved theoretical and numerical properties.

\subsection{Brief presentation of Hughes' model}
Hughes' model \cite{Hughes2002, Hughes2003} is one of the most well-known mathematical frameworks for crowd motion. It describes a population in a domain $\Omega \subset \mathbb{R}^2$ that attempts to reach one or more targets or exits as quickly as possible, in a "rational" manner, meaning that individuals tend to avoid regions of high density. The population density $\rho(x,t)$ at position $x \in \Omega$ and time $t > 0$ satisfies the conservation law
\begin{equation}\label{eq:CL-intro}
\partial_t \rho + \nabla \cdot U[\rho] = 0, \quad t > 0,~x \in \Omega,
\end{equation}
where the flux $U[\rho]$ is defined by \begin{equation}\label{HM0} U[\rho] = -\rho v(\rho)^2 \nabla u. \end{equation} 
The function $v$ represents the pedestrian speed, which typically decreases with the density. A common choice is $v(\rho) = 1 - \rho$. The potential $u$ denotes the so-called cost function or distance potential, and satisfies the weighted Eikonal equation
\begin{equation}\label{eq:Eikonal-intro}
\vert\nabla u(x,t)\vert = \frac{1}{v(\rho(x,t))}, \quad x \in \Omega \setminus \Gamma_D.
\end{equation}
 The system \eqref{eq:CL-intro}--\eqref{eq:Eikonal-intro} is complemented with an initial condition $\rho(x,0) = \rho_0(x)$ in $\Omega$, and the following boundary condition for the potential at the exits:
\begin{equation}\label{eq:BC-Hughes-intro}
    u(x,t) = 0 \quad \text{on } \Gamma_D, 
\end{equation}
where $\partial \Omega = \Gamma_D \cup \Gamma_N$. On the impermeable walls $\Gamma_N$, a zero-flux boundary condition is imposed for the population density, meaning $U[\rho] \cdot \mathbf{n} = 0$, to prevent pedestrians from crossing the walls, where $\mathbf{n}$ denotes the outward unit normal vector. 

\subsection{Challenges in Hughes' model}\label{subsection:diff-hm}

One of the main difficulties in studying Hughes' model lies in the blow-up of the right-hand side of \eqref{eq:Eikonal-intro} as the density approaches the saturated region $[\rho = 1]$. Moreover, the nonlinear dependence of the flux $U[\rho]$ in \eqref{eq:CL-intro} on the density suggests the use of entropic solutions for scalar conservation laws.
A first attempt to address this issue was made in \cite{DiF}, where a double regularization of the system \eqref{eq:CL-intro}–\eqref{eq:Eikonal-intro} was proposed. More precisely, the authors studied the system
\begin{equation}\label{eq:hughes-reg}
	\left\lbrace
	\begin{aligned}
\partial_t \rho + \nabla \cdot U[\rho] &=\epsilon\Delta \rho, \\
-\delta_1\Delta u + \vert\nabla u\vert^2 &= \frac{1}{(v(\rho)+\delta_2)^2},
	\end{aligned}
	\right.
\end{equation}
for some regularization parameters $\epsilon, \delta_1, \delta_2 > 0$, together with suitable boundary conditions.
Although the well-posedness of \eqref{eq:hughes-reg} was established only in one space dimension in \cite{DiF} (see also \cite{AMADORI2012259, Hughs2}), many works have since developed numerical schemes to approximate \eqref{eq:hughes-reg} (see, \textit{e.g.}, \cite{Silvasim, SLHughs}).
Another approach to overcome these difficulties was proposed in \cite{carrillo2016improved}, where the authors considered a locally regularized variant of Hughes' model. Specifically, they introduced a flux of the form $U[\rho] = -\rho v(\rho)\mathcal{P}(\nabla u)$, with $\mathcal{P}$ denoting a suitable smooth approximation operator, and solved the Eikonal equation \eqref{eq:Eikonal-intro} with a cost function of the form $W(x) = \frac{\chi_w(x)}{v(\rho_{\max} - \epsilon)}$, where $\chi_w(x) \in [0,1]$ identifies regions close to walls or obstacles.
In summary, the mathematical and numerical analysis of Hughes' model remains an active and challenging research area. We refer the reader to the recent survey \cite{Amadori2023} for a comprehensive overview of the topic.
\subsection{Prediction–correction approach in \cite{ennaji2023prediction}}
In \cite{ennaji2023prediction}, a prediction–correction model for crowd motion was proposed, where pedestrian dynamics are described by the evolution of crowd density over time and space.
In the prediction phase, the density evolution is governed by the transport equation
\begin{equation}\label{transport1}
\partial_t \tilde \rho + \dive(\tilde \rho V) = 0,
\end{equation}
where $V$ is a velocity field obtained by solving the Eikonal equation
$\vert\nabla \phi\vert = \kk$,
with $\kk$ a positive continuous function.
The predicted density $\tilde \rho$ obtained from \eqref{transport1} may not be admissible, \ie it can exceed the maximal value $1$.
To recover an admissible density, the authors proposed a correction step formulated as the following minimum flow problem:
\begin{equation}\label{minproc}
	\begin{array}{c}
		\inf_{(\Phi,\rho)}\left \{  \int_\Omega F(x,  \Phi  (x) )\: \dd x -  \int_{\Gamma_D} g(x)\: \Phi\cdot\bnu \:  \dd x   \: :\:    0\leq \rho\leq 1,\:   -\dive(\Phi)  =\tilde \rho -\rho   \hbox{ in }   \D'(  \Omega )   \right.  \\   \\  \hspace*{1cm}       \hbox { and }   \Phi\cdot \bnu =0  \hbox{ on }\Gamma_N        \Big\},
	\end{array}
\end{equation} 
where $\Omega$ is a bounded open subset of $\R^N$ with boundary $\partial\Omega = \Gamma_N \cup \Gamma_D$. The cost function $F$ is defined by
\[
F(x,\xi)=\frac{c(x)}{s}\vert \xi \vert ^s, \hbox{ for any } x\in \Omega \hbox{ and } \xi \in \R^N,
\]
and $g$ is an additional boundary cost defined on $\Gamma_D$.
Here, $\bnu$ denotes the unit outward normal to $\partial\Omega$, and $\eta$ is a prescribed function. For the case $s=1$, the Pedestrian Congestion Model (PCM for short) proposed in \cite{ennaji2023prediction} reads
\begin{equation}\tag{PCM}\label{evolgran0}
	\left\{
	\begin{array}{ll}
		\left. 	\begin{array}{l}
			\partial_t \rho -\dive (\Phi) +\dive(\rho \:  V)= 0 \\   \\
			\Phi=m \nabla p,\: m\geq0,  \: \vert \nabla p\vert \leq 1,\:  m(1-\vert \nabla p\vert)=0\\   \\
			\rho \in  \hbox{Sign}^+(p) \\ \\
		\end{array} \right\}  \quad & \hbox{ in } Q:=[0,T[ \times \Omega,\\  \\
		(\rho V - \Phi) \cdot \nu =0 & \hbox{ on }[0,T[ \times\Gamma_N,  \\  \\
		p=0    & \hbox{ on }[0,T[ \times\Gamma_D,
	\end{array}
	\right.
\end{equation}
where $\hbox{Sign}^{+}$ is the maximal monotone graph given by
\[
\operatorname{Sign}^{+}(r)=\left\{\begin{array}{cl}
	1 & \text { for } r>0 \\
	{[0,1]} & \text { for } r=0 \\
	0 & \text { for } r<0.
\end{array} \right.
\]
The gradient constraint imposed on the pressure field admits a natural physical interpretation in terms of limited internal stress propagation within the congested medium. In the classical hard-congestion framework, the pressure acts as a correction mechanism preventing the density from exceeding the maximal admissible value. However, without additional constraints, the resulting correction force $-\nabla p$ may become arbitrarily large, leading to unrealistically strong redistribution effects. Introducing the constraint $\vert \nabla p \vert \leq 1$ bounds the spatial variation of the pressure, thereby limiting the intensity of the corrective velocity generated by congestion. Physically, this assumption reflects the fact that the medium possesses a finite capacity to transmit forces or reorganize itself. In crowd dynamics, it models the limited ability of pedestrians to react and rearrange under compression; in granular materials, it is analogous to a maximal slope or yield criterion governing the onset of flow; in porous or congested transport systems, it represents bounded stress transmission and finite propagation of corrective effects. Consequently, the congested phase no longer behaves as a perfectly rigid, incompressible region, but rather as a medium with constrained internal response, exhibiting behavior closer to that of granular or elastoplastic materials.

Alternatively, the hard congestion constraint may also be replaced by a \emph{soft congestion} mechanism while still preserving the maximal density bound $\rho\leq 1$. Instead of considering the unilateral constraint
 \begin{equation}
	0\leq \rho \leq 1,
	\qquad
	p(1-\rho)=0,
\end{equation}
the pressure may be introduced as a continuous, increasing function of the density 
whose intensity grows as the density approaches the maximal admissible value. Typical examples include singular pressure laws, which can be equivalently formulated through an equation of state of the form $\rho=\beta(p)$.
\begin{figure}[htbp]
	\centering
	\begin{tikzpicture}
		\begin{axis}[
			width=0.44\textwidth,
			height=0.4\textwidth,
			xlabel={$r$},
			ylabel={$\text{Sign}^+(r)$},
			domain=-2:2,
			samples=100,
			ymin=-0.2, ymax=1.2,
			axis lines=middle,
			xtick={-2,-1,0,1,2},
			ytick={0,1},
			enlargelimits=true
			]
			\addplot[blue, thick, domain=-2:0] {0};
			\addplot[blue, thick, domain=0:2] {1};
			\draw[blue, thick] (axis cs:0,0) -- (axis cs:0,1);
			\node at (axis cs:0.2,0.5) [anchor=west, black] {\footnotesize $[0,1]$};
		\end{axis}
	\end{tikzpicture}%
	\hfill%
	\begin{tikzpicture}
		\begin{axis}[
			width=0.44\textwidth,
			height=0.4\textwidth,
			xlabel={$r$},
			ylabel={$\beta^{1}_{\delta}(r)$},
			domain=-2:2,
			samples=200,
			ymin=-0.2, ymax=1.2,
			axis lines=middle,
			xtick={-2,-1,0,1,2},
			ytick={0,1},
			enlargelimits=true,
			legend pos=south east,
			legend style={font=\tiny}
			]
			\addplot[red, thick] {1 / (1 + exp(-x/0.5))};
			\addplot[green, thick] {1 / (1 + exp(-x/0.2))};
			\addplot[black, thick] {1 / (1 + exp(-x/0.1))};
			\addplot[blue, thick] {1 / (1 + exp(-x/0.02))};
		\end{axis}
	\end{tikzpicture}%
	\hfill%
	\begin{tikzpicture}
		\begin{axis}[
			width=0.44\textwidth,
			height=0.4\textwidth,
			xlabel={$r$},
			ylabel={$\beta^{2}_{\delta}(r)$},
			domain=-2:2,
			samples=400, 
			ymin=-0.2, ymax=1.2,
			axis lines=middle,
			xtick={-2,-1,0,1,2},
			ytick={0,1},
			enlargelimits=true,
			legend pos=south east,
			legend style={font=\tiny}
			]
			\addplot[red, thick] {0.5 + atan(x/0.5)/180};
			\addlegendentry{$\delta=0.5$}
			
			\addplot[green, thick]{0.5 + atan(x/0.2)/180};
			\addlegendentry{$\delta=0.2$}
			
			\addplot[black, thick] {0.5 + atan(x/0.1)/180};
			\addlegendentry{$\delta=0.1$}
			
			\addplot[blue, thick] {0.5 + atan(x/0.02)/180};
			\addlegendentry{$\delta=0.02$}
		\end{axis}
	\end{tikzpicture}
	
	\caption{Left: The graph of the multivalued function $\text{Sign}^+$. Center \& Right: Smooth approximations given respectively by the logistic function $\beta^{1}_{\delta}(r) = \frac{1}{1 + \exp\left(-\frac{r}{\delta}\right)}$ and the scaled arctangent $\beta^{2}_{\delta}(r) = \frac{1}{2} + \frac{1}{\pi} \arctan\left(\frac{r}{\delta}\right)$ for different values of the scaling parameter $\delta$.}
	\label{fig:smooth-approx}
\end{figure}
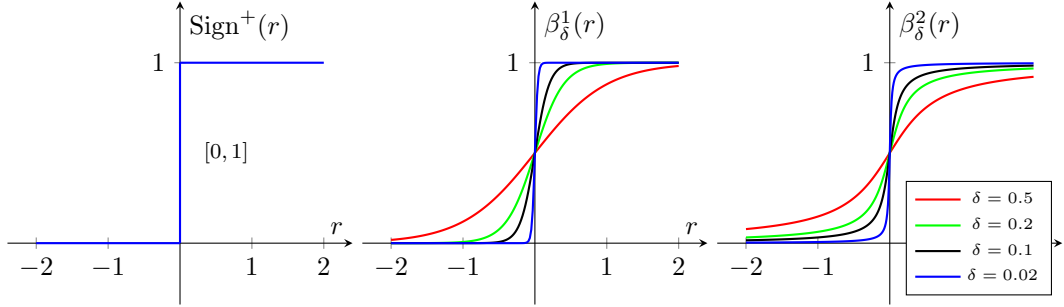
In such formulations, the congestion effects increase progressively as $\rho$ approaches the saturation threshold, generating repulsive corrections before complete saturation. The singular behavior near $\rho=1$ ensures that the density remains below the maximal value, thereby avoiding the abrupt transitions typically induced by hard congestion constraints. Physically, this describes systems capable of anticipating congestion, in which internal stresses build up continuously to produce smoother spatial reorganizations. Consequently, soft congestion models provide a continuous transition between free and congested regimes while preserving the maximal density principle $\rho<1$ and improving regularity properties.

As highlighted previously, both approaches exhibit complementary strengths and weaknesses. On the one hand, while the standard prediction-correction framework \cite{ennaji2023prediction} enforces capacity constraints, it lacks a natural dispersion mechanism. Evacuation simulations show that pedestrians strictly follow shortest paths, leaving parts of the domain unused while others remain dense. On the other hand, Hughes' strategy provides excellent anticipatory routing but suffers from severe analytical and numerical instabilities.

To overcome these respective limitations, we introduce a coupled macroscopic model in the following section. By embedding Hughes' route-choice mechanism within the stable prediction-correction framework, the resulting dynamics naturally reorient pedestrians toward less congested regions while rigorously preserving density constraints. This coupling provides a regularization to achieve a Hughes-type strategy that is both theoretically well-posed and numerically robust, promoting a highly balanced use of available space during evacuations.

 \subsection{A coupled prediction-correction Hughes' model }\label{sec:model}

We improve the model \eqref{evolgran0} by modifying the velocity field $V$. Inspired by Hughes' model, we define a velocity field that depends implicitly on the density $\rho$. The goal is to orient pedestrians towards the nearest exit while avoiding crowded areas.

We propose the velocity field $V=-\nabla \de$, where $\de$ is the solution of the following Eikonal equation:
\begin{equation}\label{eq:V-model}
	(\textup{E}_\rho):	\left\{ \begin{array}{ll}
		\vert \nabla \de\vert =\he(\rho)\quad & \hbox{ in }\Omega,\\   \\
		\de =0 & \hbox{ on }\Gamma_D,
	\end{array}
	\right.
\end{equation}
where $\he:\R \to  (0,+\infty)$ is a cost function satisfying the assumptions below.

We depart from \eqref{evolgran0} model, by handling the congestion using a continuous function $\beta$ instead of the singular maximal monotone graph $\Signp$. 
We shall therefore focus on the following coupled system:
\begin{equation}\leqnomode
    \tag{SC-HM}
    \label{eq:regularized-model}
    \left\{
    \begin{array}{ll}
        \left. \begin{aligned}
            &\partial_t \rho -\dive (\Phi) -\dive(\rho \nabla \de) = f \\
            &\Phi=m \nabla p, \quad \vert \nabla p\vert \leq 1, \quad m(1-\vert \nabla p\vert)=0\\
            &\rho=\beta (p), \quad \vert \nabla \de \vert=\he(\rho)
        \end{aligned} \right\} \quad & \text{in } Q := [0,T[ \times \Omega, \\[4ex]
        \Phi \cdot \nu = 0 & \text{on } [0,T[ \times\Gamma_N, \\[1ex]
        p = 0, \quad \de = 0 & \text{on } [0,T[ \times\Gamma_D.
    \end{array}
    \right.
\end{equation}
Hereafter, the system \eqref{eq:regularized-model} will be referred to as the \emph{Soft Congestion Hughes' Model}. In what follows, we work under the following standing assumptions:
\begin{assumption}\label{assumption:1}
\begin{itemize}
        \item $\Om$ is a bounded domain of $\R^d$ ($d \geq 1$) with boundary $\partial \Om = \Gamma_D \cup \Gamma_N$.
        \item $\beta:\R \to \R$ is a nondecreasing function, bi-Lipschitz on compact sets, and satisfying $\beta(0)=0$.
        \item $\he: \R \to (0,+\infty)$ is a continuous, locally Lipshitz and nondecreasing function. It satisfies, moreover, the compatibility assumption $\im(\beta)\subseteq \dom(\he)$.
  \item $f \in L^1(0,T;\text{Lip}_D^*(\Omega))$, where $\text{Lip}_D(\Omega) = \left\{ z \in \text{Lip}(\Omega) : z=0 \text{ on } \Gamma_D \right\}$ and $\text{Lip}_D^*(\Omega)$ is its topological dual space.
\end{itemize}
\end{assumption}

\begin{remark}
    For the practical case concerning crowd motion, the profile $\beta$ is typically chosen such that $0\leq \beta(r)<1$ and $\lim_{r\to\infty} \beta(r)=1$. This behavior aligns with the compatibility assumption, ensuring that the classical Hughes' cost $\he(\rho) = 1/(1-\rho)$ remains well-defined for any finite pressure.
\end{remark}

The velocity field $V=-\nabla\de$ depends directly on $\rho$, via the law $\rho = \beta(p)$. The function $\he$ measures the walking cost through dense regions. Namely, in congested regions ($\rho\simeq 1$), the walking cost is high ($\he \gg 1$), which forces pedestrians to move toward unoccupied zones. Conversely, in empty regions, pedestrians are allowed to move at their desired speed. In the literature, a widely used cost function takes the form 
\[
\he(\rho) = \frac{1}{v(\rho)},
\]  
where $v$ is a function penalizing high densities. Drawing inspiration from vehicular traffic flow models, a common choice in Hughes' model is $v(\rho) = 1-\rho$. Other examples include $v(\rho) = 1- \exp(-c\frac{1-\rho}{\rho})$ for $c>0$ (see \eg \cite{Silvasim} for further discussion). As discussed previously in \cref{subsection:diff-hm}, in all of these examples, the cost $\he$ blows up as $\rho\to 1$. 

\subsection{Connection with the classical Hughes' model}\label{subsection:connection-cl-HM}
  The core novelty of \eqref{eq:regularized-model} lies in the introduction of the pressure variable $p$ and the deviation flux $\Phi$.  By the complementarity condition 
\[
\Phi=m \nabla p, \quad \vert \nabla p\vert \leq 1, \quad m(1-\vert \nabla p\vert)=0,
\]
one sees that in any region where $\vert\nabla p\vert < 1$, the corrective flux $\Phi$ vanishes. In such regimes,  the \eqref{eq:regularized-model} perfectly aligns with the model \eqref{eq:CL-intro} where 
 \begin{equation}\label{VH1}
  U[\rho]= -\rho\nabla \de,    
 \end{equation}  
and  \begin{equation}\label{VH2}\vert \nabla \de \vert =\frac{1}{v(\rho)},  \end{equation}  
which corresponds to a close variant of Hughes' model.  Whether this equivalence holds universally remains an open question, although numerical simulations presented in \cref{subsection:comp_H} strongly suggest that this is indeed the case.  

Although the present model does not exactly coincide with the classical Hughes' framework (where $U[\rho]$ is given by \eqref{HM0}), it retains its essential strategic ingredient: the ability of agents to anticipate future congestion through an Eikonal potential and to adapt their trajectories accordingly. The departure from the standard Hughes' model concerns mainly the relation between the strategic potential and the resulting transport velocity.
 To make the connection with the classical Hughes' model more explicit, one may consider the following variant instead of \eqref{eq:regularized-model} variant
\begin{equation}\leqnomode
   \label{eq:regularized-model1}
    \left\{
    \begin{array}{ll}
        \left. \begin{aligned}
            &\partial_t \rho - \dive(\Phi)-\dive\!\big(\rho v(\rho)^2\nabla \de\big)=f,\\
            &\Phi=m\nabla p,\qquad |\nabla p|\leq 1,\qquad m(1-|\nabla p|)=0,\\
            &\rho=\beta(p),\qquad |\nabla \de|=\he(\rho),
        \end{aligned}\right\}
        & \text{in } Q,\\[3ex]
        \Phi\cdot\nu=0
        & \text{on } (0,T)\times\Gamma_N,\\[1ex]
        p=0,\qquad \de=0
        & \text{on } (0,T)\times\Gamma_D.
    \end{array}
    \right.
\end{equation}
with
$$v(\rho)=\frac{1}{\he(\rho) }.$$
Indeed, under the eikonal constraint $|\nabla \de|=\he(\rho)$, the effective transport velocity
 $$V=v(\rho)^2 \nabla \de$$
 satisfies
 $$|V|=\frac{1}{\he(\rho)}. $$
Hence, setting $\he(\rho)=1/v(\rho)$, where $v(\rho)$ denotes the desired walking speed, one recovers the classical Hughes' relation $|V|=v(\rho)$, with a velocity that decreases as the local density increases.

From a behavioral viewpoint, this formulation is closely related to the original Hughes' model. In both cases, agents determine their trajectories through a potential solving an Eikonal equation, thereby anticipating future congestion and selecting routes that avoid highly crowded regions. The potential $\de$ can then be interpreted as a perceived travel-cost or travel-time function. Congestion influences the dynamics through two complementary mechanisms: it increases the cost of traversing dense regions and simultaneously reduces the effective walking speed.

By contrast, the model considered in the present work corresponds to the simpler choice $v(\rho)\equiv 1$. In this case, the transport velocity is directly proportional to $\nabla\de$, and the Eikonal constraint implies
\[
|V|=|\nabla\de|=\he(\rho).
\]
Therefore, larger densities generate steeper gradients of the strategic potential. Rather than slowing down the agents, congestion amplifies their tendency to reorganize their trajectories and move away from crowded areas. The potential $\de$ should thus be interpreted as a strategic pressure field governing collective redistribution rather than as a travel-time function. While the behavioral interpretation differs from the classical Hughes' framework, the underlying decision mechanism remains similar: agents anticipate congestion through the Eikonal equation and adapt their motion accordingly.

For the sake of clarity and to avoid unnecessary technical complications, we shall restrict ourselves throughout the paper to the case $v(\rho)\equiv 1$. All the analytical arguments developed below can be extended to more general mobility functions $v$, at the expense of dealing with a genuinely nonlinear parabolic-hyperbolic coupling in the first equation.

In summary, the density-dependent velocity field $V=-\nabla\de$ governed by \eqref{eq:V-model} places \eqref{eq:regularized-model} within the family of Hughes-type models, while the constitutive relation $\rho = \beta(p)$ places it in the scope of soft congestion dynamics.

\subsection{Connection with with hard congestion models}\label{sec:p-model}
The choice of the constitutive relation $\rho = \beta(p)$ in \eqref{eq:regularized-model} is pertinent for several reasons. It ensures the continuous dependence of the pressure on density variations. Moreover, it provides a natural bridge towards hard congestion models. Indeed, introducing a scaling parameter $\delta >0$ and choosing a sequence of functions $\beta_\delta$ that converges to $\Signp$ as $\delta\to 0$ (\cf \cref{fig:smooth-approx}),  the system \eqref{eq:regularized-model} can be viewed as a smooth approximation of a hard congestion problem. Hence, \eqref{eq:regularized-model} acts as a natural regularization of the following Hughes-like model 
\begin{equation*}\leqnomode\tag{HC-HM}
\label{evolution223}
	\left\{
	\begin{array}{ll}
		\left. 	\begin{array}{l}
			\partial_t \rho -\dive (\Phi) -\dive(\rho \:  \nabla \de)= f  \\   \\
			\Phi=m \nabla p,  \: \vert \nabla p\vert \leq 1,\:  m(1-\vert \nabla p\vert)=0\\   \\
			\rho \in  \Signp(p),\:  \vert \nabla \de \vert=\he(\rho) \\ \\
		\end{array} \right\}  \quad & \hbox{ in } Q:=[0,T[ \times \Omega,\\  \\
		\Phi \cdot \nu =0 & \hbox{ on }[0,T[ \times\Gamma_N,  \\  \\
		p=0, \:	\de =0    & \hbox{ on }[0,T[ \times\Gamma_D,
	\end{array}
	\right.
\end{equation*}
where the acronym \eqref{evolution223} stands for \emph{Hard Congestion Hughes' Model}. The model \eqref{evolution223} can be seen as a direct generalization of \eqref{evolgran0} where the hard congestion constraint is now coupled to a density-dependent velocity field $V=-\nabla\de$.

From a mathematical perspective, the analysis of \eqref{evolution223} is challenging, mainly due to the singular relation $\rho \in \Signp(p)$ which induces a strong elliptic-parabolic degeneracy. This makes the use of compactness results such as the Aubin-Lions-Simon (\cref{lem:ALS}) fail as they require the graph to be bi-Lipschitz, a property which is not satisfied by $\Signp$. While other nonlinear compactness tools such as the Alt-Luckhaus approach \cite{Alt&Luckhaus} or Kruzhkov's and Ma\^itre's lemmas (see \eg \cite{EM,BA,SK}) have been developed to handle specific regularity requirements on the graph, the \eqref{evolution223} model presents another challenge because of the coupling through the Eikonal equation $\vert \nabla \de \vert = \he(\rho)$ and the advection term $\dive(\rho \nabla \de)$. Thus, passing to the limit in these nonlinear terms would require strong compactness on the variables $\rho$ and $p$. However, the degeneracy of the $\text{Sign}^+$ graph inherently prevents the derivation of such strong estimates. 

To bypass these challenges, we return to the \eqref{eq:regularized-model}. As we shall see, the bi-Lipschitz property of the regularizer $\beta$ provides the necessary strong compactness on the density, which is fundamentally lacking in hard congestion models. This allows us to rigorously pass to the limit in the highly nonlinear cross-coupled terms, namely the advection term $\dive(\rho \nabla \de)$ and the Eikonal equation.
\subsection{Contributions and organization of the paper}
 The main contributions of the paper are: i) the theoretical study of the \eqref{eq:regularized-model}, and ii) the development of a robust numerical framework handling both soft and hard congestion dynamics.

From a theoretical point of view, we establish the rigorous existence of a variational solution to the \eqref{eq:regularized-model} system. We first address the associated stationary problem \eqref{eq:stat-regularized-model}, for which the existence of a variational solution is proven in \cref{prop:existence-stat}. Then, we show in \cref{thm:th2} the existence of a variational solution to the regularized evolution model using a semi-discrete approximation.

From a numerical perspective, we adopt a prediction-correction approach for the proposed model. It is worth mentioning that, while the singular graph $\Signp$ in \eqref{evolution223} induces several theoretical difficulties, it turns out to be highly tractable from a numerical point of view. Indeed, thanks to our primal-dual optimization framework, the multi-valued nature of the $\Signp$ graph is efficiently resolved via straightforward projections. This being said, even though the theoretical results mainly concern the regularized model \eqref{eq:regularized-model}, we still present in \cref{section:4} the splitting approach for both \eqref{evolution223} and \eqref{eq:regularized-model}. Since the approximation of the prediction (or transport) step is standard and is given explicitly by \eqref{t2}, we focus on the correction step by providing the primal-dual iterates in \cref{alg:pd}. Finally, we present several examples in \cref{sm_Hughes} to illustrate our approach.

The rest of the paper is organized as follows. In \cref{section:existence}, we provide the main theoretical results of the paper, namely the well-posedness of the proposed regularized model \eqref{eq:regularized-model}. \cref{section:4} addresses the discretization of the model in both its prediction and correction components. In \cref{sm_Hughes}, we present various numerical results and evacuation scenarios to illustrate the robustness of our approach. Finally, the reader may find in the appendix several technical results used throughout the manuscript.


\section{Existence of a variational solution}\label{section:existence}
\subsection{The stationary problem}
Before presenting the appropriate notion of solution for \eqref{eq:regularized-model}, let us recall a fundamental result providing a variational characterization of the distance function $\de$.

\begin{theorem}[\cite{EnnajiaugmentedLag}]\label{EnnajiaugmentedLag}
    Let $\kk \in C(\overline{\Omega})$ be such that $\kk \ge 0$. Then, the Eikonal equation
    \begin{equation}\label{eq:eikonal-1}
    \left\lbrace\begin{array}{ll}
        \vert \nabla \de \vert = \kk & \hbox{ in }\Om,\\ 
        \de = 0 & \hbox{ on } \Gamma_D,
    \end{array}
    \right.
    \end{equation}
    admits a unique solution $\de \in W_D^{1,\infty}(\Om)$, which can be characterized as 
    \begin{equation}
        \de = \argmax_{z \in W^{1,\infty}(\Om)}\left\{ \int_\Omega z \: \dd x : \vert \nabla z\vert \leq \kk \mbox{ and } z=0 \mbox{ on } \Gamma_D \right\}.
    \end{equation}
\end{theorem}	
The mathematical study of \eqref{eq:regularized-model} relies first on the analysis of the corresponding stationary problem:
\begin{equation}\leqnomode
    \tag{$S_{\textup{stat}}$}
    \label{eq:stat-regularized-model}
    \left\{
    \begin{array}{ll}
    \left. \begin{aligned}
        &\rho -\dive (\Phi) -\dive(\rho \nabla \de) = f  \\
        &\Phi = m \nabla p, \quad m\geq0, \quad \vert \nabla p\vert \leq 1, \quad m(1-\vert \nabla p\vert)=0 \\
        &\rho = \beta(p), \quad \vert \nabla \de \vert=\kk 
        \end{aligned} \right\} \quad & \text{in } \Omega, \\[4ex]
        \Phi \cdot \nu = 0 & \text{on } \Gamma_N, \\[1ex]
        p = 0, \quad \de = 0, & \text{on } \Gamma_D.
    \end{array}
    \right.
\end{equation}
where $\kk\in C(\overline{\Omega})$ is a nonnegative function. The justification behind taking $\de$ as a solution to \eqref{eq:eikonal-1} rather than \eqref{eq:V-model} lies in the study of the semi-discrete model \eqref{eq:discrete-regularized-model} we present below. Indeed, in our prediction-correction framework, the potential $\de^{i+1}$ at step $i+1$ is obtained by solving \eqref{eq:eikonal-1} with $\kk = \he(\rho^{i})$, using the density $\rho^{i}$ computed at the previous step $i$.

We summarize the main steps of our theoretical analysis as follows. First, we prove the existence of a variational solution to the stationary problem \eqref{eq:stat-regularized-model} for a given, fixed potential. Then, we rely on a semi-discrete Euler scheme in time to construct a sequence of approximate solutions. By deriving suitable a priori estimates and employing compactness arguments, we pass to the limit as the time step goes to zero, ultimately obtaining a variational solution to the full regularized time-dependent model \eqref{eq:regularized-model}.


\subsection{Study of the stationary problem \eqref{eq:stat-regularized-model}}

As is well known, the sub-gradient constraint operator gives rise to a divergence operator involving measure-valued fluxes (cf. \cite{Ig1, Ig2, DePascale&Jimenez}). Combined with the nonlinear constitutive relation
$\rho=\beta(p)$, this leads to a technically challenging framework in which the effective flux of the equation may exhibit singular behavior. Such difficulties would considerably increase the mathematical complexity of the presentation and potentially divert the reader from the main objective of the present work.

Consequently, throughout this paper, we adopt the notion of variational solutions. Although an equivalence with weak solutions involving the effective PDE flux is expected, this issue remains an open problem; we refer the reader to \cite{IgUrbano} for further discussions. In what follows, we first introduce the definition in the stationary setting, while the corresponding notion for the evolution problem will be presented in the next section. This proposed framework naturally builds upon the approaches developed in \cite{MCIg, IgUrbano}.
\begin{definition}[Variational solution]\label{def:var-sol}
     Given $f \in \X:=L^{1}(\Om)$ and a fixed potential $\de$, a variational solution to \eqref{eq:stat-regularized-model} is a pair $(\rho,p) \in \X \times \K$ such that $\rho = \beta(p)$ almost everywhere, and which satisfies the following variational inequality for any test function $\xi \in \K$:
    \begin{equation}
        \int_{\Om} \rho (p-\xi) \dd x + \int_{\Omega} \rho\nabla\de\cdot\nabla(p-\xi)\dd x \leq \int_{\Om} f(p-\xi)\dd x.
    \end{equation}
\end{definition}
In the following result, we establish the existence of a variational solution to \eqref{eq:stat-regularized-model} and prove its continuous dependence on the source term $f$.
\begin{proposition}[Existence and $L^1$-contraction]\label{prop:existence-stat}
     Given $f \in \X$ and a potential $\de$, there exists a variational solution $(\rho,p)$ to \eqref{eq:stat-regularized-model}. Moreover, if $(\rho_1, p_1)$ and $(\rho_2, p_2)$ are variational solutions corresponding to source terms $f_1, f_2 \in \X$, respectively, then the comparison principle 
    \begin{equation}\label{eq:est-1}
        \Vert (\rho_1 - \rho_2)^{+} \Vert_{L^1} \leq \Vert (f_1 - f_2)^{+} \Vert_{L^1}
    \end{equation}
    holds, and consequently, we have the $L^1$-contraction estimate
    \begin{equation}\label{eq:est-2}
        \Vert \rho_1 - \rho_2 \Vert_{L^1} \leq \Vert f_1 - f_2 \Vert_{L^1}.
    \end{equation}
\end{proposition}
\begin{proof}
    The existence of a variational solution is a direct consequence of \cref{thm:app-thm1}. Indeed, since by \cref{lem:app-lem-1} the operator $\T:\W\to\V^{*}$ defined by \eqref{eq:op-T} is pseudo-monotone with $\V = W^{1,s}_{D}(\Om)$ and $\W=\K$, we deduce the existence of $p\in\K$ such that
    \[
    \langle\T(p),p-\xi\rangle \leq \langle f,p-\xi\rangle, \quad \text{for all } \xi\in \K.
    \]
    This yields
    \begin{equation}
        \int_{\Om}\rho (p-\xi) \dd x + \int_{\Omega} \rho\nabla\de\cdot\nabla(p-\xi)\dd x \leq \int_{\Om} f(p-\xi)\dd x,
    \end{equation}
    with $\rho = \beta(p)$, \ie $(\rho,p)$ is a variational solution to \eqref{eq:stat-regularized-model} in the sense of \cref{def:var-sol}.
	
    Now let us prove the comparison principle \eqref{eq:est-1}. Let $(\rho_1, p_1)$ and $(\rho_2, p_2)$ be two variational solutions, corresponding to the source terms $f_1$ and $f_2$, respectively. This means that for $i=1,2$, we have
    \begin{equation}
        \label{eq:ineq-1}
        \int_{\Om}\rho_i (p_i-\xi_i) \dd x + \int_{\Omega} \rho_i\nabla\de\cdot\nabla(p_i-\xi_i)\dd x \leq \int_{\Om} f_i(p_i-\xi_i)\dd x,
    \end{equation}
    for any test function $\xi_i\in\K$. Next, let us define the standard truncation operator at level $k>0$,
    \begin{equation}
        T_{k}^{+}(r) = 
        \begin{cases}
            0 & \text{if } r < 0,\\
            r & \text{if } 0 \leq r \leq k,\\
            k & \text{if } r > k,
        \end{cases}
    \end{equation}
    and consider the test functions
    \[
        \xi_{1} = p_1 - T_{k}^{+}(p_1 - p_2) \quad \text{and} \quad \xi_{2} = p_2 + T_{k}^{+}(p_1-p_2).
    \]
    Using $\xi_1$ and $\xi_2$ as test functions in \eqref{eq:ineq-1} for $i=1$ and $i=2$, respectively, we obtain  
       \begin{equation}
        \label{eq:ineq-2}
        \int_{\Om}\rho_1 T_{k}^{+}(p_1-p_2) \dd x + \int_{\Omega} \rho_1\nabla\de\cdot\nabla T_{k}^{+}(p_1-p_2)\dd x \leq \int_{\Om} f_1 T_{k}^{+}(p_1-p_2)\dd x,
    \end{equation}
    and
    \begin{equation}
        \label{eq:ineq-3}
        -\int_{\Om}\rho_2 T_{k}^{+}(p_1-p_2) \dd x - \int_{\Omega} \rho_2\nabla\de\cdot\nabla T_{k}^{+}(p_1-p_2)\dd x \leq -\int_{\Om} f_2 T_{k}^{+}(p_1-p_2)\dd x.
    \end{equation}
    Summing inequalities \eqref{eq:ineq-2} and \eqref{eq:ineq-3} and dividing by $k>0$, we get	
    \begin{equation}
        \label{eq:ineq-4}
        \int_{\Om}(\rho_1-\rho_2) \frac{T_{k}^{+}(p_1-p_2)}{k} \dd x + \int_{\Omega} \frac{(\rho_1 - \rho_2)}{k}\nabla\de\cdot\nabla T_{k}^{+}(p_1-p_2)\dd x \leq \int_{\Om} (f_1-f_2)\frac{ T_{k}^{+}(p_1-p_2)}{k}\dd x.
    \end{equation}
    Since $\beta$ is nondecreasing, $\rho_1 \geq \rho_2$ whenever $p_1 \geq p_2$. Thus, by Lebesgue's dominated convergence theorem, as $k \to 0$, we have
    \[
        \int_{\Om}(\rho_1-\rho_2) \frac{T_{k}^{+}(p_1-p_2)}{k} \dd x \longrightarrow \int_{\Om}(\rho_1-\rho_2)^{+}\dd x,
    \]
    and 
    \[
        \int_{\Om} (f_1-f_2)\frac{ T_{k}^{+}(p_1-p_2)}{k}\dd x \longrightarrow   \int_{\Om} (f_1-f_2)^{+}\dd x.
    \]
To treat the remaining gradient term, we first note that the test functions $p_1, p_2 \in \K$ satisfy $\vert \nabla p_i \vert \leq 1$ almost everywhere and vanish on $\Gamma_D$. Since $\Omega$ is bounded, then $p_1,p_2\in L^\infty(\Omega)$. Therefore, there exists $M > 0$ such that $p_1(x), p_2(x) \in [-M, M]$, \pp $x \in \Omega$.

From our assumptions, $\beta$ is Lipschitz continuous on compact sets. Thus, there exists a local Lipschitz constant $L_M > 0$ such that
    \[
        \vert \rho_1 - \rho_2\vert = \vert \beta(p_1) - \beta(p_2)\vert \leq L_M \vert p_1 - p_2\vert, \quad \text{a.e. in } \Omega.
    \]
Consequently, we get, using the fact that $(T_k^+)' = 1$ on the set $\{0 < p_1 - p_2 < k\}$,
    \begin{equation}
        \begin{aligned}
            \left\vert \int_{\Omega} \frac{(\rho_1 - \rho_2)}{k}\nabla\de\cdot\nabla T_{k}^{+}(p_1-p_2)\dd x \right\vert &\leq L_M \int_{\Omega} \left\vert \frac{(p_1 - p_2)}{k}\nabla\de\cdot\nabla T_{k}^{+}(p_1-p_2)\right\vert \dd x\\
            &\leq L_M \int_{\Omega\cap\{0 < p_1-p_2 < k\}} \left\vert \frac{(p_1 - p_2)}{k} \nabla\de\cdot\nabla (p_1-p_2)\right\vert \dd x\\
            &\leq L_M \int_{\Omega\cap\{0 < p_1-p_2 < k\}} \left\vert\nabla\de\cdot\nabla (p_1-p_2)\right\vert \dd x.
        \end{aligned}
    \end{equation}
   Since $\nabla \de$ and $\nabla(p_1-p_2)$ are uniformly bounded in $L^\infty(\Omega)$, Lebesgue's dominated convergence theorem ensures that
    \[
        \left\vert \int_{\Omega} \frac{(\rho_1 - \rho_2)}{k}\nabla\de\cdot\nabla T_{k}^{+}(p_1-p_2)\dd x\right\vert \longrightarrow 0 \quad \text{as } k\to 0.
    \]
    Passing to the limit $k \to 0$ in \eqref{eq:ineq-4} finally yields
    \[
        \int_{\Om}(\rho_1-\rho_2)^{+}\dd x \leq \int_{\Om} (f_1-f_2)^{+}\dd x,
    \]
as desired.
\end{proof}
\subsection{Study of \eqref{eq:regularized-model}}

We now turn to the full time-dependent problem \eqref{eq:regularized-model}. Let us first define the notion of a variational solution for this evolution equation.

\begin{definition}[Variational solution of the evolution problem]\label{def:var-sol-regularized-model}
    Given an initial datum $\rho_0\in L^{1}(\Om)$ and a source terme $f\in L^{1}(0,T;\Lip_{D}^{*}(\Om))$, a variational solution to \eqref{eq:regularized-model} is a triplet $(\rho,p,\de)$ such that
    \begin{itemize}
    	\item $\rho \in L^1(0,T; C(\Omega)) \cap L^\infty(Q)$,   $p(t, \cdot) \in \K$ for all $t \in [0,T]$, satisfying $\rho = \beta(p)$ almost everywhere in $\Omega$,
    	\item $\de\in L^{s}(0,T;W^{1,s}(\Om))$ is such that for \pp $t\in (0,T), \de(t,.)$ solves \eqref{eq:V-model}.
    	\end{itemize}
    	Furthermore, for any test function $\xi \in \K$, any $k > 0$, and any non-negative test function $\sigma \in C^{\infty}_{c}([0,T))$, the following inequality holds:
    \begin{equation}
        \begin{aligned}
            &-\int_{0}^{T} \int_{\Om} \partial_{t}\sigma(t) \left( \int_{0}^{\rho(t,x)} T_{k}(\beta(s)-\xi)\dd s \right) \dd x \dd t - \sigma(0)\int_{\Om}\int_{0}^{\rho_0(x)} T_{k}(\beta(s)-\xi)\dd s\dd x\\
            &+ \int_{0}^{T} \int_{\Omega} \rho \nabla\de \cdot \nabla T_k(p-\xi) \sigma(t) \dd x \dd t 
            \leq \int_{0}^{T} \int_{\Om} f T_{k}(p-\xi) \sigma(t) \dd x \dd t,
        \end{aligned}
    \end{equation}
    where $T_k(z) = \min(k, \max(z,-k))$ is the standard truncation function at level $k$.
\end{definition}

Using an implicit Euler scheme in time, we consider the following semi-discrete version of \eqref{eq:regularized-model}:
\begin{equation}\leqnomode
    \tag{$S_{\beta,\epsilon}$}
    \label{eq:discrete-regularized-model}
    \left\{
    \begin{array}{ll}
        \left. \begin{aligned}
            &\frac{\rho^{i+1}-\rho^{i}}{\epsilon} -\dive (\Phi^{i+1}) -\dive(\rho^{i+1} \nabla \de^{i+1}) = f^{i+1} \\
            &\Phi^{i+1}=m \nabla p^{i+1}, \quad m\geq0, \quad \vert \nabla p^{i+1}\vert \leq 1, \quad m(1-\vert \nabla p^{i+1}\vert)=0\\
            &\rho^{i+1}=\beta(p^{i+1}), \quad \vert \nabla \de^{i+1}\vert =\he(\rho^{i})
        \end{aligned} \right\} \quad & \text{in } \Omega, \\[4ex]
        \Phi^{i+1} \cdot \nu = 0 & \text{on } \Gamma_N, \\[1ex]
        p^{i+1} = 0, \quad \de^{i+1} = 0 & \text{on } \Gamma_D,
    \end{array}
    \right.
\end{equation}
where the sequence $\{f^i\}_{i=1}^n \subset \Lip_D^*(\Om)$ is chosen such that 	
\begin{equation}
    \sum_{i=1}^{n}\int_{t_{i-1}}^{t_{i}} \Vert f(t)-f^{i} \Vert_{\Lip_D^*} \: \dd t \le \epsilon.
\end{equation}	   
Owing to \cref{prop:existence-stat}, we deduce the existence of a sequence of solutions 
\[
(\rho^{i+1}, p^{i+1}, \de^{i+1}) \in L^{\infty}(\Om) \times \K \times W^{1,\infty}(\Om),
\]
to the semi-discrete problem \eqref{eq:discrete-regularized-model}. At each time step, these updated variables satisfy the variational inequality
\begin{equation}\label{eq:evol_i}
    \int_{\Omega} \frac{\rho^{i+1}-\rho^{i}}{\epsilon} T_k (p^{i+1}-\xi) \dd x + \int_{\Omega} \rho^{i+1} \nabla \de^{i+1} \cdot \nabla T_k(p^{i+1}-\xi) \dd x \le \int_{\Omega} f^{i+1} T_k(p^{i+1}-\xi) \dd x,
\end{equation}
for all test functions $\xi \in \K$ and any $k>0$.

 To establish the existence of a variational solution to  \eqref{eq:regularized-model}, we construct a sequence of approximate solutions using the semi-discrete scheme \eqref{eq:discrete-regularized-model}. Given a time step $\epsilon>$, we define the piecewise constant interpolants $(\rho^\epsilon,p^\epsilon,\de^\epsilon)$ and the piecewise linear interpolant function $\tilde{\rho}^\epsilon$ as follows.
\begin{definition}\label{def:interpolants}
Given $\epsilon >0$, we define, for $t \in (t_i, t_{i+1}]$, and $i = 0, \dots, n-1,$
\begin{equation}\label{eq:approx-sol}
	    \rho^\epsilon(t) = \rho^{i+1}, \quad p^\epsilon(t) = p^{i+1}, \quad \text{and} \quad \de^\epsilon(t) = \de^{i+1},
\end{equation}
with the initial data defined at $t=0$ as $\rho^\epsilon(0) = \rho^0$, $p^\epsilon(0) = p^0$, and $\de^\epsilon(0) = \de^0$. 
Moreover, we define the piecewise linear interpolant function $\tilde{\rho}^\epsilon$ as
\begin{equation}\label{eq:tilde-rho}
    \tilde{\rho}^\epsilon(t) = \frac{t - t_i}{\epsilon}\rho^{i+1} + \frac{t_{i+1} - t}{\epsilon}\rho^{i}, \quad \text{for } t \in (t_i, t_{i+1}].
\end{equation}
\end{definition}

By construction, these approximate solutions satisfy the following fundamental discrete inequality that will play an important role in the sequel.
\begin{lemma}\label{lem:d-ineq}
	For all $\epsilon>0$, any test function $\xi\in\K$ and any $k>0$, the approximate solutions defined in \cref{def:interpolants} satisfy the following inequality for \pp $t\in (0,T)$:
	\begin{equation}\label{eq:final-discrete-ineq}
    \int_{\Omega} \partial_t \tilde{\rho}^{\epsilon} T_k(\beta^{-1}(\tilde{\rho}^{\epsilon})-\xi) \dd x + \int_{\Omega} \rho^{\epsilon} \nabla \de^{\epsilon} \cdot \nabla T_k(p^{\epsilon}-\xi) \dd x \le \int_{\Omega} f^{\epsilon} T_k(p^{\epsilon}-\xi) \dd x.
\end{equation}
\end{lemma}
\begin{proof}
	By definition, $\partial_t \tilde{\rho}^\epsilon = \frac{\rho^{i+1}-\rho^{i}}{\epsilon}$. Thanks to \eqref{eq:evol_i}, we have, for all $t \in (t_i, t_{i+1}]$,
\begin{equation}\label{eq:evol_e}
    \int_{\Omega} \partial_t \tilde{\rho}^{\epsilon} T_k (p^\epsilon-\xi) \dd x + \int_{\Omega} \rho^\epsilon \nabla \de^\epsilon \cdot \nabla T_k(p^\epsilon-\xi) \dd x \le \int_{\Omega} f^\epsilon T_k(p^\epsilon-\xi) \dd x, \quad \text{for all } \xi \in \K.
\end{equation}

We now seek to properly bound from below the first term on the left-hand side of \eqref{eq:evol_e}. Since $p^\epsilon = \beta^{-1}(\rho^\epsilon)$, the time derivative term can be written as $\partial_t \tilde{\rho}^{\epsilon} T_k (\beta^{-1}(\rho^\epsilon)-\xi)$. Notice that for any $t \in (t_i, t_{i+1}]$, we have
\[
    \partial_t\tilde{\rho}^{\epsilon}(\rho^\epsilon-\tilde{\rho}^\epsilon) = \frac{(\rho^{i+1}-\rho^{i})^2}{\epsilon^2}(t_{i+1}-t) \geq 0.
\]
Since the function $\beta^{-1}$ is nondecreasing, the terms $(\rho^\epsilon-\tilde{\rho}^\epsilon)$ and $(\beta^{-1}(\rho^\epsilon)-\beta^{-1}(\tilde{\rho}^\epsilon))$ have the same sign. It follows that 
\[
    \int_{\Omega} \partial_t\tilde{\rho}^{\epsilon} \left( \beta^{-1}(\rho^\epsilon) - \beta^{-1}(\tilde{\rho}^\epsilon) \right) \dd x \geq 0.
\]
Since $T_k$ is a nondecreasing function, the composition $r \mapsto T_k(\beta^{-1}(r) - \xi)$ is also nondecreasing. Therefore, the difference $\left( T_k(\beta^{-1}(\rho^\epsilon) - \xi) - T_k(\beta^{-1}(\tilde{\rho}^\epsilon) - \xi)\right)$ shares the same sign as $(\beta^{-1}(\rho^\epsilon)-\beta^{-1}(\tilde{\rho}^\epsilon))$. This yields:
\[
    \int_{\Omega} \partial_t \tilde{\rho}^\epsilon \left( T_k(\beta^{-1}(\rho^\epsilon) - \xi) - T_k(\beta^{-1}(\tilde{\rho}^\epsilon) - \xi)\right)\dd x \geq 0.
\]
Combining this with \eqref{eq:evol_e}, we obtain the key discrete inequality \eqref{eq:final-discrete-ineq}.
\end{proof}
In what follows, we prove that the sequences $(\rho^\epsilon)_{\epsilon}$, $(p^\epsilon)_{\epsilon}$, and $(\de^\epsilon)_{\epsilon}$ converge respectively to $\rho$, $p$, and $\de$, where the triplet $(\rho,p,\de)$ is a variational solution to problem \eqref{eq:regularized-model} in the sense of \cref{def:var-sol-regularized-model}.
\begin{theorem}\label{thm:th2}
     Given an initial datum $\rho_0\in L^{1}(\Om)$ and any source terme $f\in L^{1}(0,T;\Lip_{D}^{*}(\Om))$, the regularized problem \eqref{eq:regularized-model} admits a variational solution $(\rho, p, \de)$ in the sense of \cref{def:var-sol-regularized-model}. More precisely, $\rho \in L^1(0,T; C(\Omega)) \cap L^\infty(Q)$, $p \in L^\infty(0,T; \K)$, with $\rho = \beta(p)$ almost everywhere, and $\de \in L^s(0,T; W^{1,s}(\Omega))$. 
    Here, $\de(t, \cdot)$ is the maximal subsolution to the Eikonal equation
    \begin{equation}\label{eq:eikonal}
        \begin{cases}
            \vert \nabla \de \vert = \he(\rho) & \text{in } \Om,\\
            \de = 0 & \text{on } \Gamma_D,
        \end{cases}
    \end{equation}
    as characterized in \cref{EnnajiaugmentedLag}. Furthermore, for any test function $\xi \in \K$, any $k > 0$, and any non-negative test function $\sigma \in C^{\infty}_{c}([0,T))$, the following inequality holds:
    \begin{equation}\label{eq:evol1}
        \begin{aligned}
            &-\int_{0}^{T} \int_{\Omega} \partial_t \sigma(t) \left( \int_{0}^{\rho(t,x)} T_k(\beta^{-1}(s)-\xi) \dd s \right) \dd x \dd t - \sigma(0) \int_{\Omega}\int_{0}^{\rho_0(x)} T_k(\beta^{-1}(s)-\xi) \dd s \dd x\\
            &+ \int_{0}^{T} \int_{\Omega} \rho \nabla \de \cdot \nabla T_k(p-\xi) \sigma(t) \dd x \dd t \le \int_{0}^{T} \int_{\Om} f T_k(p-\xi) \sigma(t) \dd x \dd t.
        \end{aligned}
    \end{equation}
\end{theorem}

The proof of \cref{thm:th2} relies on the following lemmas.
\begin{lemma}\label{lem:rho}
    The sequences $(p^\epsilon)_{\epsilon}$ and $(\tilde{\rho}^{\epsilon})_{\epsilon}$ are bounded in $L^{\infty}(0,T; \Lip_{D}(\Omega))$, and the sequence $(\partial_t \tilde{\rho}^{\epsilon})_{\epsilon}$ is bounded in $L^1(0,T; \Lip_{D}^*(\Om))$.
\end{lemma}

\begin{proof}
First, we note that $p^{\epsilon}(t) \in \K$, for all $t \in [0,T]$. Since the space $\K$ is a bounded subset of $\Lip_D(\Omega)$ and $L^\infty(\Omega)$, the sequence $(p^{\epsilon})_{\epsilon}$ is bounded in $L^{\infty}(0,T; \Lip_D(\Omega))$. Moreover, since $\rho^\epsilon = \beta(p^{\epsilon})$ \pp and $\beta$ is Lipschitz continuous, we deduce that $(\rho^\epsilon)_{\epsilon}$ is bounded in $L^{\infty}(0,T; \Lip_{D}(\Omega))$. By convex combination, the piecewise linear interpolant $(\tilde{\rho}^{\epsilon})_{\epsilon}$ is itself bounded in $L^\infty(0,T; \Lip_{D}(\Omega))$.

     To bound the time derivative $\partial_t \tilde{\rho}^{\epsilon}$, we observe that since $\K$ is symmetric, $-\xi \in \K$ for any $\xi \in \K$. Moreover, going back to the proof of \cref{lem:d-ineq}, we have  for all $t \in (t_i, t_{i+1}]$,
\begin{equation}\label{eq:v_ineq_lemrho}
    \int_{\Omega} \partial_t \tilde{\rho}^{\epsilon} T_k (p^\epsilon-\xi) \dd x + \int_{\Omega} \rho^\epsilon \nabla \de^\epsilon \cdot \nabla T_k(p^\epsilon-\xi) \dd x \le \int_{\Omega} f^\epsilon T_k(p^\epsilon-\xi) \dd x, \quad \text{for all } \xi \in \K.
\end{equation}
    Since $p^\epsilon$ and $\xi$ belong to $\K \subset L^\infty(\Omega)$, taking $k > \Vert p^\epsilon - \xi \Vert_\infty$ ensures that $T_k(p^\epsilon - \xi) = p^\epsilon - \xi$. Thus, the inequality \eqref{eq:v_ineq_lemrho} simplifies to
\begin{equation}\label{eq:v_ineq_k_large}
    \int_{\Omega} \partial_t \tilde{\rho}^{\epsilon} (p^\epsilon-\xi) \dd x + \int_{\Omega} \rho^{\epsilon}\nabla \de^\epsilon \cdot \nabla(p^{\epsilon}-\xi) \dd x \le \int_{\Omega} f^{\epsilon} (p^{\epsilon}-\xi) \dd x,
\end{equation}
     for any $\xi \in \K$. Substituting $-\xi$ for $\xi$ in the \eqref{eq:v_ineq_k_large}, we obtain
    \begin{equation}
        \int_{\Omega} \partial_t \tilde{\rho}^{\epsilon} (\beta^{-1}(\tilde{\rho}^{\epsilon})+\xi) \dd x + \int_{\Omega} \rho^{\epsilon}\nabla \de^\epsilon \cdot \nabla(p^{\epsilon}+\xi) \dd x \le \int_{\Omega} f^{\epsilon} (p^{\epsilon}+\xi) \dd x.
    \end{equation}	
    Expanding the first term yields
    \begin{equation}\label{eq:bound_xi}
        \int_{\Omega} \partial_t \tilde{\rho}^{\epsilon} \xi \dd x \leq - \int_{\Omega} \rho^{\epsilon}\nabla \de^\epsilon \cdot \nabla(p^{\epsilon}+\xi) \dd x + \int_{\Omega} f^{\epsilon} (p^{\epsilon}+\xi) \dd x - \int_{\Omega} \partial_t \tilde{\rho}^{\epsilon} \beta^{-1}(\tilde{\rho}^{\epsilon}) \dd x, 
    \end{equation}
    for all  $t \in [t_i, t_{i+1})$.
    Since $\partial_t \tilde{\rho}^{\epsilon} \beta^{-1}(\tilde{\rho}^{\epsilon}) = \frac{\dd}{\dd t} \int_{0}^{\tilde{\rho}^{\epsilon}(t)} \beta^{-1}(s) \dd s$, we get, by substituting this into \eqref{eq:bound_xi},
    \begin{equation}\label{eq:eqxx}
        \int_{\Omega} \partial_t \tilde{\rho}^{\epsilon} \xi \dd x \leq -\int_{\Omega} \rho^\epsilon \nabla \de^\epsilon \cdot \nabla(p^\epsilon +\xi) \dd x + \int_{\Omega} f^{\epsilon} (p^{\epsilon}+\xi) \dd x - \frac{\dd}{\dd t} \int_{\Omega} \int_{0}^{\tilde{\rho}^{\epsilon}(t, x)} \beta^{-1}(s) \dd s \dd x, 
    \end{equation}
    for all $t \in [t_i, t_{i+1})$.
    Integrating \eqref{eq:eqxx} over the entire time interval $[0,T]$, we obtain
    \begin{equation*}
        \begin{aligned}
            \int_{0}^{T}\int_{\Omega} \partial_t \tilde{\rho}^{\epsilon} \xi \dd x \dd t 
            &\leq -\int_{0}^{T} \int_{\Omega} \rho^\epsilon \nabla \de^\epsilon \cdot \nabla(p^\epsilon +\xi) \dd x \dd t + \int_{0}^{T}\int_{\Omega} f^{\epsilon} (p^{\epsilon}+\xi) \dd x \dd t \\
            &\quad + \int_{\Omega}\int_{0}^{\rho_0(x)} \beta^{-1}(s)\dd s \dd x - \int_{\Omega} \int_{0}^{\tilde{\rho}^\epsilon(T, x)} \beta^{-1}(s) \dd s \dd x.
        \end{aligned}
    \end{equation*}
    Since $\rho^\epsilon$ and $p^\epsilon$ are bounded in $L^{\infty}(0,T; \Lip_{D}(\Omega))$, $\nabla \de^\epsilon$ is bounded, and $\beta^{-1}$ is continuous, the right-hand side is bounded by a constant $C > 0$ independent of $\epsilon$. Taking the supremum over all test functions $\xi \in \K$, we deduce that
    \begin{equation}
        \sup_{\xi \in \K} \int_{0}^{T} \int_{\Omega} \partial_t \tilde{\rho}^{\epsilon} \xi \dd x \dd t \leq C < \infty,
    \end{equation}
    which means that the sequence $(\partial_t \tilde{\rho}^{\epsilon})_{\epsilon}$ is bounded in $L^1(0,T; \Lip_{D}^*(\Om))$.
\end{proof}
	\begin{lemma}\label{lem:rhoep}
    There exist two subsequences of $(\tilde{\rho}^{\epsilon})_{\epsilon}$ and $(\rho^{\epsilon})_{\epsilon}$, denoted respectively by $(\tilde{\rho}^{\epsilon_k})_{k}$ and $(\rho^{\epsilon_k})_{k}$, such that 	
    \[
        \tilde{\rho}^{\epsilon_k} \to \rho \quad \text{in } L^1(0,T; C(\Omega)) \quad \text{and} \quad \rho^{\epsilon_k} \to \rho \quad \text{in } L^1(0,T; C(\Omega)).
    \]
\end{lemma}
\begin{proof}
    Thanks to \cref{lem:rho}, the sequence $(\tilde{\rho}^{\epsilon})_{\epsilon}$ is bounded in $L^{\infty}(0,T; \Lip_D(\Omega))$ and its time derivative $(\partial_t \tilde{\rho}^{\epsilon})_{\epsilon}$ is bounded in $L^1(0,T; \Lip_{D}^*(\Om))$. 
    Recall the inclusions
    \[
        \Lip_{D}(\Om) \hookrightarrow C(\Om) \hookrightarrow \Lip_{D}^*(\Om).
    \]
    Since the embedding $\Lip_{D}(\Om) \hookrightarrow C(\Om)$ is compact, and $C(\Om) \hookrightarrow \Lip_{D}^*(\Om)$ is continuous, the Aubin-Lions-Simon lemma (\cref{lem:ALS}) ensures the existence of a subsequence $(\tilde{\rho}^{\epsilon_k})_{k}$ and a limit function $\rho$ such that
    \begin{equation}\label{eq:vek}
        \tilde{\rho}^{\epsilon_k} \to \rho \quad \text{in } L^s(0,T; C(\Omega)), \quad \text{for all } 1 \le s < +\infty.
    \end{equation}
     By definition, for any $t \in [t_i, t_{i+1})$, we have $\rho^\epsilon(t) = \rho^{i+1}$, which yields, by the definition of $\tilde{\rho}^{\epsilon}$,
    \begin{equation}
        \tilde{\rho}^{\epsilon}(t) - \rho^\epsilon(t) = (t - t_{i+1}) \partial_t \tilde{\rho}^{\epsilon}(t).
    \end{equation}
    We get the estimate
    \begin{equation}\label{eq:vek2}
        \Vert\tilde{\rho}^{\epsilon} - \rho^{\epsilon}\Vert_{L^1(0,T; \Lip_{D}^*(\Om))} \leq \epsilon \Vert\partial_t \tilde{\rho}^{\epsilon}\Vert_{L^1(0,T; \Lip_{D}^*(\Om))}.
    \end{equation}
    From \cref{lem:rho}, $(\partial_t \tilde{\rho}^{\epsilon})_{\epsilon}$ is bounded in $L^1(0,T; \Lip_{D}^*(\Om))$. Thus, taking the limit as $\epsilon \to 0$ in \eqref{eq:vek2}, we deduce that $\Vert\tilde{\rho}^{\epsilon} - \rho^{\epsilon}\Vert_{L^1(0,T; \Lip_{D}^*(\Om))} \to 0$. By the triangle inequality,
    \begin{equation}\label{eq:rhoep}
        \Vert\rho^{\epsilon_k} - \rho\Vert_{L^1(0,T; \Lip^{*}_{D}(\Om))} \leq \Vert\rho^{\epsilon_k} - \tilde{\rho}^{\epsilon_k}\Vert_{L^1(0,T; \Lip^{*}_{D}(\Om))} + \Vert\tilde{\rho}^{\epsilon_k} - \rho\Vert_{L^1(0,T; \Lip^{*}_{D}(\Om))}.
    \end{equation}
    Since $\tilde{\rho}^{\epsilon_k} \to \rho$ in $L^1(0,T; C(\Omega))$, it also converges in $L^1(0,T; \Lip_D^*(\Omega))$. Combining this with the vanishing difference \eqref{eq:vek2}, we conclude that 	
    \begin{equation}\label{eq:convrho}
        \rho^{\epsilon_k} \to \rho \quad \text{in } L^1(0,T; \Lip_D^*(\Om)).
    \end{equation}
 We know from \cref{lem:rho} that the sequence $(\rho^{\epsilon_k})_{k}$ is bounded in $L^{\infty}(0,T; \Lip_D(\Omega))$. Therefore, 
    \begin{equation*}
        \rho^{\epsilon_k} \to \rho \quad \text{in } L^1(0,T; C(\Omega)),
    \end{equation*}
    which completes the proof.
\end{proof}
\begin{lemma}\label{lem:peps}
    There exists a subsequence $(p^{\epsilon_k})_{k}$ of $(p^{\epsilon})_{\epsilon}$ and a function $p \in L^1(0,T; C(\Omega))$, with $\rho=\beta(p)$, such that
    \begin{equation*}
        p^{\epsilon_k} \to p \quad \text{in } L^1(0,T; C(\Omega)).
    \end{equation*}
    Moreover, $p(t, \cdot) \in \K$ for almost all $t \in (0,T)$.
\end{lemma}
\begin{proof}

Thanks to \cref{lem:rhoep}, we know that 
    \begin{equation}\label{eq:rr}
        \rho^{\epsilon_k} \to \rho \quad \text{in } L^1(0,T; C(\Omega)).
    \end{equation}
    Let $p= \beta^{-1}(\rho)$. From \cref{lem:rho}, the sequence $(p^{\epsilon_k})_{k}$ is bounded in $L^\infty(0,T; \Lip_D(\Omega))$, and thus in $L^\infty((0,T)\times\Omega)$. Since $\beta^{-1}$ is Lipschitz continuous on compact sets, there exists a constant $L > 0$ such that, for almost every $(t,x)$, 
    \begin{equation}\label{eq:mg}
        \vert p^{\epsilon_k}(t,x) - p(t,x) \vert \leq L \vert \rho^{\epsilon_k}(t,x) - \rho(t,x) \vert,
    \end{equation}
    which implies that $\Vert p^{\epsilon_k} - p \Vert_{L^1(0,T; C(\Omega))} \leq L \Vert \rho^{\epsilon_k} - \rho \Vert_{L^1(0,T; C(\Omega))}$. Using \eqref{eq:rr}, we deduce that $p^{\epsilon_k} \to p$ in $L^1(0,T; C(\Omega))$.
    To show that $p(t, \cdot) \in \K$ for \pp $t\in (0,T)$, we notice that for any $\epsilon_k >0$, the function $p^{\epsilon_k}(t, \cdot) \in \K$ is $1$-Lipschitz continuous, \ie
        \begin{equation}\label{eq:lipsh}
        \vert p^{\epsilon_k}(t,x) - p^{\epsilon_k}(t,y) \vert \le \vert x-y \vert, \quad \text{for any } x, y \in \Omega.
    \end{equation}
    Passing to the limit in \eqref{eq:lipsh} as $k \to \infty$ yields
    \begin{equation*}
        \vert p(t,x) - p(t,y)\vert \le \vert x-y \vert, \quad \text{for any } x, y \in \Omega.
    \end{equation*}
    Thus, $p(t, \cdot)$ is also uniformly $1$-Lipschitz continuous. By Rademacher's theorem, $p(t, \cdot)$ is differentiable almost everywhere in $\Omega$ with $\vert \nabla p(t, \cdot) \vert \leq 1$. Furthermore, the uniform convergence preserves the homogeneous Dirichlet boundary condition on $\Gamma_D$. Consequently, $p(t, \cdot) \in \K$ for \pp $t\in (0,T)$.
\end{proof}
\begin{lemma}\label{lem:l5}
    Let $\xi \in W^{1,\infty}(\Omega)$ be such that $\vert\nabla \xi(x)\vert \le c(x)$ \pp, and let $(c_n)_{n \in \mathbb{N}}$ be a sequence of positive functions in $C(\overline{\Omega})$ such that $c_n \to c$ in $C(\overline{\Omega})$ as $n \to \infty$. Then, for any $\alpha \in (0,1)$, there exists $n_\alpha\in\N$ such that, for any $n \ge n_\alpha$, 
    \[
        \vert\nabla \xi\vert\le \frac{c_n}{1-\alpha} \quad \text{\pp~ in } \Omega.
    \]
\end{lemma}

\begin{proof}
    Since $c_n(x) > 0$, we can write $\vert\nabla \xi(x)\vert \le c(x) = \left(\frac{c(x)}{c_n(x)}\right) c_n(x)$, and since $c_n \to c$ uniformly in $C(\overline{\Omega})$, the ratio $\frac{c}{c_n}$ converges to $1$ uniformly. Thus, for any $\epsilon > 0$, there exists $n_0 \in \N$ such that for all $n \geq n_0$ and all $x \in \Omega$, we have $\frac{c(x)}{c_n(x)} \le 1+\epsilon$.
    This implies that 
    \[
        \vert\nabla \xi(x)\vert \leq (1+\epsilon) c_n(x), \quad \text{for all } n \ge n_0.
    \]
    Now pick $\alpha \in (0,1)$, and thus $\frac{1}{1-\alpha} > 1$. Taking $\epsilon > 0$ small enough such that $1+\epsilon \le \frac{1}{1-\alpha}$ and $n_\alpha = n_0$ yields $\vert\nabla \xi\vert \le \frac{c_n}{1-\alpha}$, as claimed.
\end{proof}
\begin{lemma}\label{lem:Deps}
    There exists a subsequence of $(\de^\epsilon)_{\epsilon}$, denoted by $(\de^{\epsilon_k})_{k}$, and a limit function $\de \in L^s(0,T; W^{1,s}(\Omega))$, such that 	
    \begin{equation*}
        \de^{\epsilon_k} \to \de \quad \text{in } L^s(0,T; W^{1,s}(\Omega)).
    \end{equation*}
    Moreover, for \pp $t \in [0,T]$, $\de(t, \cdot)$ is the maximal subsolution to the Eikonal equation
    \begin{equation}\leqnomode
        \tag{$\textup{E}$}
        \label{eq:eikonal-H}
        \begin{cases}
            \vert \nabla \de(t, \cdot) \vert = \he(\rho(t, \cdot)) & \text{in } \Omega, \\
            \de(t, \cdot) = 0 & \text{on } \Gamma_D.
        \end{cases}
    \end{equation}
\end{lemma}
\begin{proof}
    Since $0\leq \rho^{\epsilon_k}\leq 1$ \pp, the sequence $(\rho^{\epsilon_k})_{k}$ is uniformly bounded in $L^{\infty}(Q)$. Since $\he$ is continuous, the sequence $\kk_k = \he(\rho^{\epsilon_k})$ is also bounded in $L^{\infty}(Q)$, which implies that $(\de^{\epsilon_k})_{k}$ is bounded in $L^{s}(0,T; W^{1,s}(\Omega))$. Consequently, there exists a subsequence, still denoted by $(\de^{\epsilon_k})_{k}$, and a limit function $\de \in L^{s}(0,T; W^{1,s}(\Omega))$ such that
    \begin{equation}\label{eq:weak-de}
        \de^{\epsilon_k} \rightharpoonup \de \quad \text{weakly in } L^s(0,T; W^{1,s}(\Omega)).
    \end{equation}
    We now prove that, for almost every $t \in [0,T]$, $\de(t, \cdot)$ is the maximal subsolution to \eqref{eq:eikonal-H} in the sense of \cref{EnnajiaugmentedLag}. Take $\kk = \he(\rho)$. Since $p^{\epsilon_k} \to p$ in $L^1(0,T; C(\Omega))$ and $\beta$ is Lipschitz continuous, we have that $\rho^{\epsilon_k} =\beta(p^{\epsilon_k}) \to \beta(p):=\rho$ in $L^1(0,T; C(\Omega))$. Since $\he$ is Lipschitz continuous, we have $\kk_k \to \kk$ in $L^1(0,T; C(\Omega))$. Moreover, up to a subsequence, we have that $\kk_k(t) \to \kk(t)$ uniformly in $C(\overline{\Omega})$ for \pp $t \in [0,T]$.

    Now take $t\in [0,T]$. For any test function $\xi \in W^{1,\infty}_D(\Om)$ satisfying $\vert\nabla \xi\vert \le \kk(t)$ \pp, and for any $\delta \in (0,1)$, we define $\xi_\delta = (1-\delta) \xi$. By \cref{lem:l5}, since $\kk_k(t) \to \kk(t)$ in $C(\overline{\Omega})$, there exists an integer $k_0 \ge 0$ such that for all $k \ge k_0$, we have $\vert\nabla \xi\vert \le \frac{\kk_k(t)}{1-\delta}$, which implies $\vert\nabla \xi_\delta\vert \le \kk_k(t)$.

    Since $\de^{\epsilon_k}(t, \cdot)$ is the maximal subsolution of the regularized Eikonal equation, and $\xi_\delta$ is an admissible subsolution, we have by definition
    \[
        \int_{\Omega} \de^{\epsilon_k}(t,x) \dd x \ge \int_{\Omega} \xi_\delta(x) \dd x = (1-\delta) \int_{\Omega} \xi(x) \dd x, \quad \text{for all } k \ge k_0.
    \]
    Multiplying this inequality by a non-negative test function $\sigma \in \D(0,T)$ ($\sigma \ge 0$) and integrating over $[0,T]$, we obtain
    \[
        \int_{0}^{T} \int_{\Omega} \de^{\epsilon_k}(t,x) \sigma(t) \dd x \dd t \ge (1-\delta) \int_{\Omega} \xi(x) \int_{0}^{T} \sigma(t) \dd x \dd t.
    \]
    Using the weak convergence \eqref{eq:weak-de} to pass to the limit as $k \to \infty$, gives 
    \[
        \int_{0}^{T} \int_{\Omega} \de(t,x) \sigma(t) \dd x \dd t \ge (1-\delta) \int_{\Omega} \xi(x) \int_{0}^{T} \sigma(t) \dd x \dd t.
    \]
    Taking the limit as $\delta \to 0^+$, yields
    \[
        \int_{0}^{T} \int_{\Omega} \de(t,x) \sigma(t) \dd x \dd t \ge \int_{\Omega} \xi(x) \int_{0}^{T} \sigma(t) \dd x \dd t.
    \]
    Since this holds for any nonnegative test function $\sigma \in \D(0,T)$, 
    \[
        \int_{\Omega} \de(t,x) \dd x \ge \int_{\Omega} \xi(x) \dd x,~\text{for \pp } t\in [0,T].
    \]
    We deduce that $\de(t, \cdot)$ is the maximal subsolution to \eqref{eq:eikonal-H}, that is  $\vert\nabla \de(t, \cdot)\vert = \he(\rho(t, \cdot))$ \pp in $\Omega$. 
    
    Finally, notice that since $\he(\rho^{\epsilon_k}) \to \he(\rho)$ strongly in $L^s(0,T; C(\Omega))$, we have $\vert\nabla \de^{\epsilon_k}\vert \to \vert\nabla \de\vert$ strongly in $L^s(0,T; C(\Omega))$, and thus in $L^s(0,T; L^s(\Omega))$. Because $\de^{\epsilon_k} = 0$ on $\Gamma_D$, Poincaré's inequality ensures that $\Vert\nabla u \Vert_{L^s((0,T)\times\Omega)}$ is an equivalent norm to $\Vert u\Vert_{L^s(0,T; W^{1,s}_D(\Omega))}$ on the space $L^s(0,T; W^{1,s}_D(\Omega))$. Since this space is uniformly convex for $1 < s < \infty$ (see, \textit{e.g.}, \cite[Theorem 3.6]{Adam&Fournier}) and 
    \[
        \limsup_{k \to \infty} \Vert\de^{\epsilon_k}\Vert_{L^s(0,T; W^{1,s}_D(\Omega))} \leq \Vert \de\Vert_{L^s(0,T; W^{1,s}_D(\Omega))},
    \]
    we conclude by the Radon-Riesz property (see, \textit{e.g.}, \cite[Proposition 3.32]{brezis-FASSPDE}) that 
    \[
        \de^{\epsilon_k} \to \de \quad \text{strongly in } L^s(0,T; W^{1,s}(\Omega)),
    \]
    which completes the proof.
   \end{proof}
\begin{proof}[Proof of \cref{thm:th2}]
    First, recall that by testing the regularized problem against $\xi \in \K$, we obtain, for \pp $t \in [t_i, t_{i+1})$,
    \begin{equation}\label{eq:xix1}
        \int_{\Omega} \partial_t \tilde{\rho}^{\epsilon} T_k(\beta^{-1}(\tilde{\rho}^{\epsilon})-\xi) \dd x + \int_{\Omega} \rho^{\epsilon} \nabla \de^{\epsilon} \cdot \nabla T_k(p^{\epsilon}-\xi) \dd x \le \int_{\Omega} f^{\epsilon} T_k(p^{\epsilon}-\xi) \dd x.
    \end{equation}
    Multiplying \eqref{eq:xix1} by a non-negative test function $\sigma \in C^{\infty}_{c}([0,T))$ and integrating by parts over $[0,T]$, we get
    \begin{equation}\label{eq:pet}
        \begin{aligned}
        &-\int_{0}^{T}\partial_t \sigma(t) \int_{\Omega} \left(\int_{0}^{\tilde{\rho}^{\epsilon}(t,x)}T_k(\beta^{-1}(s)-\xi) \dd s \right) \dd x \dd t  - \sigma(0) \int_{\Om}\int_{0}^{{\rho}_0(x)} T_{k}(\beta^{-1}(s) -\xi)\dd x\\
        & \int_{0}^{T}\int_{\Omega} \rho^{\epsilon} \nabla \de^{\epsilon} \cdot \nabla T_k(p^{\epsilon}-\xi)\sigma(t) \dd x \dd t \leq \int_{0}^{T} \int_{\Omega} f^{\epsilon} T_k(p^{\epsilon}-\xi) \sigma(t) \dd x \dd t.
        \end{aligned}
    \end{equation}
    We now pass to the limit as $\epsilon \to 0$ in each term of \eqref{eq:pet}. 
    
    Thanks to \cref{lem:rhoep}, the sequence $(\tilde{\rho}^{\epsilon})_{\epsilon>0}$ converges, up to a subsequence, to a function $\rho \in L^1(0,T; C(\overline{\Omega}))$. Since the sequence is uniformly bounded and the function $z \mapsto \int_0^z T_k(\beta^{-1}(s)-\xi) \dd s$ is continuous, Lebesgue's dominated convergence theorem ensures that the first term of \eqref{eq:pet} converges to 
    \[
        -\int_{0}^{T}\partial_t \sigma(t) \int_{\Omega} \left(\int_{0}^{\rho(t,x)} T_k(\beta^{-1}(s)-\xi) \dd s \right) \dd x \dd t.
    \]
   
    Regarding the second term on the left-hand side of \eqref{eq:pet}, we know from Lemmas \ref{lem:rhoep}, \ref{lem:peps}, and \ref{lem:Deps} that the sequences $(\rho^{\epsilon_j})_{j}$, $(p^{\epsilon_j})_{j}$, and $(\de^{\epsilon_j})_{j}$ converge respectively to $\rho$, $p$, and $\de$, where $\rho=\beta(p)$ and $\de$ is the maximal subsolution of \eqref{eq:eikonal-H}. Since $T_k$ is Lipschitz continuous, the gradients $\nabla T_k(p^{\epsilon_j} -\xi)$ are uniformly bounded in $L^\infty(Q)$, and thus they converge weakly-* in $L^\infty(Q)$ and weakly in any $L^s(Q)$ to $\nabla T_k(p-\xi)$. By combining this with the strong convergence of $\rho^{\epsilon_j} \to \rho$ in $L^s(0,T; C(\overline{\Omega}))$ and $\nabla \de^{\epsilon_j} \to \nabla \de$ strongly in $L^s(0,T; L^s(\Omega))$, we obtain
    \begin{equation*}
        \lim_{j \to \infty} \int_{0}^{T}\int_{\Omega} \rho^{\epsilon_j} \nabla \de^{\epsilon_k} \cdot \nabla T_k(p^{\epsilon_j}-\xi)\sigma(t) \dd x \dd t = \int_{0}^{T}\int_{\Omega} \rho \nabla \de \cdot \nabla T_k(p-\xi)\sigma(t) \dd x \dd t. 
    \end{equation*}
    Finally, since $f^{\epsilon_j} \to f$ strongly in $L^1(0,T; \Lip_D^*(\Om))$ and the test function $T_k(p^{\epsilon_j}-\xi)\sigma(t)$ is uniformly bounded in $L^\infty(0,T; \Lip_D(\Omega))$ and converges to $T_k(p-\xi)\sigma(t)$, passing to the limit as $j\to\infty$ on the right-hand side of \eqref{eq:pet}, yields
    \begin{equation*}
        \lim_{j \to \infty} \int_{0}^{T} \langle f^{\epsilon_j}, T_k(p^{\epsilon_j}-\xi) \rangle_{\Lip_D^*, \Lip_D} \sigma(t) \dd t = \int_{0}^{T} \langle f, T_k(p-\xi) \rangle_{\Lip_D^*, \Lip_D} \sigma(t) \dd t.
    \end{equation*}
   Consequently, passing to the limit $\epsilon \to0$ in \eqref{eq:pet}, we obtain that for any $\xi \in \K$ and any nonnegative test function $\sigma \in C^{\infty}_{c}([0,T))$, 
    \begin{equation}
        \begin{aligned}
        -\int_{0}^{T}\partial_t \sigma(t) \int_{\Omega} &\left( \int_{0}^{\rho(t,x)} T_k(\beta^{-1}(s)-\xi) \dd s \right) \dd x \dd t  -\sigma(0) \int_{\Om}\int_{0}^{\rho_{0}(x)} T_{k}(\beta^{-1}(s) -\xi)\dd x \\
        & \int_{0}^{T}\int_{\Omega} \rho \nabla \de \cdot \nabla T_k(p-\xi)\sigma(t) \dd x \dd t \le \int_{0}^{T} \langle f, T_k(p-\xi) \rangle_{\Lip_D^*, \Lip_D} \sigma(t) \dd t.
        \end{aligned}
    \end{equation}
   In other words, $(\rho, p, \de)$ is a variational solution to \eqref{eq:regularized-model} in the sense of \cref{thm:th2}. This concludes the proof.
\end{proof}
			
\section{Numerical approximation}\label{section:4}
We follow the main lines of the prediction-correction approach developed in \cite{ennaji2023prediction}. It is worth noting that while the theoretical analysis is focused on the regularized model \eqref{eq:regularized-model}, our proposed numerical scheme is flexible enough to directly approximate and simulate the hard congestion Hughes' model \eqref{evolution223} in the presence of the maximal monotone $\Signp$ graph.
\subsection{Domain discretization}
We consider a spatial domain $\Omega$ representing a room, with its boundary partitioned as $\partial \Omega = \Gamma_N \cup \Gamma_D$. Here, $\Gamma_D$ corresponds to the exit doors, and $\Gamma_N$ represents the impenetrable walls, as illustrated in \cref{fig:domain}. 

\begin{figure}[H]
	\centering
	\includegraphics[scale=0.27]{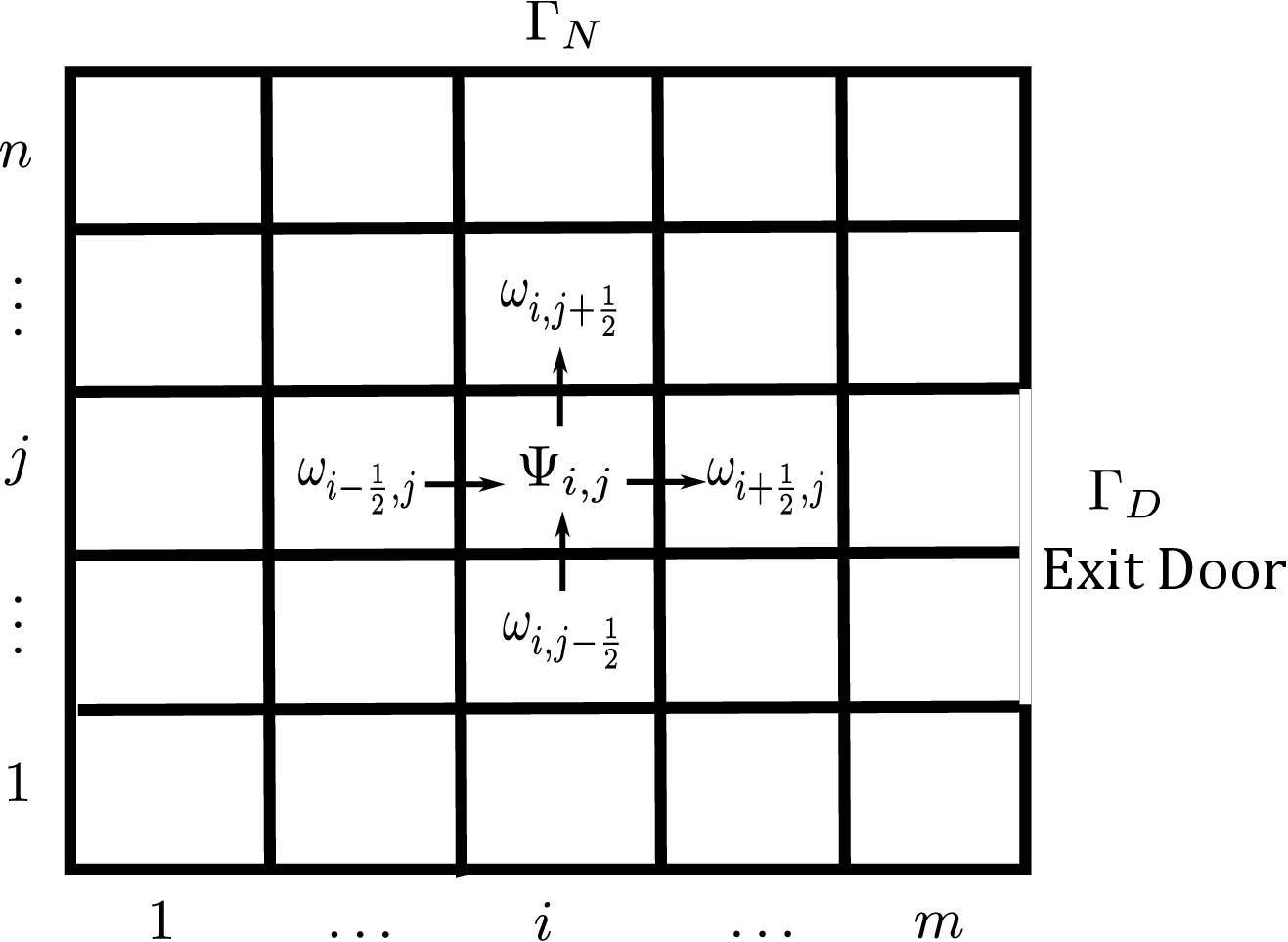}
	\caption{Discretization of the domain $\Omega$.}
	\label{fig:domain}
\end{figure}

We discretize $\Omega$ into a Cartesian grid of $m \times n$ square control volumes of side length $h$. For $1 \leq i \leq m$ and $1 \leq j \leq n$, let $C_{i,j}$ denote the cell at position $(i,j)$. To build our finite-volume scheme, we distinguish between cell-centered and face-centered variables. We denote by $\Psi_{i,j}$ the cell-averaged value of a generic scalar field $\Psi$ over $C_{i,j}$, while variables denoted by $w_{i-\frac{1}{2},j}$ and $w_{i+\frac{1}{2},j}$ are evaluated at the cell interfaces and typically represent numerical fluxes or gradients across the boundaries of $C_{i,j}$.

Depending on the considered step of our numerical scheme, these generic discrete quantities take on different physical and mathematical meanings.

\begin{itemize}
	\item In the prediction step: $\Psi_{i,j}$ represents the intermediate population density $\rho^{k+\frac{1}{2}}$, while the interface values correspond to the components of the transport flux $\rho V$.	
	\item In the correction step (primal problem): $\Psi_{i,j}$ represents the newly corrected density $\rho^{k+1}$, and the interface variables correspond to the components of the decongestion flux $\Phi$. 
    \item In the correction step (dual problem): $\Psi_{i,j}$ represents the congestion pressure $p$ (the dual variable), and the interface values correspond to the components of its discrete gradient $\nabla_h p$.
\end{itemize}
\subsection{Operator Splitting Method}
We aim to approximate the solution of \eqref{eq:regularized-model} using an operator splitting scheme. We discretize the time interval $[0,T]$ into sub-intervals of the form $[t_k,t_{k+1}]$, where $t_k = k\tau$ for $k=0,\dots,N-1$, and $\tau > 0$ is the time step size.

Given the density $\rho^k$ and the congestion pressure $p^k$ at time $t_k$, the updated variables $\rho^{k+1}$ and $p^{k+1}$ at time $t_{k+1}$ are computed in two steps: a \emph{prediction} step (or transport) followed by a \emph{correction} step (or decongestion). To do so, we introduce an intermediate time $t_{k+\frac{1}{2}}$.

\begin{itemize}
	\item \textbf{Prediction step:} Starting from the current density $\rho^k$, we compute an intermediate predicted density $\rho^{k+\frac{1}{2}}$ by solving the continuity equation. We denote by $\tilde{\rho}$ the solution of the transport sub-problem
	\begin{equation}\label{H11}
		\left\{
		\begin{array}{ll}			
			\partial_t \tilde{\rho} + \dive(\tilde{\rho}\: V^k ) = 0 \quad & \hbox{in } (t_k, t_{k+\frac{1}{2}}] \times \Omega, \\ 
			\tilde{\rho}(t_{k}, \cdot) = \rho^k,
		\end{array}
		\right.
	\end{equation}
	and we set the predicted density as $\rho^{k+\frac{1}{2}} = \tilde{\rho}(t_{k+\frac{1}{2}}, \cdot)$. The advection field is given by $V^k=-\nabla \de^k$, where the potential $\de^k$ solves the Eikonal equation at instant $t_k$,
	\begin{equation}
	\left\lbrace\begin{array}{ll}
		\vert \nabla \de^k \vert = \he(\rho^k) & \hbox{in }\Om,\\ 
		\de^k = 0 & \hbox{on } \Gamma_D.
	\end{array}
	\right.
	\end{equation}

	\item \textbf{Correction step:} The predicted density $\rho^{k+\frac{1}{2}}$ might violate the maximal density constraint. To obtain an admissible density at time $t_{k+1}$, we correct $\rho^{k+\frac{1}{2}}$ by considering the decongestion part of our modified Hughes' model, taking $\rho^{k+\frac{1}{2}}$ as the initial condition for the following problem:
	 	\begin{equation} \label{eq_hu}	\left\{
		\begin{array}{ll}
			\left. 	\begin{array}{l}
				\partial_t\bar{\rho} -\dive(\Phi)  =0,   \\  \\
				\Phi=m \nabla p, \: m\geq0, \: \vert \nabla p\vert \leq 1,\:  m(1-\vert \nabla p\vert)=0\\   \\
				\bar{\rho}  =   \beta(p) 
			\end{array} \right\}  \quad & \hbox{ in } [t_{k+\frac{1}{2}},t_{k+1}[ \times \Omega,\\  \\
			\Phi \cdot \nu =0  & \hbox{ on }[t_{k+\frac{1}{2}},t_{k+1}[ \times\Gamma_N,\\  \\
			p=0   & \hbox{ on }[t_{k+\frac{1}{2}},t_{k+1}[ \times\Gamma_D, \\ \\
			\bar{\rho}(t_{k+\frac{1}{2}})=\tilde{\rho}_{k+\frac{1}{2}}.
		\end{array}
		\right.
	\end{equation} 

\end{itemize}

By applying an implicit Euler time discretization to the continuity equation with time step $\tau$, we obtain the following stationary system for the updated variables $(\rho^{k+1}, \Phi^{k+1}, p^{k+1})$ at time $t_{k+1}$:
\begin{equation} \label{eq_hu2}
    \left\{
    \begin{array}{ll}
        \rho^{k+1} - \tau \:\dive(\Phi^{k+1}) = \rho^{k+\frac{1}{2}} & \hbox{in } \Omega,\\ 
        \Phi^{k+1} = m \nabla p^{k+1}, \quad m\geq0, \quad \vert \nabla p^{k+1}\vert \leq 1,\quad m(1-\vert \nabla p^{k+1}\vert)=0 & \hbox{in } \Omega,\\ 
        \rho^{k+1} = \beta(p^{k+1}) & \hbox{in } \Omega,\\ 
        \Phi^{k+1} \cdot \nu = 0 & \hbox{on } \Gamma_N,\\ 
        p^{k+1} = 0 & \hbox{on } \Gamma_D.
    \end{array}
    \right.
\end{equation}

Using the results of \cite{ennaji2023prediction}, the solutions $\rho^{k+1}$ and $p^{k+1}$ of \eqref{eq_hu2} can be recast as the solutions of the Beckmann-like optimization problem
\begin{equation}\label{minproc22}
    \begin{array}{l}
        \inf\left \{  \int_\Omega \tau \vert\Phi(x) \vert\: \dd x +  \int_{\Omega} \tau \BB(\rho(x))\dd x   \: :\: \rho \in L^\infty(\Omega), \right.  \\  
        \left. \hspace*{1.5cm}  \Phi\in (L^1(\Omega))^N,\: \Phi \cdot \nu=0 \hbox{ on } \Gamma_N,\:    -\tau\dive(\Phi) = \rho^{k+\frac{1}{2}} -\rho    \hbox{ in  }  \Omega \right\},
    \end{array}
\end{equation}
where $\BB(r) = \int_{0}^{r}\beta(s)\dd s$. Of course, the main difference between \eqref{minproc22} and the exact problem considered in \cite{ennaji2023prediction} is the presence of the smooth functional $\BB(\rho)$, instead of the rigid indicator function $\I_{[0,1]}(\rho)$.

Notice that the corresponding dual problem associated with \eqref{minproc22} reads
\begin{equation}\label{dual221}
    \min_{z}\left\{ \int_{\Om} \tau \BB^{*}(z)\dd x  - \int_\Om  z\:  \rho^{k+\frac{1}{2}} \: \dd x \: :\:  z\in W^{1,\infty}_{D}(\Omega)~\hbox{and}~\vert\nabla z \vert \leq 1~\pp\right\}.
\end{equation}

In the next section, we present the implementation of the proposed splitting method to compute the density $\rho$, which is the solution of the problem \eqref{eq:regularized-model}. \\
\subsection{Discretization of the transport equation \eqref{H11}}
In the prediction step, we compute the intermediate density $\rho^{k+\frac{1}{2}} = \tilde{\rho}(t_{k+\frac{1}{2}}, \cdot)$ by solving the continuity equation
\begin{equation}
	\left\{
	\begin{array}{ll}
		\partial_t\tilde{\rho} + \dive(\tilde{\rho}\: V^k )=0\quad & \hbox{ in } (t_k, t_{k+\frac{1}{2}}] \times \Omega, \\  \\
		\tilde{\rho}(t_{k}, \cdot)=\rho^k,
	\end{array}
	\right.
\end{equation}
where $V^k = -\nabla \de^k$ is the velocity field evaluated at time $t_k$. This can be rewritten in scalar form as 	
\begin{equation} \label{VF1}
	\left\{
	\begin{array}{ll}
		\partial_t\tilde{\rho}(t,x,y) + \partial_x \tilde{\FF}(t,x,y) + \partial_y \tilde{\GG}(t,x,y)=0\quad & \hbox{ in } (t_k, t_{k+\frac{1}{2}}] \times \Omega, \\   \\
		\tilde{\rho}(t_{k}, x, y)=\rho^k(x, y),
	\end{array}
	\right.
\end{equation}
with the flux components given by $\tilde{\FF} = \tilde{\rho} V^k_x$ and $\tilde{\GG} = \tilde{\rho} V^k_y$.

We discretize equation \eqref{VF1} by combining a Finite Volume Method (FVM) in space and an explicit Euler scheme in time \cite{finitevolumeGallouet}. This leads to the following update:
\begin{equation*}
	\frac{\rho_{i,j}^{k+\frac{1}{2}}-\rho_{i,j}^{k}}{\tau} + \frac{1}{h} \left(\tilde{\FF}^k_{i+\frac{1}{2},j}-\tilde{\FF}^k_{i-\frac{1}{2},j}\right) + \frac{1}{h} \left(\tilde{\GG}^k_{i,j+\frac{1}{2}}-\tilde{\GG}^k_{i,j-\frac{1}{2}}\right)=0,
\end{equation*}
where $\tilde{\FF}^k_{i\pm\frac{1}{2},j}$ and $\tilde{\GG}^k_{i,j\pm\frac{1}{2}}$ are the numerical fluxes evaluated at the interfaces of the cell $C_{i,j}$ (see \cref{fig:domain}). 

To ensure numerical stability and prevent oscillations, these numerical fluxes must be carefully approximated using upwind-type schemes. Well-known examples include the Godunov, Lax-Wendroff, and Rusanov (also known as local Lax-Friedrichs) schemes \cite{finitevolumeGallouet, Rusanovflux}. In this work, we employ the Rusanov approximation. Since the advection velocity $V^k$ is fixed during the prediction step, the transport equation is linear with respect to the density. Therefore, the local wave speed at the interface is determined solely by the magnitudes of the velocity components. The numerical flux at the interface $(i+\frac{1}{2}, j)$ is given by
\begin{equation}\label{eq:rusanov}
\FF_{i+\frac{1}{2},j}^{k,\textup{Rus}} = \frac{1}{2}\left(\tilde{\FF}^k_{i,j} + \tilde{\FF}^k_{i+1,j}\right) - \frac{\lambda_{i+\frac{1}{2},j}^x}{2}\left(\rho^k_{i+1,j}-\rho^k_{i,j}\right),
\end{equation}
where $\tilde{\FF}^k_{i,j} = \rho^k_{i,j} (V^k_x)_{i,j}$ and the local maximum wave speed is defined as 
\[
\lambda_{i+\frac{1}{2},j}^x = \max\left(\vert (V^k_x)_{i,j} \vert , \vert (V^k_x)_{i+1,j} \vert\right).
\]
A similar formulation is applied for the $y$-directional flux $\GG^{k,\textup{Rus}}_{i,j+\frac{1}{2}}$.

To sum up, the explicit implementation of the prediction step reads
\begin{equation}\label{t2}
	\rho_{i,j}^{k+\frac{1}{2}} = \rho_{i,j}^{k} - \frac{\tau}{h} \left(\FF_{i+\frac{1}{2},j}^{k,\textup{Rus}}-\FF_{i-\frac{1}{2},j}^{k,\textup{Rus}}\right) - \frac{\tau}{h} \left({\GG}^{k,\textup{Rus}}_{i,j+\frac{1}{2}}-{\GG}^{k,\textup{Rus}}_{i,j-\frac{1}{2}}\right).
\end{equation}
\subsection{Discretization of the minimum flow problem \eqref{minproc22}}

Recall that the density $\rho^{k+\frac{1}{2}}$ obtained in the prediction step is corrected by solving the minimum flow problem
\begin{equation}
    \label{beckmann1}
    \inf_{(\rho,\Phi)}\left\{  \int_\Om\tau\vert\Phi(x)\vert\dd x + \int_{\Om}\tau \BB(\rho(x))\dd x:~-\tau \:\dive(\Phi) = \rho^{k+\frac{1}{2}} - \rho\hbox{ in }  \Omega,\ \Phi\cdot \nu=0\hbox{ on }\Gamma_N \right\}.
\end{equation}
Here, the congestion constraint is entirely handled by the potential functional $\BB(\rho)$. This general formulation encompasses both our models:
\begin{itemize}
    \item For \eqref{eq:regularized-model}, $\BB(\rho) = \int_0^\rho \beta(s)\dd s$, where $\beta$ satisfies \cref{assumption:1}.
    \item For \eqref{evolution223} , $\BB(\rho) = \I_{[0,1]}(\rho)$, which acts as a hard constraint enforcing $0 \leq \rho \leq 1$.
\end{itemize}

The problem \eqref{beckmann1} can be recast in the abstract optimization form
\begin{equation}
    \leqnomode
    \tag{M}
    \label{formulation}
    \min_{(\rho,\Phi)} \A(\rho,\Phi) +  \I_{\CC} (\Lambda(\rho,\Phi)),
\end{equation}
where $\A(\rho,\Phi) =\int_\Om \tau\vert\Phi(x)\vert\dd x + \int_{\Om}\tau \BB(\rho(x))\dd x$, the linear operator is given by $\Lambda(\rho,\Phi)  = \rho-\tau\dive\Phi$, and the constraint set is $\CC = \{ \rho^{k+\frac{1}{2}} \}$.

Based on the discrete gradient and divergence operators defined in \cref{section:discrete_op}, we propose a fully discrete version of \eqref{formulation}, denoted $\M_d$:
\begin{equation}
    \label{P1}
    \min_{(\rho,\Phi)} \A_h(\rho,\Phi) + \B_h(\Lambda_h(\rho,\Phi)),
\end{equation}
where the discrete functionals are defined by
\begin{equation}
    \label{eq:functionals_h}
    \A_h(\rho,\Phi) = h^2\sum_{i=1}^{m}\sum_{j=1}^{n}\left( \tau\Vert\Phi_{i,j}\Vert  + \tau\BB(\rho_{i,j}) \right) \quad \mbox{and} \quad \B_h=\I_{\CC_h},
\end{equation}
with $\CC_h:=\left\{ (a_{i,j})\: : \: a_{i,j} = \rho_{i,j}^{k+\frac{1}{2}},~~ \forall (i,j) \right\}$ and $\Lambda_h (\rho,\Phi) = \rho -\tau\dive_h\Phi$.
By introducing the dual variable $p$ (which acts as the congestion pressure), \eqref{P1} can be written in a primal-dual saddle-point form
\begin{equation}
    \label{eq:P1PD}
    \min_{(\rho,\Phi)}\max_{p} \A_h(\rho,\Phi) + \langle p, \Lambda_{h}(\rho,\Phi)\rangle - \B_{h}^{*}(p).
\end{equation}

This saddle-point problem can be efficiently solved using the Chambolle-Pock primal-dual algorithm (PD) \cite{chambolle2011first}. The algorithm requires computing the proximal operators for $\A_h$ and $\B_h^*$. Since $\A_h$ is separable in its variables, its proximal operator decouples into
\begin{equation}
    \label{eq:prox_F}
    \left(\prox_{\sigma\A_h}(\rho,\Phi)\right)_{i,j} = \left( \prox_{\sigma\tau\BB}(\rho_{i,j}), \ \max\left(0, 1 - \frac{\sigma\tau}{\vert\Phi_{i, j}\vert}\right)\Phi_{i, j} \right).
\end{equation}
The computation of $\prox_{\sigma\tau\BB}(\rho)$ depends on the chosen model:
\begin{itemize}
    \item\eqref{evolution223}: $\BB = \I_{[0,1]}$, thus the proximal operator is the standard projection onto the unit interval
    \[
     \prox_{\sigma\tau\BB}(r) = \max(0, \min(1, r)). 
    \]
    \item\eqref{eq:regularized-model}: For a smooth $\beta_\epsilon$, the proximal step reads $\prox_{\sigma\tau\BB}(r) = \argmin_{q} \frac{1}{2}(q-r)^2 + \sigma\tau\BB(q)$. The first-order optimality condition yields the nonlinear equation
    \begin{equation}\label{eq:prox_BB_OC}
        q + \sigma\tau\beta_{\epsilon}(q) = r.
    \end{equation}
    Since $q \mapsto \beta_{\epsilon}(q)$ is strictly monotonically increasing, the function $q \mapsto q+\sigma\tau\beta_{\epsilon}(q)$ is strictly increasing, ensuring that \eqref{eq:prox_BB_OC} admits a unique solution $q$, which can be easily computed using Newton's method.
\end{itemize}

As for the dual functional $\B_{h}^{*}$, we make use of Moreau's identity $p = \prox_{\alpha\B^*_h}(p) + \alpha\prox_{\alpha^{-1}\B_h}(p/\alpha)$ and the fact that $\B_h$ is the indicator of $\CC_h$, which yields
\begin{equation}
    \left(\prox_{\alpha\B^*_h}(p)\right)_{i,j} = p_{i,j} - \alpha\proj_{\CC_{i,j}}(p_{i,j}/\alpha) = p_{i,j} - \alpha \rho_{i,j}^{k+\frac{1}{2}}.
\end{equation}

In summary, denoting $\sigma$ and $\alpha$ the primal and dual step sizes respectively, the proposed algorithm for the correction step is formulated as follows:

\begin{algorithm}[H]
    \caption{Chambolle-Pock (PD) iterations for the Correction Step}
    \label{alg:pd}
    \begin{algorithmic}[1]
        \State $\textbf{Initialization:}$ Choose stepsizes $\sigma, \alpha > 0$ such that $\sigma\alpha\Vert\Lambda_{h}\Vert^2 < 1$. Set $l=0$, initialize primal variables $(\rho^0, \Phi^0)$, and dual variables $p^0 = \bar{p}^0 = p_0$.
        \State $\textbf{Primal step:}$ Update density and fluxes:
        \begin{equation*}
            \begin{aligned}
                \rho_{i,j}^{l+1} &= \prox_{\sigma\tau\BB}\left( \rho_{i,j}^{l} - \sigma \bar{p}^{l}_{i,j} \right) \\
                \Phi_{i, j}^{l+1} &= \max\left(0, 1 - \frac{\sigma\tau}{\vert\Phi_{i, j}^{l} - \sigma\tau\nabla_h \bar{p}_{i,j}^{l}\vert}\right) \left(\Phi_{i, j}^{l} - \sigma\tau\nabla_h \bar{p}_{i,j}^{l}\right).
            \end{aligned}
        \end{equation*}
        \State $\textbf{Dual step:}$ Update the congestion pressure:
        \begin{equation*}
            \begin{aligned}
                v^{l+1}_{i,j} &= p^{l}_{i,j} + \alpha \left( \rho^{l+1}_{i,j} - \tau\dive_h(\Phi^{l+1})_{i,j} \right) \\
                p^{l+1}_{i,j} &= v^{l+1}_{i,j} - \alpha \rho_{i,j}^{k+\frac{1}{2}}.
            \end{aligned}
        \end{equation*}
        \State $\textbf{Extragradient step:}$ Update the extrapolated dual variable:
        \begin{equation*}
            \bar{p}^{l+1} = 2p^{l+1} - p^{l}.
        \end{equation*}
        \State Set $l \gets l+1$ and repeat until convergence.
    \end{algorithmic}
\end{algorithm}
\cref{alg:pd} allows calculating the corrected density $\rho^{k+1}$ in the primal step, and the pressure $p^{k+1}$ in the dual step. The updated pressure $p^{k+1}$ is then used to compute the updated velocity field $V$ for the next time loop.

\begin{remark}
 Notice that in \cref{alg:pd}, if $\Vert\nabla_h p^{l+1}\Vert<1$ is given in a region, then the proximal step for the flux gives $\Phi = 0$. Consequently, one has $\rho^{k+1} = \rho^{k+\frac{1}{2}}$, \ie in this scenario, the correction step leaves the intermediate density unchanged, which means that the numerical scheme reduces to a standard discretization of the Hughes' variant model \eqref{eq:CL-intro}-\eqref{VH1}-\eqref{VH2}.
\end{remark}

\section{Numerical simulations}\label{sm_Hughes}
This section presents numerical experiments to illustrate the effectiveness of our approach. These simulations allow for a comparison among three distinct formulations: the constant-velocity model \eqref{evolgran0} proposed in \cite{ennaji2023prediction}, the soft congestion model \eqref{eq:regularized-model}, and the hard congestion model \eqref{evolution223} introduced in this paper.

\subsection{Setting and configuration}

\subsubsection{Choice of the velocity field}
In all scenarios, we use the macroscopic velocity $V = -\nabla\de$ where $\de$ solves the Eikonal equation \eqref{eq:V-model}. For the constant velocity model \eqref{evolgran0}, we assume that $\he(\rho) \equiv 1$, meaning that $\de$ is simply the Euclidean distance to the exit set $\Gamma_D$. 
For both the soft congestion model \eqref{eq:regularized-model} and the hard congestion variant \eqref{evolution223}, we make use of the exponential cost function $\he(\rho) = e^{\lambda \rho}$ with a sensitivity parameter $\lambda >0$, which is suitable as it smoothly penalizes regions with high pressure. We discuss in \cref{subsection:comp_H,subsection:impact_H} the impact of the parameter $\lambda$ and compare with other costs.
\subsubsection{The choice of the soft congestion constitutive law $\beta$}
To implement the numerical scheme, we must define an explicit choice of the law $\beta(p)$ introduced in the \eqref{eq:regularized-model} model. To this end, we introduce a family of functions $\beta_{\delta}$, where the parameter $\delta>0$ dictates the stiffness of the congestion.
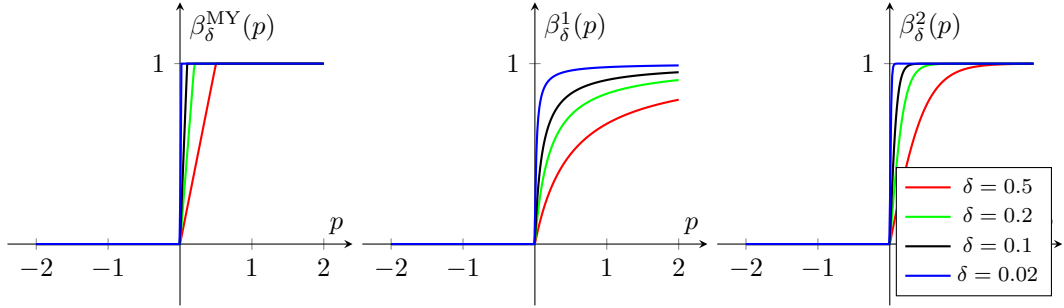
\begin{figure}[htbp]
    \centering
    
    \begin{tikzpicture}
        \begin{axis}[
            width=0.44\textwidth,
            height=0.4\textwidth,
            xlabel={$p$},
            ylabel={$\beta^{\text{MY}}_{\delta}(p)$},
            domain=-2:2,
            samples=300,
            ymin=-0.2, ymax=1.2,
            axis lines=middle,
            xtick={-2,-1,0,1,2},
            ytick={0,1},
            enlargelimits=true
        ]
            \addplot[red, thick] {min(1, max(0, x/0.5))};
            
            \addplot[green, thick] {min(1, max(0, x/0.2))};
            
            \addplot[black, thick] {min(1, max(0, x/0.1))};
            
            \addplot[blue, thick] {min(1, max(0, x/0.02))};
        \end{axis}
    \end{tikzpicture}
    \begin{tikzpicture}
        \begin{axis}[
            width=0.44\textwidth,
            height=0.4\textwidth,
            xlabel={$p$},
            ylabel={$\beta^{1}_{\delta}(p)$},
            domain=-2:2,
            samples=300,
            ymin=-0.2, ymax=1.2,
            axis lines=middle,
            xtick={-2,-1,0,1,2},
            ytick={0,1},
            enlargelimits=true
        ]
            \addplot[red, thick] {x < 0 ? 0 : x / (x + 0.5)};
            \addplot[green, thick] {x < 0 ? 0 : x / (x + 0.2)};
            \addplot[black, thick] {x < 0 ? 0 : x / (x + 0.1)};
            \addplot[blue, thick] {x < 0 ? 0 : x / (x + 0.02)};
        \end{axis}
    \end{tikzpicture}
    \begin{tikzpicture}
        \begin{axis}[
            width=0.44\textwidth,
            height=0.4\textwidth,
            xlabel={$p$},
            ylabel={$\beta^{2}_{\delta}(p)$},
            domain=-2:2,
            samples=300,
            ymin=-0.2, ymax=1.2,
            axis lines=middle,
            xtick={-2,-1,0,1,2},
            ytick={0,1},
            enlargelimits=true,
            legend pos=south east,
            legend style={font=\footnotesize}
        ]
            \addplot[red, thick] {x < 0 ? 0 : tanh(x/0.5)};
            \addlegendentry{$\delta=0.5$}
            
            \addplot[green, thick] {x < 0 ? 0 : tanh(x/0.2)};
            \addlegendentry{$\delta=0.2$}
            
            \addplot[black, thick] {x < 0 ? 0 : tanh(x/0.1)};
            \addlegendentry{$\delta=0.1$}
            
            \addplot[blue, thick] {x < 0 ? 0 : tanh(x/0.02)};
            \addlegendentry{$\delta=0.02$}
        \end{axis}
    \end{tikzpicture}
    
    \caption{Left: The Moreau-Yosida approximation of $\text{Sign}^+$. Middle: The homographic approximation $\beta^{1}_{\delta}(p) = \frac{\max(0,p)}{\max(0,p)+\delta}$. Right: The hyperbolic tangent approximation $\beta^{2}_{\delta}(p) = \max\left(0, \tanh\left(\frac{p}{\delta}\right)\right)$. Both regularizations $\beta^1_\delta$ and $\beta^2_\delta$ assign zero value for negative and null pressures.}
    \label{fig:tanh-approx}
\end{figure}

A standard choice in the literature is to use the sigmoid or the arctangent functions. However, as illustrated in \cref{fig:smooth-approx}, these regularizations satisfy $\beta_\delta(0) = 1/2$. In our model, the pressure satisfies a homogeneous Dirichlet boundary condition $p=0$ on the exit doors $\Gamma_D$. Consequently, using such functions for $\beta$ would artificially force the density to $\rho = 1/2$ at the exits, creating a non-physical boundary layer that restricts the evacuation flow. To avoid this numerical artifact and in compliance with our theoretical assumptions (\cf  \cref{assumption:1}), we must select a constitutive law satisfying $\beta(0) = 0$. 
While the exact Moreau-Yosida approximation (see \cref{fig:tanh-approx}, Left) achieves this, its lack of continuous differentiability and strict monotonicity makes it unsuitable for Newton-based solvers. Thus, we propose using the hyperbolic tangent regularization by taking $\beta(p) = \beta_{\delta}(p)$ where (see \cref{fig:tanh-approx}, Right):
\begin{equation}\label{eq:tanh}
    \beta_\delta(p) = \max\left(0,\tanh\left(\frac{p}{\delta}\right)\right).
\end{equation}
It is worth noting that even though the maximum operator in \eqref{eq:tanh} introduces a kink at $0$, we use in practice $\beta_\delta(p) = \tanh\left(\frac{p}{\delta}\right)$. This is justified by the fact that the pressure $p$, acting as a dual variable associated with the density $\rho$, satisfies $p\geq 0$. Thus, in practice, the model operates in the regime where $\beta(p) = \tanh\left(\frac{p}{\delta}\right)$. Therefore, we can restrict our theoretical and numerical analysis to $\mathbb{R}^+$. On any compact subset $[0, M] \subset \mathbb{R}^+$, the proposed function $\beta_\delta$ is smooth, correctly enforces the boundary condition $\beta_\delta(0)=0$, and maintains a strictly positive derivative, thus satisfying the required bi-Lipschitz property for our framework.

\subsubsection{Boundary condition} While Neumann boundary conditions acting on the total flux remain essentially unchanged for both hard and soft congestion regimes, Dirichlet boundary conditions become more delicate to handle in the soft congestion framework. Indeed, the prediction-correction framework intrinsically requires the introduction of an outlet boundary condition on the exit region $\Gamma_D$. In the hard-congestion setting, the condition
\[
p=0
\quad \text{on } \Gamma_D
\]
naturally drives pedestrians toward the exits while keeping the density unconstrained at the boundary, namely
\[
0\leq \rho \leq 1
\quad \text{on } \Gamma_D.
\]
Consequently, pedestrians are allowed to leave the computational domain without imposing any prescribed density at the exit. 

The soft-congestion case is more delicate. Indeed, since the pressure is explicitly related to the density through a constitutive law of the form $\rho=\beta(p)$, any boundary condition imposed on the pressure automatically induces a boundary condition on the density. This may significantly restrict the range of admissible evacuation scenarios and reduce the model's flexibility. Moreover, from a numerical standpoint, imposing outflow conditions together with an outward velocity field may generate boundary-layer effects and spurious density accumulations near the exits.

To overcome these difficulties, we introduce auxiliary fictitious cells (often referred to as ghost cells) outside the computational domain, in the vicinity of the exit region. The Dirichlet conditions are then imposed on these external cells rather than directly on $\Gamma_D$. This construction creates an artificial evacuation buffer, allowing pedestrians to progressively leave the room while avoiding undesirable boundary effects and preserving the free evolution of the density at the physical boundary.
\subsubsection{Domain discretization and primal-dual parameters}
We assume that the crowd moves in a room represented by the domain $\Omega=[0,1]^2$, discretized with a rectangular grid with a mesh size $h = 0.02$. As for the time variable, we use a timestep $\tau = 0.006$. 
To adjust the parameters in \cref{alg:pd}, we use the approximation $\Vert \Lambda_h\Vert \simeq \sqrt{1 + 8/h^2}:=\LL$ and take $\sigma = \alpha = 0.99/\LL$.
\subsection{Numerical tests}\label{section:tests}
We perform several tests across different scenarios\footnote{Demonstration videos are available at \url{https://github.com/enhamza/PC-Hughes-Model}}: different initial densities, single or multiple exits. We also provide tests demonstrating the influence of the velocity field model \eqref{eq:V-model} on crowd dynamics by considering different values of $\lambda$. 

\subsubsection{Evacuation of a room with one exit.}
We begin by comparing the behavior of the models in a standard evacuation setup. The crowd is evacuated from the square room $\Omega=[0,1]^2$ through a single exit located on the right boundary, defined by
\begin{equation}
    \Gamma_D = \{1\} \times [0.4, 0.6].
\end{equation}
The maximal evacuation time is set to $T=2$.
We consider two distinct initial density profiles to test the models' ability to handle both discontinuous (patch-like) and smooth distributions.

\begin{itemize}
    \item \textbf{Scenario 1: Two Groups.}
    The initial density represents two distinct groups of agents positioned at the top-left and bottom-left corners. It is defined by a scaled indicator function \begin{equation}\label{eq:rho0_scenario1}
    	\rho_{0}(x,y) = 0.9 \cdot \chi_{S}(x,y),
    	   \end{equation}
    where $S$ is the union of two rectangular regions 
    \[
    S = \left([0, 0.5] \times [0, 1/3]\right) \cup \left([0, 0.5] \times [2/3, 1]\right).
    \]
    \item \textbf{Scenario 2: Gaussian Mixture.}
    The initial density is given by a sum of three unnormalized Gaussian functions centered at $\mathbf{c}_1=(0.2, 0.2)$, $\mathbf{c}_2=(0.2, 0.8)$, and $\mathbf{c}_3=(0.75, 0.5)$: 
        \begin{equation}\label{eq:rho0_scenario2}
        \rho_0(x,y) = \sum_{i=1}^3 \exp\left(-\frac{\Vert(x,y)-\mathbf{c}_i\Vert^2}{0.02}\right).
    \end{equation}
    This configuration allows us to observe the merging dynamics of sub-groups with smooth interfaces.
\end{itemize}
For the cost function $\he(\rho) = e^{\lambda \rho}$, we fix the sensitivity parameter to $\lambda = 2.75$.
 \begin{figure}[H]
	\centering
	\includegraphics[width=0.8\columnwidth]{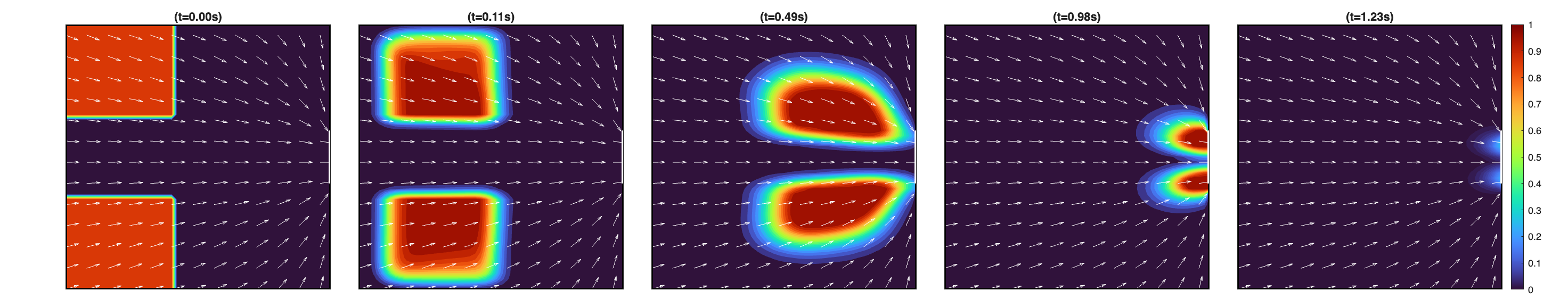}\\
	\includegraphics[width=0.8\columnwidth]{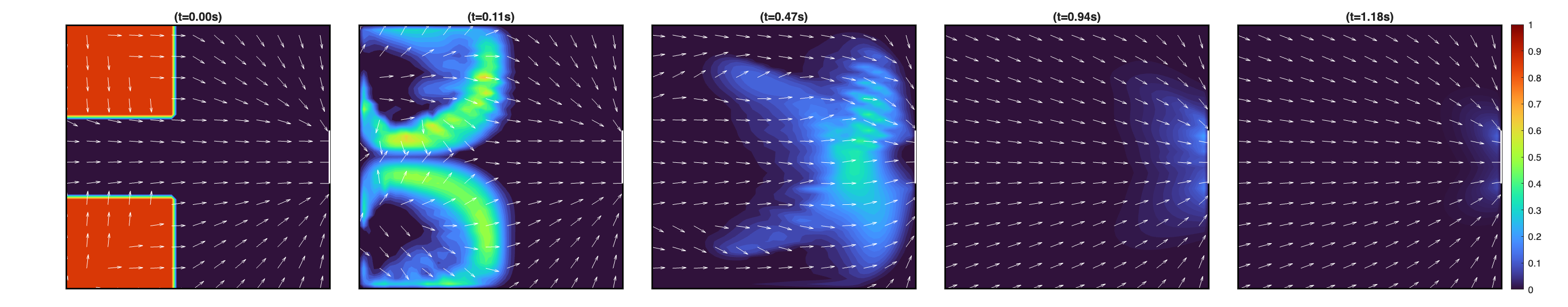}\\
	\includegraphics[width=0.8\columnwidth]{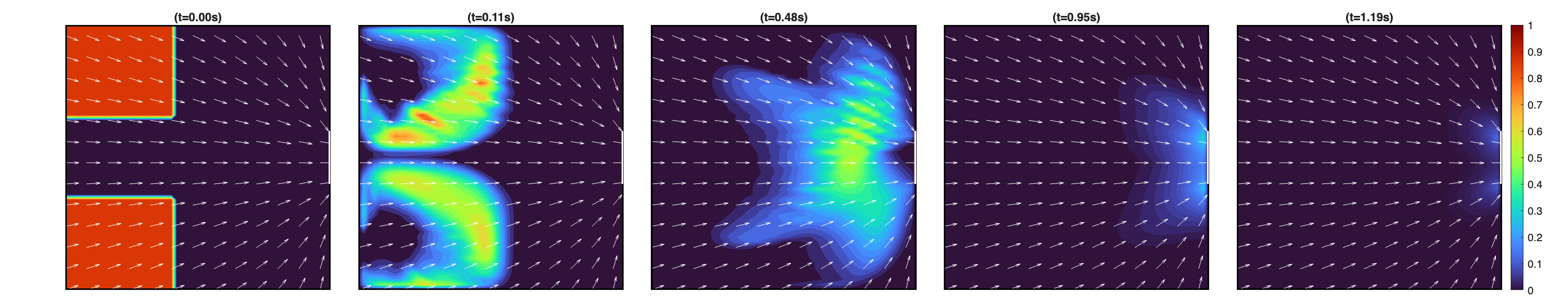}\\
\caption{Time evolution of the crowd density $\rho$ computed at different instants. 
\textbf{First row:} The constant velocity  model governed by \eqref{evolgran0}. 
\textbf{Second row:} The hard congestion Hughes' model governed by \eqref{evolution223}.
\textbf{Third row:} The soft congestion Hughes' Model governed by \eqref{eq:regularized-model}.}
\label{fig:comparison_models}
	\label{fig:eg1}
\end{figure}
 \begin{figure}[H]
	\centering
	\includegraphics[width=0.8\columnwidth]{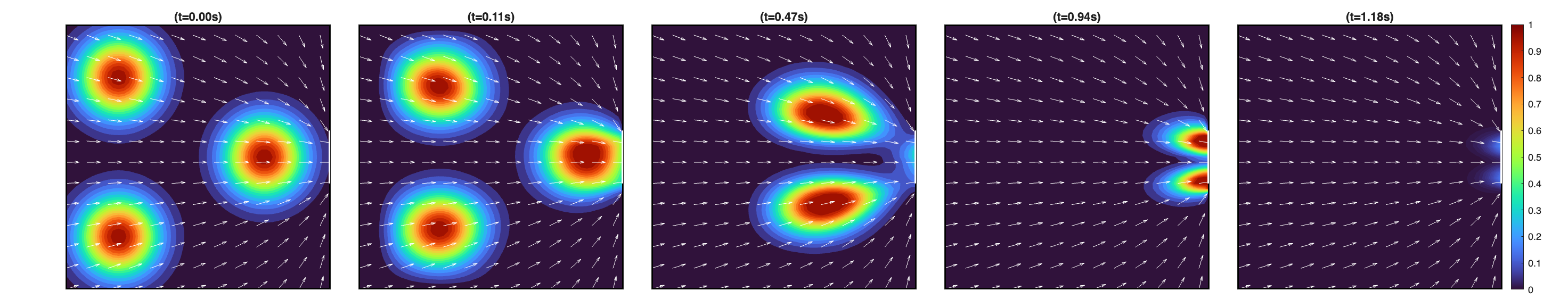}\\
	\includegraphics[width=0.8\columnwidth]{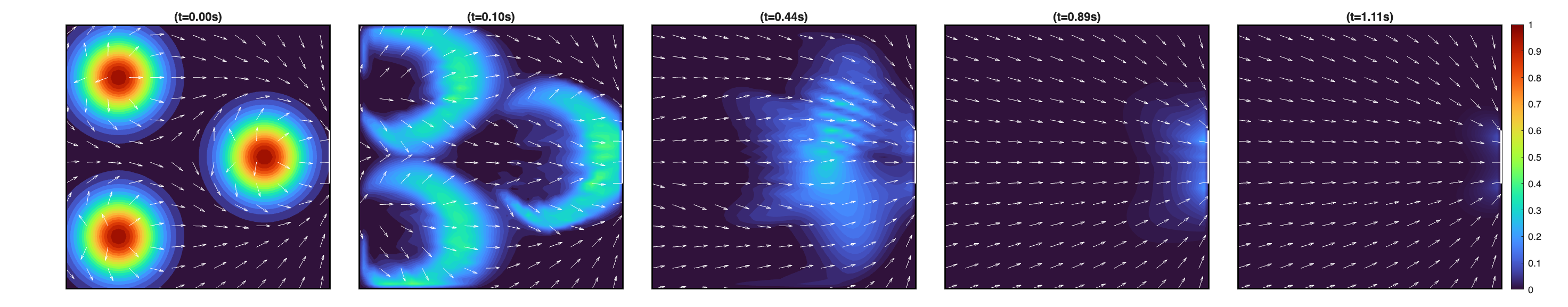}\\
	\includegraphics[width=0.8\columnwidth]{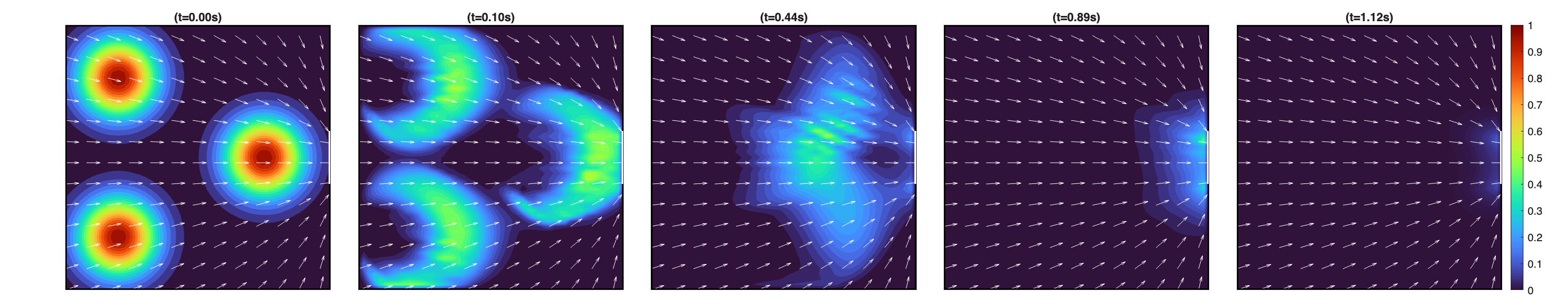}\\
	\caption{Time evolution of the crowd density $\rho$ computed at different instants. 
\textbf{First row:} The constant velocity  model governed by \eqref{evolgran0}. 
\textbf{Second row:} The hard congestion Hughes' model governed by \eqref{evolution223}.
\textbf{Third row:} The soft congestion Hughes' Model governed by \eqref{eq:regularized-model}.}
		\label{fig:eg1_gaussians}
\end{figure}
We observe that in the first row of both \cref{fig:eg1,fig:eg1_gaussians}, the model \eqref{evolgran0} transports the population in a rigid way. Pedestrians remain highly concentrated and move straight towards the exit without exploiting the available empty space in the room.

The second and third rows illustrate the dynamics of the proposed models \eqref{evolution223} and \eqref{eq:regularized-model}. The crowd diffuses into the low-density zones to avoid congestion. For instance, we observe in \cref{fig:eg1,fig:eg1_gaussians} that the initial density is smoothly deformed and spreads, significantly reducing the maximum local density early in the evacuation process. Furthermore, we observe that the soft congestion model \eqref{eq:regularized-model} yields density profiles that are virtually identical to those of the hard congestion framework \eqref{evolution223}. This confirms that our mathematical regularization provides a highly accurate approximation of the singular dynamics.

\subsubsection{Evacuation of a room with two exits.} 
In this second test, we examine a more complex configuration, in which the crowd is evacuated through two distinct exits located at the bottom and top boundaries of the domain. The exit set $\Gamma_D = D_1 \cup D_2$ is defined by
\begin{equation}
    D_1 = [0.4, 0.6] \times \{0\} \quad \text{and} \quad D_2 = [0.6, 0.75] \times \{1\}.
\end{equation}
The maximal evacuation time is maintained at $T=2$.
Similar to the previous test, we consider two initial density profiles to evaluate the impact of the geometry on both compact and smooth distributions.
\begin{itemize}
    \item \textbf{Scenario 1: Circular Cluster.}
    The initial density is concentrated in a single circular group located slightly to the left of the domain center. It is defined by the scaled characteristic function
    \begin{equation}
        \rho_{0}(x,y) = 0.9 \cdot \chi_{S}(x,y),
    \end{equation}
    where $S = B(\mathbf{c}, r)$ is the Euclidean ball centered at $\mathbf{c}=(0.3, 0.5)$ with radius $r=0.25$.

    \item \textbf{Scenario 2: Gaussian Mixture.}
    We use the same smooth distribution as in the single-exit case given by \eqref{eq:rho0_scenario2}. This configuration tests the models' ability to partition the crowd effectively between the two available exits.
\end{itemize}
The sensitivity parameter $\lambda$ remains unchanged from the previous simulation.
   \begin{figure}[H]
	\centering
	\includegraphics[width=0.8\columnwidth]{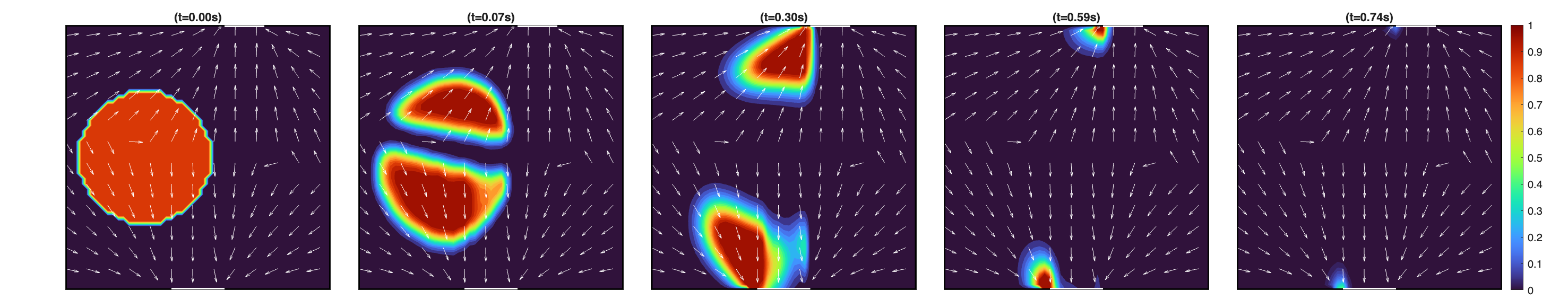}\\
	\includegraphics[width=0.8\columnwidth]{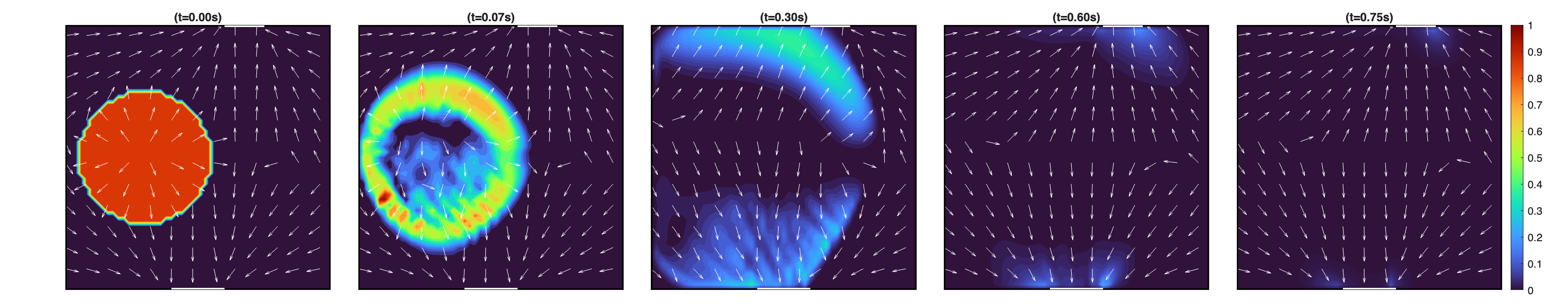}\\
		\includegraphics[width=0.8\columnwidth]{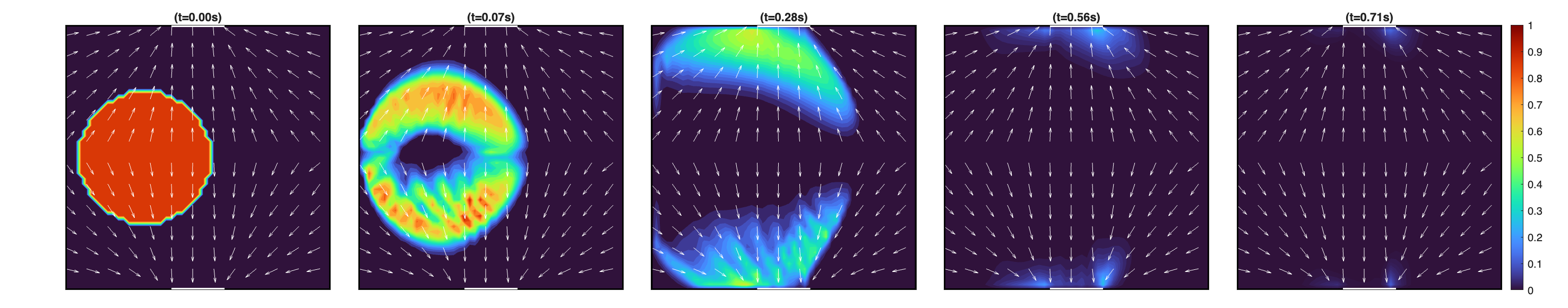}\\

\caption{Time evolution of the crowd density $\rho$ computed at different instants. 
\textbf{First row:} The constant velocity  model governed by \eqref{evolgran0}. 
\textbf{Second row:} The hard congestion Hughes' model governed by \eqref{evolution223}.
\textbf{Third row:} The soft congestion Hughes' Model governed by \eqref{eq:regularized-model}.}
\label{fig:eg2}
\end{figure}
 \begin{figure}[H]
	\centering
	\includegraphics[width=0.8\columnwidth]{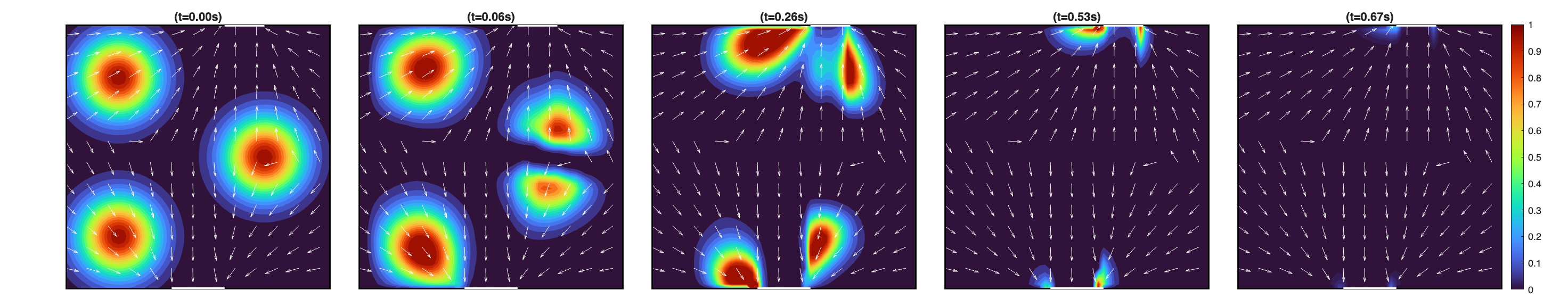}\\
	\includegraphics[width=0.8\columnwidth]{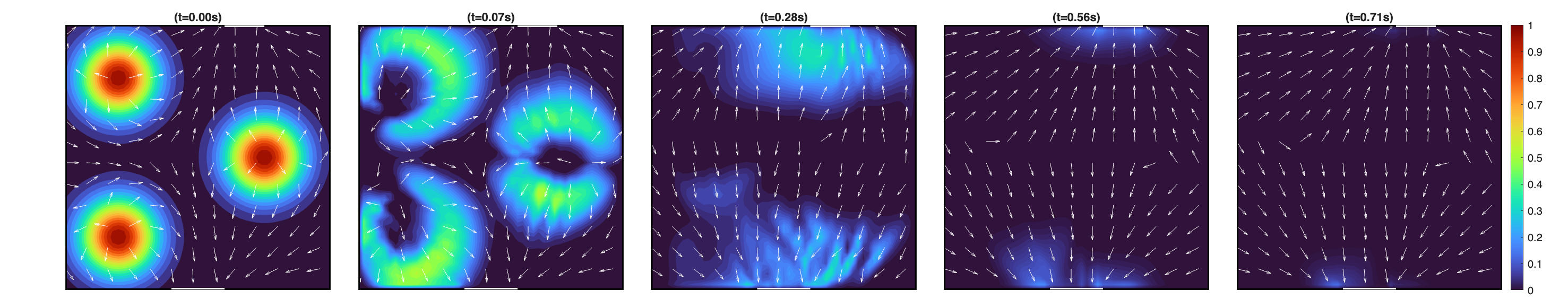}\\
		\includegraphics[width=0.8\columnwidth]{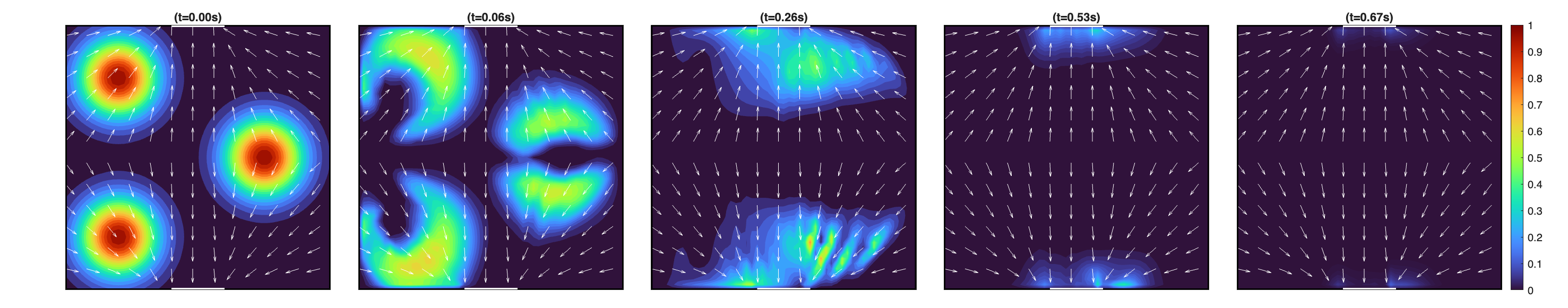}\\
\caption{Time evolution of the crowd density $\rho$ computed at different instants. 
\textbf{First row:} The constant velocity  model governed by \eqref{evolgran0}. 
\textbf{Second row:} The hard congestion Hughes' model governed by \eqref{evolution223}.
\textbf{Third row:} The soft congestion Hughes' Model governed by \eqref{eq:regularized-model}.}
\label{fig:eg2_gaussians}
\end{figure}
As observed in the first row of \cref{fig:eg2,fig:eg2_gaussians}, the constant velocity model \eqref{evolgran0} simply splits the population into two rigid groups moving directly towards the nearest exit. Since the velocity field does not take into account the local pressure, the created clusters remain highly dense and compact throughout the evacuation process.

In contrast, the second and third rows highlight the interplay between path optimization and congestion avoidance in models \eqref{eq:regularized-model} and \eqref{evolution223}. We observe that the crowd splits to utilize both exits while exhibiting a diffusive behavior. For instance, in \cref{fig:eg2}, the initially circular density is hollowed out from the center, expanding into a wide crescent-like front that lowers the maximal density. Similarly, in \cref{fig:eg2_gaussians}, the three initial groups expand and smoothly merge before bifurcating towards the exits. Once again, the evolution of the soft congestion model \eqref{eq:regularized-model} is visually indistinguishable from the singular hard congestion framework \eqref{evolution223}, further validating our theoretical approximation. 

 \subsubsection{Evacuation of a room with multiple exits.} 
In this final test, we investigate the evacuation dynamics in a room equipped with three exit doors distributed along the left and right boundaries. The exit set $\Gamma_D = D_1 \cup D_2 \cup D_3$ is defined by
\begin{equation}
    D_1 = \{0\} \times [0.15, 0.2], \quad D_2 = \{0\} \times [0.6, 0.8], \quad \text{and} \quad D_3 = \{1\} \times [0.4, 0.55].
\end{equation}
The maximal evacuation time is set to $T=2$.
We introduce two specific initial configurations to test the models against symmetric and periodic structures.
\begin{itemize}
    \item \textbf{Scenario 1: The Annulus.}
    The crowd is initially distributed in a ring shape centered at the middle of the room. The density is defined by
    \begin{equation}\label{eq:rho_annulus}
        \rho_{0}(x,y) = 0.9 \cdot \chi_{\mathcal{C}}(x,y),
    \end{equation}
    where the support $\mathcal{C}$ is an annulus centered at $\mathbf{c}=(0.5, 0.5)$ with inner radius $r_{\text{in}} = 0.15$ and outer radius $r_{\text{out}} = 0.35$:
    \[
        \mathcal{C} = \left\{ (x,y) \in \Omega : r_{\text{in}} < \Vert(x,y) - \mathbf{c}\Vert_2 < r_{\text{out}}\right\}.
    \]

    \item \textbf{Scenario 2: The Checkerboard.}
    The initial density follows a periodic pattern of alternating empty and occupied squares of size $1/8 \times 1/8$. It is given by
    \begin{equation}
        \rho_0(x,y) = 0.9 \cdot \chi_{\mathcal{C}}(x,y),
    \end{equation}
    where the domain $\mathcal{C}$ is defined by
    \[
        \mathcal{C} = \left\{ (x,y) \in \Omega : \sin(8\pi x)\sin(8\pi y) > 0 \right\}.
    \]
\end{itemize}
   \begin{figure}[H]
	\centering
	\includegraphics[width=0.8\columnwidth]{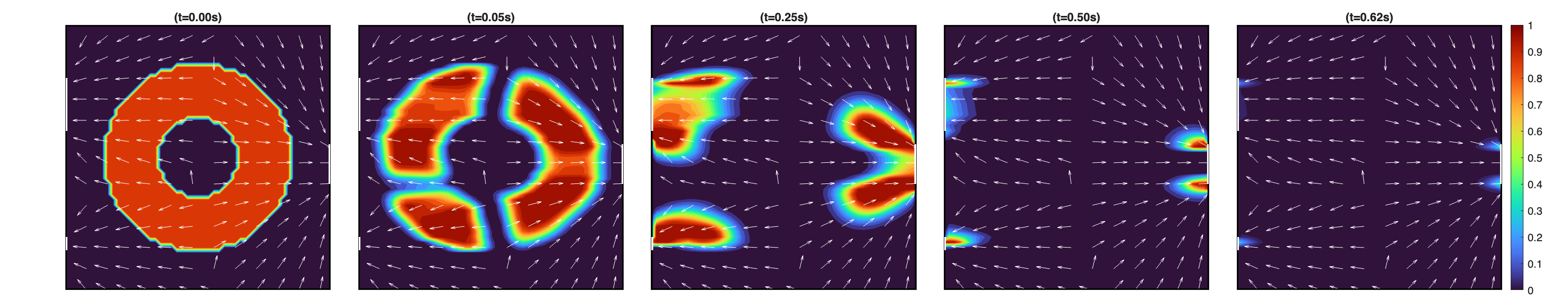}\\
	\includegraphics[width=0.8\columnwidth]{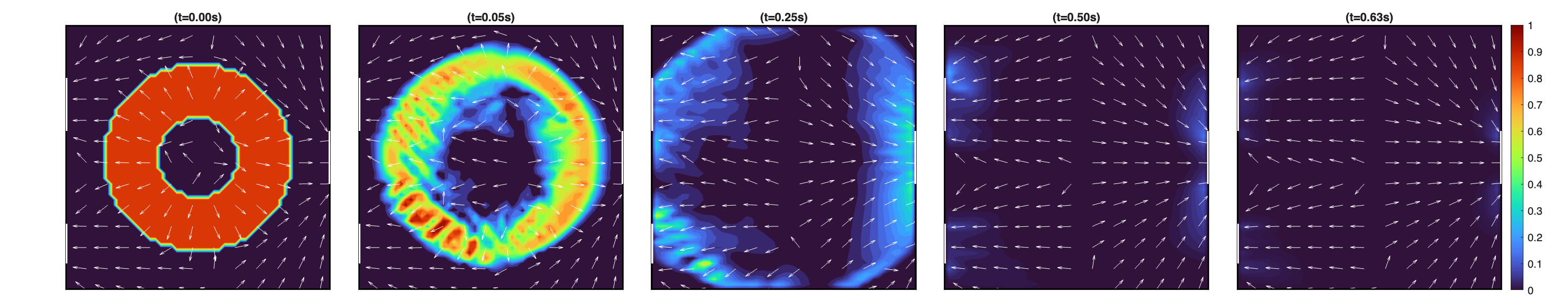}\\
		\includegraphics[width=0.8\columnwidth]{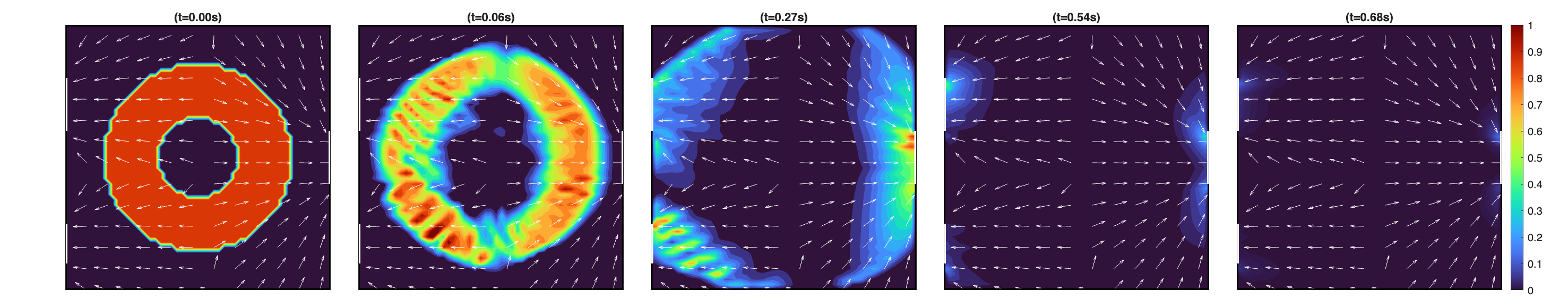}
\caption{Time evolution of the crowd density $\rho$ computed at different instants. 
\textbf{First row:} The constant velocity  model governed by \eqref{evolgran0}. 
\textbf{Second row:} The hard congestion Hughes' model governed by \eqref{evolution223}.
\textbf{Third row:} The soft congestion Hughes' Model governed by \eqref{eq:regularized-model}.}
\label{fig:eg3_ring}
\end{figure}
   \begin{figure}[H]
	\centering
	\includegraphics[width=0.8\columnwidth]{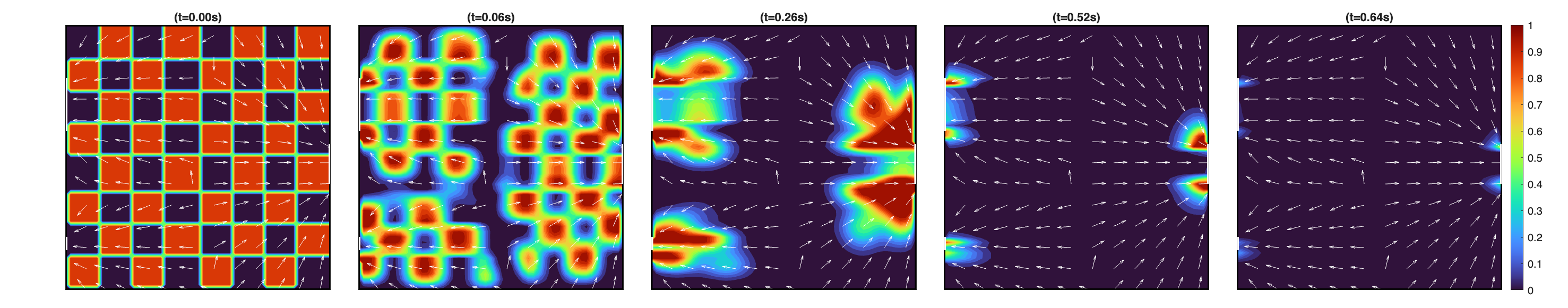}\\
	\includegraphics[width=0.8\columnwidth]{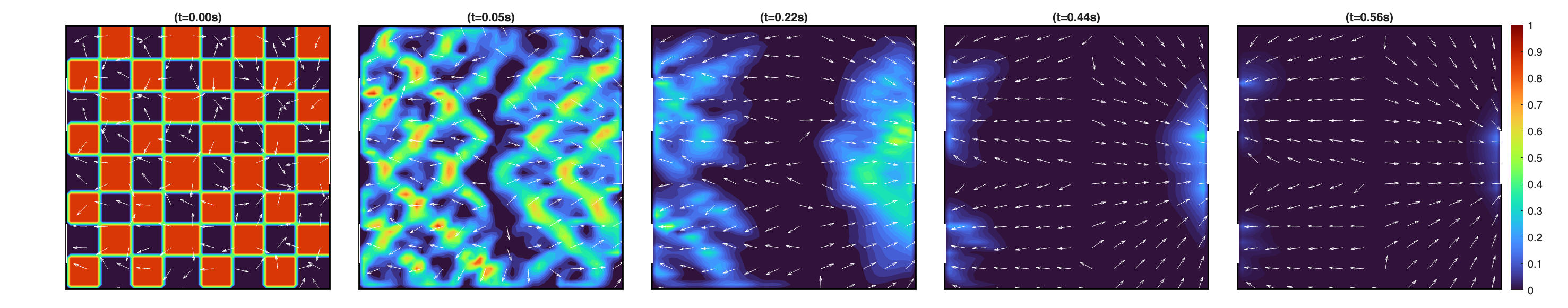}\\
		\includegraphics[width=0.8\columnwidth]{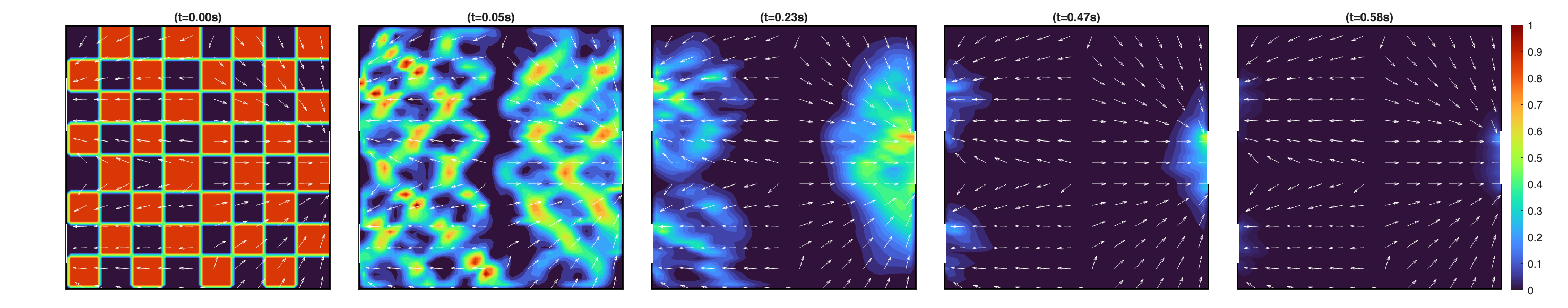}\\
\caption{Time evolution of the crowd density $\rho$ computed at different instants. 
\textbf{First row:} The constant velocity  model governed by \eqref{evolgran0}. 
\textbf{Second row:} The hard congestion Hughes' model governed by \eqref{evolution223}.
\textbf{Third row:} The soft congestion Hughes' Model governed by \eqref{eq:regularized-model}.}
\label{fig:eg3_checkerboard}
\end{figure}

The first row of both \cref{fig:eg3_checkerboard,fig:eg3_ring} shows again that the constant velocity model \eqref{evolgran0} moves the initial shapes in a completely rigid manner. Since pedestrians do not anticipate congestion, the created groups simply slide towards the nearest exits. This lack of diffusion creates high-density clusters that overlap near the boundaries.

In contrast, the second and third rows demonstrate the robust, fluid-like behavior of the models \eqref{eq:regularized-model} and \eqref{evolution223}. In particular, we see how they handle the complex initial geometries: the dense ring smoothly expands and divides, while the disjointed checkerboard squares rapidly diffuse and merge into a cohesive flow. This pressure-driven expansion efficiently dissipates the initial congestion, allowing the crowd to optimally organize and bifurcate towards the three available exits. Once again, the soft congestion model \eqref{eq:regularized-model} visually replicates the dynamics of the hard congestion model \eqref{evolution223}. 

\begin{table}[!ht]
    \centering
    \caption{Evacuation times (in seconds) for different mathematical models and initial density configurations.}
    \label{tab:models_evacuation}
    \begin{tabular}{|l|ccc|}
        \hline
        \textbf{Initial Density} & 
        \textbf{\eqref{evolgran0}} & 
        \textbf{\eqref{evolution223} } & 
        \textbf{\eqref{eq:regularized-model}} \\
        \hline
        \hline
        
        Two Blocks & 
        \texttt{1.23s} & 
        \texttt{1.176s} & 
        \texttt{1.194s} \\
        \hline
        
        Gaussians & 
        \texttt{1.176s} & 
        \texttt{1.110s} & 
        \texttt{1.116s} \\
        \hline

        Disc & 
        \texttt{0.744s} & 
        \texttt{0.75s} & 
        \texttt{0.696s} \\
        \hline

        Annulus & 
        \texttt{0.624s} & 
        \texttt{0.630s} & 
        \texttt{0.678s} \\
        \hline

        Checkerboard & 
        \texttt{0.642s} & 
        \texttt{0.558s} & 
        \texttt{0.588s} \\
        \hline
    \end{tabular}

\end{table}
\subsubsection{Impact of the velocity field } \label{subsection:impact_H}
In what follows, we provide several examples to demonstrate the effect of the velocity field choice in the soft congestion model \eqref{eq:regularized-model} and its hard congestion counterpart \eqref{evolution223}. Specifically, we analyze the sensitivity of the dynamics to the congestion parameter $\lambda$ appearing in the exponential cost function:
\begin{equation}
    \he(\rho) = \exp(\lambda\rho), \quad \text{with } \lambda \in \{1.5, \, 3.75, \, 7.25\}.
\end{equation}
Higher values of $\lambda$ correspond to a stronger repulsion effect from high-density regions, leading to more pronounced detour behaviors. To illustrate this, we consider two specific geometrical configurations.

\begin{itemize}
    \item \textbf{Example 1: The C-Shape.}
    We consider a single-exit scenario in which the initial density is distributed along a C shape, forcing the agents to circumvent the empty center. The exit is located at $\Gamma_D = \{1\} \times [0.4, 0.6]$. The initial density is defined by
    \begin{equation}
        \rho_0(x,y) = 0.95 \cdot \chi_{\mathcal{C}}(x,y),
    \end{equation}
    where the support $\mathcal{C}$ is the intersection of an annulus and a domain excluding the right-side opening:
    \[
        \mathcal{C} = \left\{ (x,y)\in \Omega : 0.2 < \Vert(x,y)-\mathbf{c}\Vert_2 < 0.4 \right\} \cap \left( \{x < 0.5\} \cup \{\vert y-0.5\vert > 0.1\} \right),
    \]
    with $\mathbf{c} = (0.5, 0.5)$.

    \item \textbf{Example 2: The Cross-Shape.}
    In this multipoint evacuation test, the room is equipped with four exits situated on the vertical boundaries
    \[
        D_1 = \{0\}\times [0.1, 0.25], \quad D_2 = \{0\}\times [0.6, 0.8],
    \]
    \[
        D_3 = \{1\}\times [0.2, 0.4], \quad D_4 = \{1\}\times [0.78, 0.9].
    \]
    The initial density forms a cross centered in the domain:
    \begin{equation}
        \rho_0(x,y) = 0.95 \cdot \chi_{\mathcal{C}}(x,y),
    \end{equation}
    where $\mathcal{C}$ is the union of a horizontal and a vertical strip of width $0.2$,
    \[
        \mathcal{C} = \left\{ \vert x-0.5\vert < 0.1 \right\} \cup \left\{ \vert y-0.5\vert < 0.1 \right\}.
    \]
\end{itemize}
 \begin{figure}[H]
	\centering
	\includegraphics[width=0.8\columnwidth]{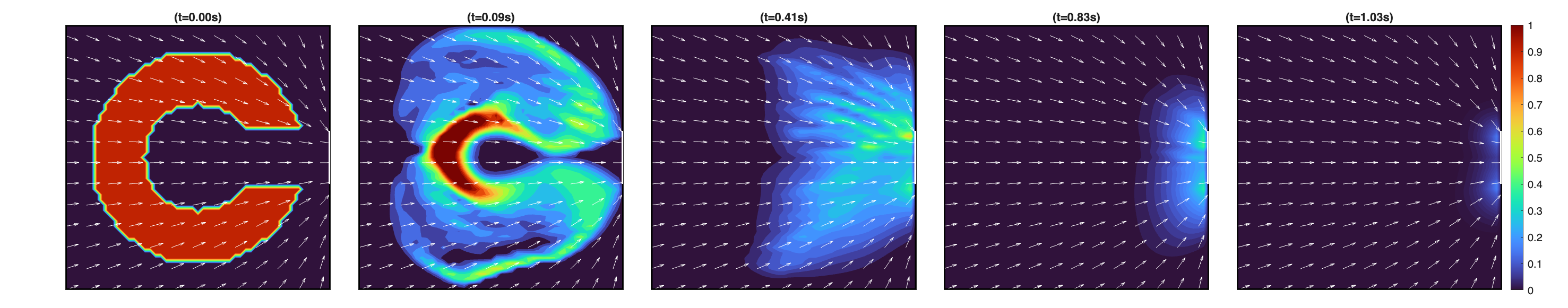}\\
	\includegraphics[width=0.8\columnwidth]{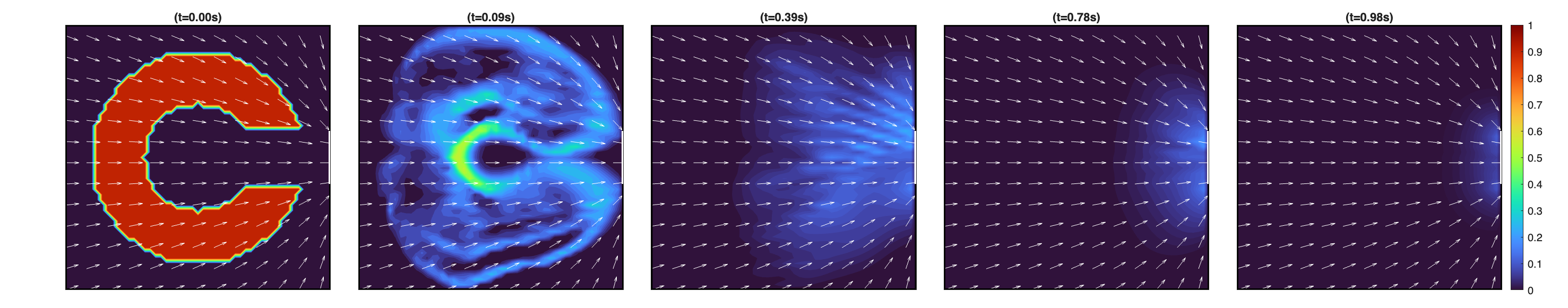}\\
	\includegraphics[width=0.8\columnwidth]{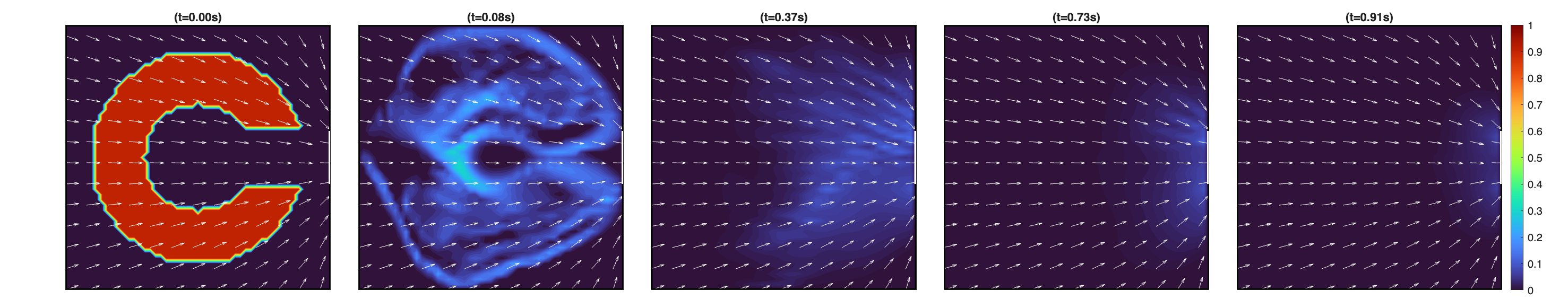}
\caption{The crowd density $\rho$ computed by the model \eqref{eq:regularized-model} at different time steps. Top row: $\lambda = 1.5$. Middle row: $\lambda = 3.75$. Bottom row: $\lambda = 7.25$.}
	\label{fig:eg1-reg-lambda}
\end{figure}
 \begin{figure}[h!]
	\centering
	\includegraphics[width=0.8\columnwidth]{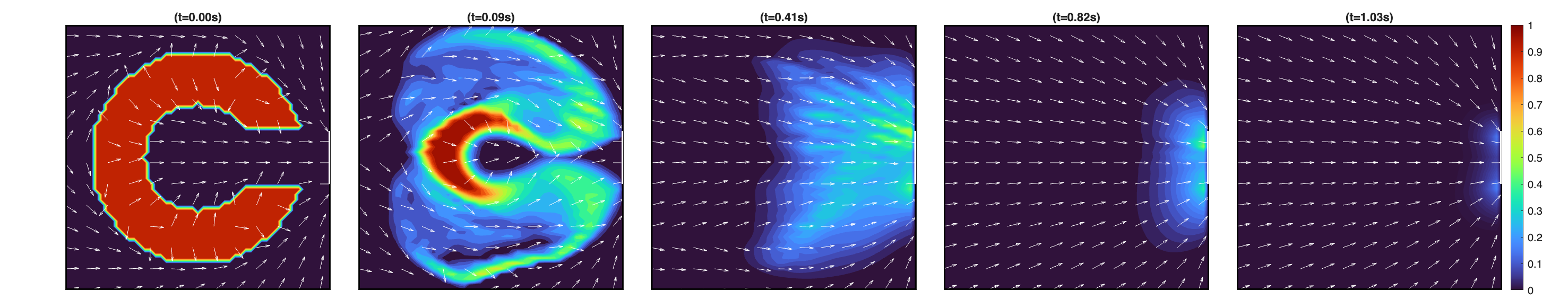}\\
	\includegraphics[width=0.8\columnwidth]{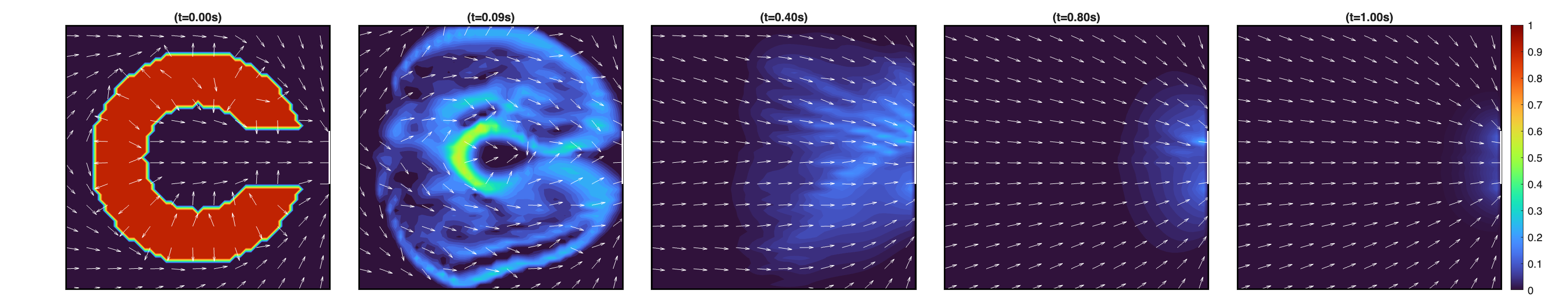}\\
	\includegraphics[width=0.8\columnwidth]{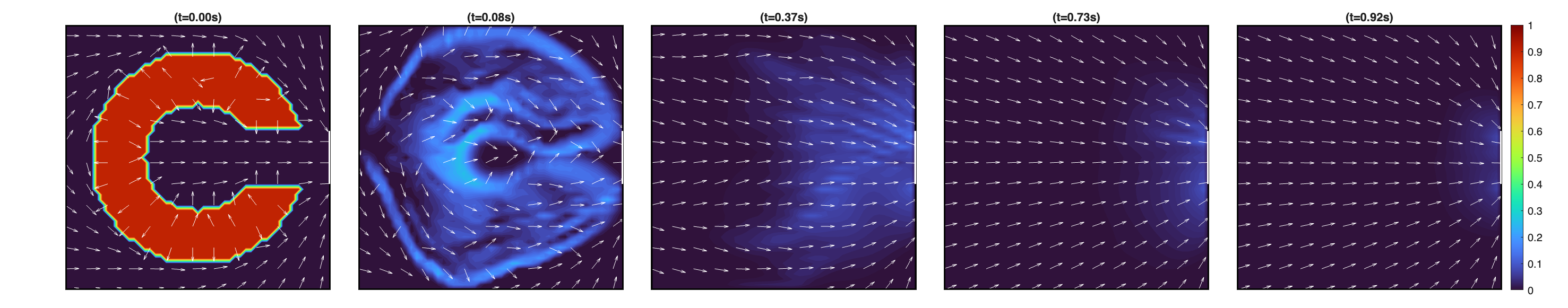}
\caption{The crowd density $\rho$ computed by the model \eqref{evolution223} at different time steps. Top row: $\lambda = 1.5$. Middle row: $\lambda = 3.75$. Bottom row: $\lambda = 7.25$.}
	\label{fig:eg1-lambda}
\end{figure}

 \begin{figure}[H]
	\centering
	\includegraphics[width=0.8\columnwidth]{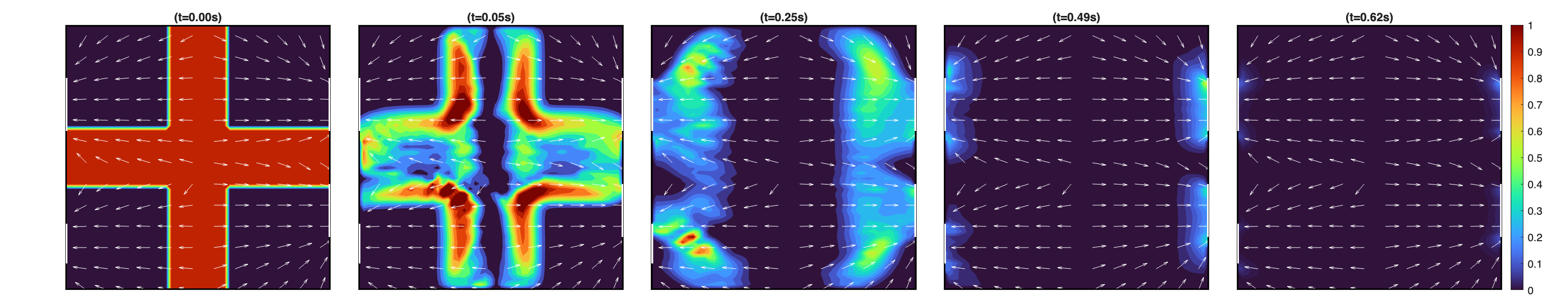}\\
	\includegraphics[width=0.8\columnwidth]{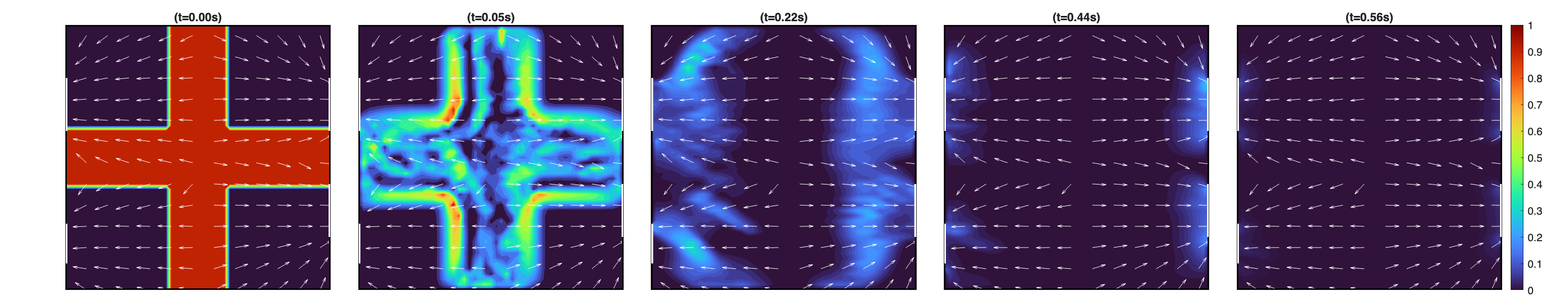}\\
	\includegraphics[width=0.8\columnwidth]{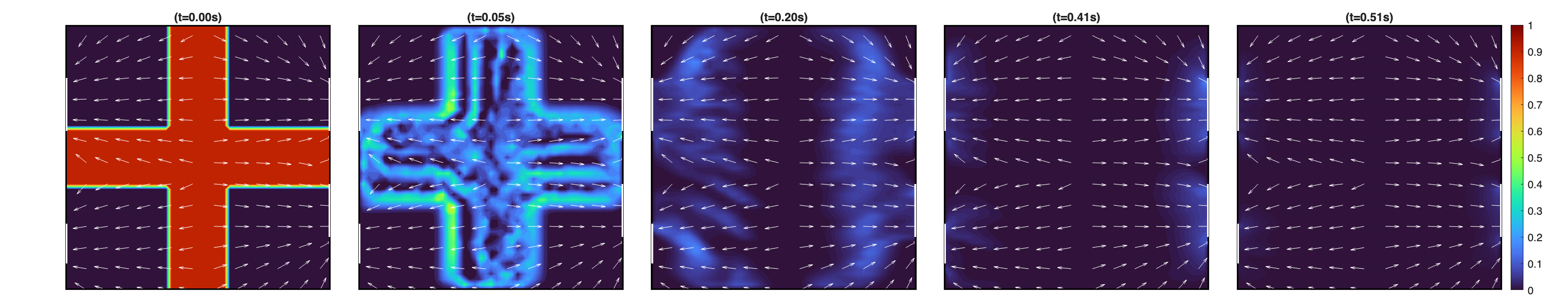}
\caption{The crowd density $\rho$ computed by the model \eqref{eq:regularized-model} at different time steps. Top row: $\lambda = 1.5$. Middle row: $\lambda = 3.75$. Bottom row: $\lambda = 7.25$.}
	\label{fig:eg2-reg-lambda}
\end{figure}
 \begin{figure}[H]
	\centering
	\includegraphics[width=0.8\columnwidth]{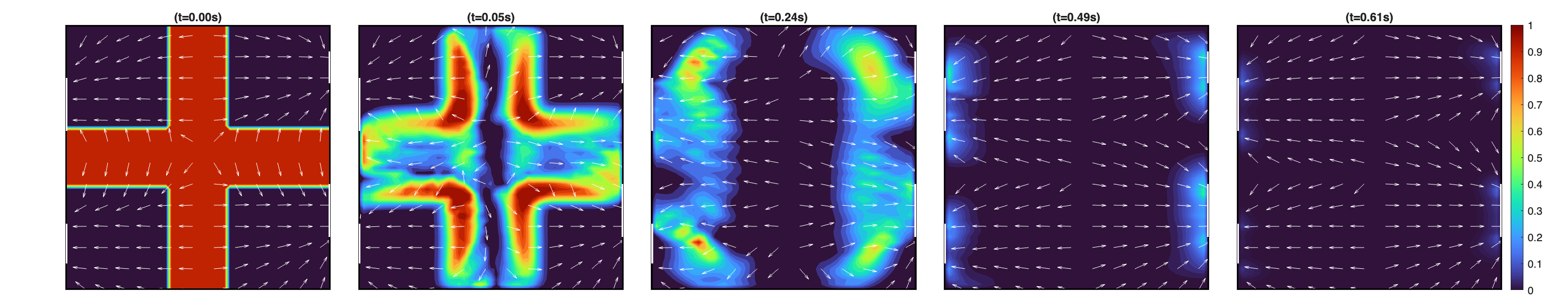}\\
	\includegraphics[width=0.8\columnwidth]{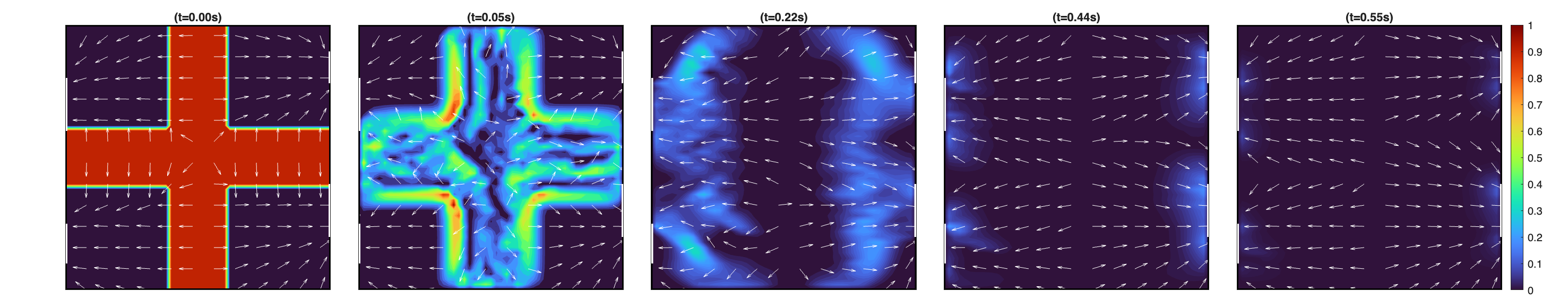}\\
	\includegraphics[width=0.8\columnwidth]{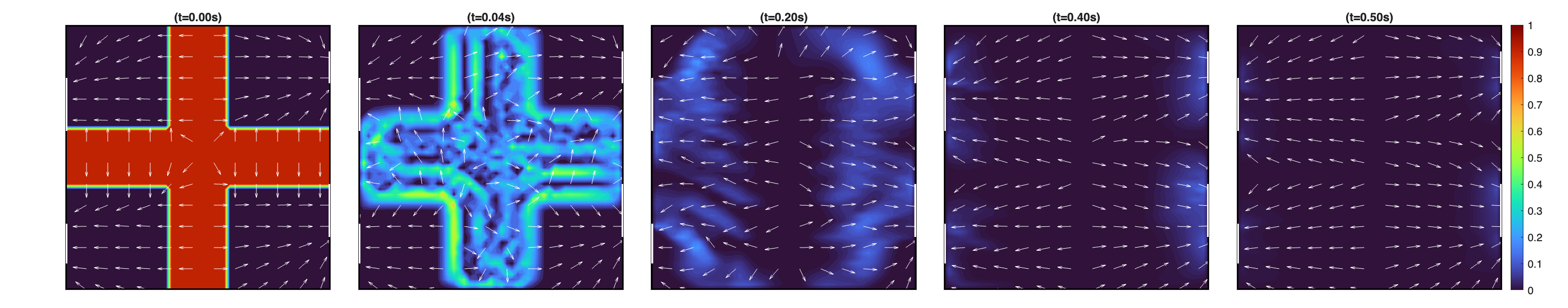}
\caption{The crowd density $\rho$ computed by the model \eqref{evolution223} at different time steps. Top row: $\lambda = 1.5$. Middle row: $\lambda = 3.75$. Bottom row: $\lambda = 7.25$.}
	\label{fig:eg2-lambda}
\end{figure}
\begin{table}[H]
    \centering
    \caption{Evacuation times for \eqref{eq:regularized-model} for different values of the congestion parameter $\lambda$.}
    \label{tab:reg_model_lambda}
    \begin{tabular}{|l|ccc|}
        \hline
        \textbf{Congestion parameter} & 
        \textbf{$\lambda = 1.5$} & 
        \textbf{$\lambda = 3.75$} & 
        \textbf{$\lambda = 7.25$}\\
        \hline
        \hline
        
        C-Shape & 
        \texttt{1.032s} & 
        \texttt{0.978s} & 
        \texttt{0.912s}\\
        \hline
        
        Cross & 
        \texttt{0.618s} & 
        \texttt{0.558s} & 
        \texttt{0.510s}\\
        \hline

    \end{tabular}
\end{table}
\begin{table}[H]
    \centering
    \caption{Evacuation times for \eqref{evolution223} for different values of the congestion parameter $\lambda$.}
    \label{tab:model_lambda}
    \begin{tabular}{|l|ccc|}
        \hline
        \textbf{Congestion parameter} & 
        \textbf{$\lambda = 1.5$} & 
        \textbf{$\lambda = 3.75$} & 
        \textbf{$\lambda = 7.25$}\\
        \hline
        \hline
        
        C-Shape & 
        \texttt{1.026s} & 
        \texttt{0.996s} & 
        \texttt{0.918s}\\
        \hline
        
        Cross & 
        \texttt{0.606s} & 
        \texttt{0.552s} & 
        \texttt{0.498s}\\
        \hline

    \end{tabular}
\end{table}
As seen in \cref{fig:eg1-lambda,fig:eg2-lambda,fig:eg1-reg-lambda,fig:eg2-reg-lambda}, the parameter $\lambda$ plays a crucial role in spreading the crowd and preventing severe congestion in saturated areas. For smaller values ($\lambda = 1.5$, top rows), the population remains relatively compact. This results in pronounced, high-density regions (highlighted in red) as pedestrians make their way toward the exit.

Conversely, as $\lambda$ increases (middle and bottom rows), spatial dispersion becomes much more noticeable. The crowd diffuses rapidly into the available empty spaces. Specifically, we can see the initial "C" shape expanding and the empty quadrants of the "Cross" pattern quickly filling up. This natural expansion actively prevents the formation of dense clusters, leading to a highly fluid and cooperative movement.

Naturally, this enhanced spatial distribution has a direct impact on the overall evacuation efficiency. As shown in \cref{tab:model_lambda,tab:reg_model_lambda}, larger values of $\lambda$ noticeably reduce the total evacuation time.
\subsection{Comparison of cost functions}\label{subsection:comp_H}
To further investigate the effect of the cost $\he$ on the crowd dynamics, we perform a comparison using the one-exit scenario. We use  both models \eqref{eq:regularized-model} and \eqref{evolution223} with three cost functions
\begin{equation}
	\he_1(\rho) = \exp(\lambda \rho),\quad \he_2(\rho) = \frac{1}{v(\rho)},\quad\mbox{and}\quad\he_{3}(\rho) = \frac{1}{f^{\delta}(\rho)},
\end{equation}
where $v(\rho) = 1-\rho$ is the classical Hughes' model, and $v^{\delta} = 1- \exp\left(-c\frac{1-\rho}{\max(\delta,\rho)}\right)$ is an exponential barrier model for a given $c>0$. We fix the sensitivity parameter at $\lambda = 1.2$,  the safety truncation factor $\delta = 10^{-3}$ and set $c=1$.

 The initial configuration is given by three blocks of densities:
\begin{equation}\label{eq:init-asym}
    \rho_0(x,y) = 0.95 \chi_{\mathcal{C}_1}(x,y) + 0.5 \chi_{\mathcal{C}_2}(x,y) + 0.75 \chi_{\mathcal{C}_3}(x,y),
\end{equation}
where $\mathcal{C}_1 = [0.1, 0.35] \times [0.6, 0.8]$, $\mathcal{C}_2 = [0.2, 0.6] \times [0.15, 0.35]$, and $\mathcal{C}_3 = [0.66, 0.85] \times [0.6, 0.8]$.
  \begin{figure}[H]
	\centering
	\includegraphics[width=0.8\columnwidth]{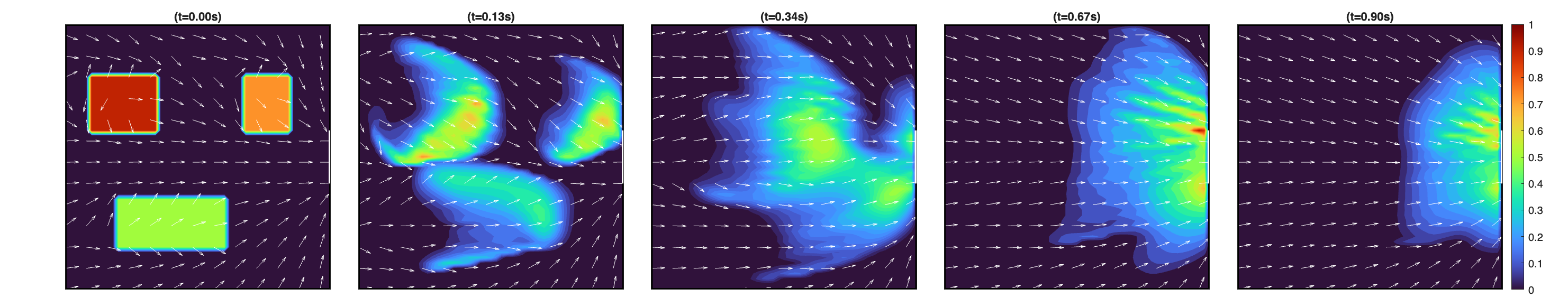}\\
	\includegraphics[width=0.8\columnwidth]{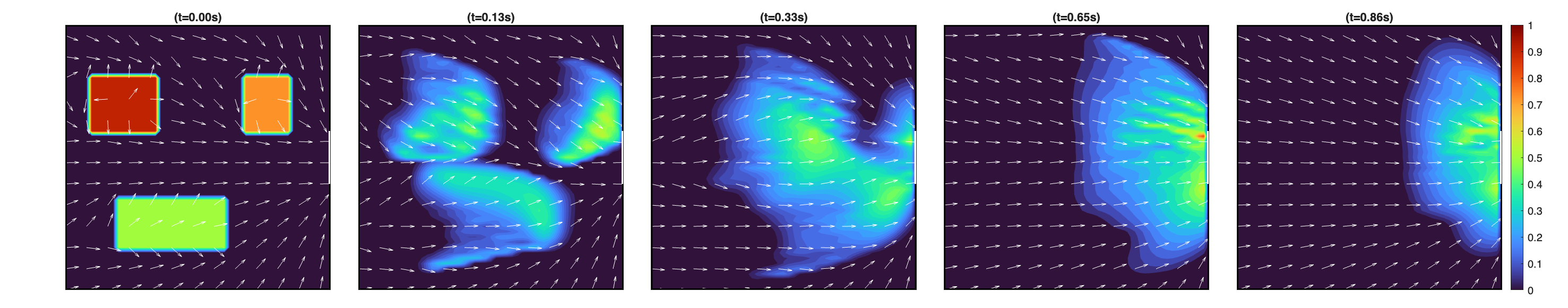}\\
	\includegraphics[width=0.8\columnwidth]{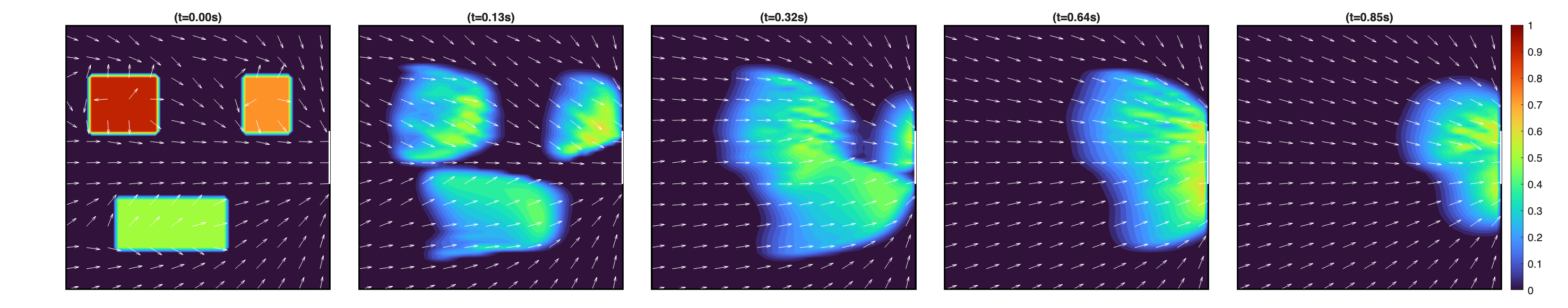}\\
\caption{The crowd density $\rho$ computed by the \eqref{eq:regularized-model} model at different time steps. Top row: $\he_1$. Middle row: $\he_2$. Bottom row: $\he_3$.}
	\label{fig:eg_H_SC}
\end{figure}
 \begin{figure}[H]
	\centering
	\includegraphics[width=0.8\columnwidth]{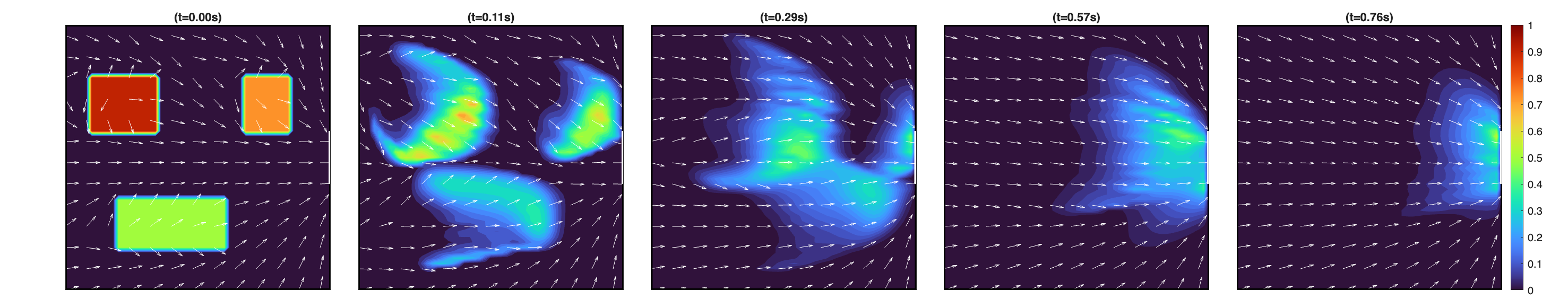}\\
	\includegraphics[width=0.8\columnwidth]{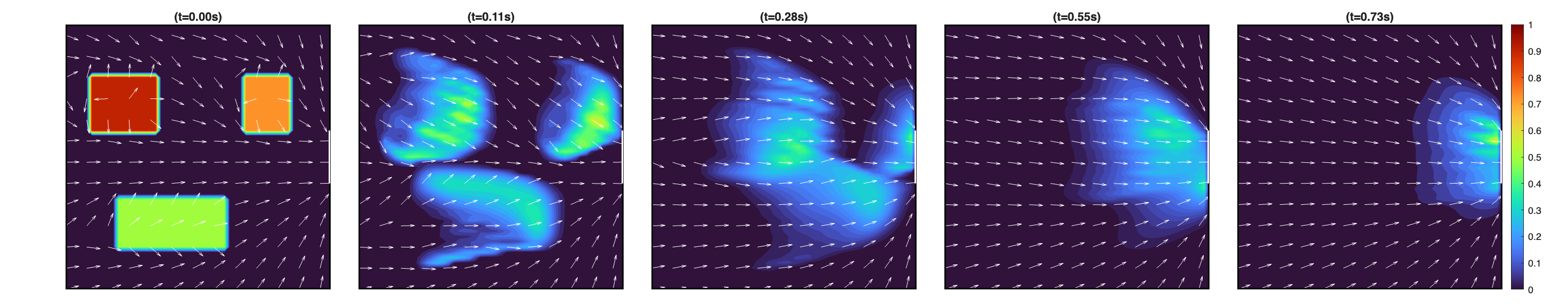}\\
	\includegraphics[width=0.8\columnwidth]{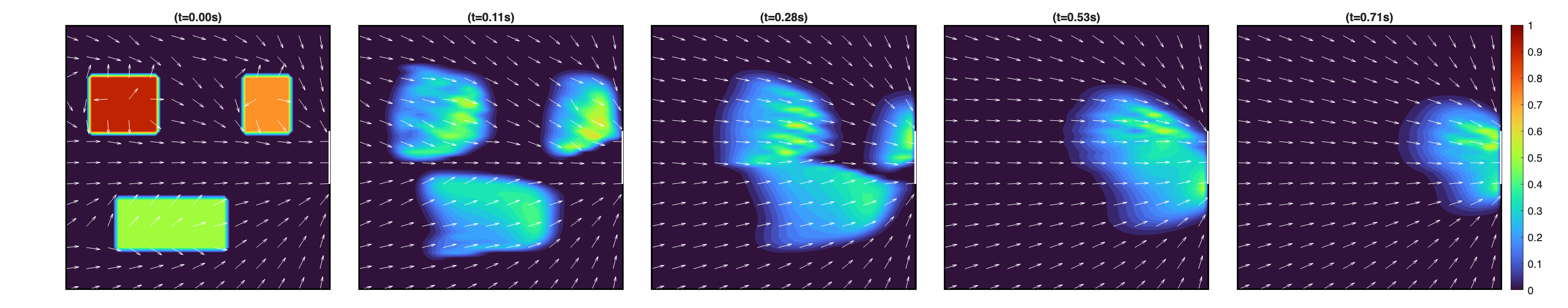}\\
\caption{The crowd density $\rho$ computed by the \eqref{evolution223} model at different time steps. Top row: $\he_1$. Middle row: $\he_2$. Bottom row: $\he_3$.}
	\label{fig:eg_H_HC}
\end{figure}
In the top rows (with cost $\he_1$) of both \cref{fig:eg_H_SC} and \cref{fig:eg_H_HC}, we observe that the deformation of the density occurs early and smoothly. Since the cost grows smoothly even at intermediate densities, pedestrians adjust their velocities to avoid dense regions. This results in a fluid merging of the three blocks. In contrast, the models using $\he_2$ and $\he_3$ display sharper densities. Since these costs remain relatively flat at low and medium densities and only blow up as $\rho\to 1$, pedestrians maintain straight paths and only start detouring once they collide. 
In \cref{fig:eg_H_SC} we see that densities are more diffuse than in \cref{fig:eg_H_HC}. The penalty function $\beta(p)$ seems to slow down pedestrians but allows a continuous wave-like compression of the crowd.

 \begin{figure}[H]
	\centering
	\includegraphics[width=0.6\columnwidth]{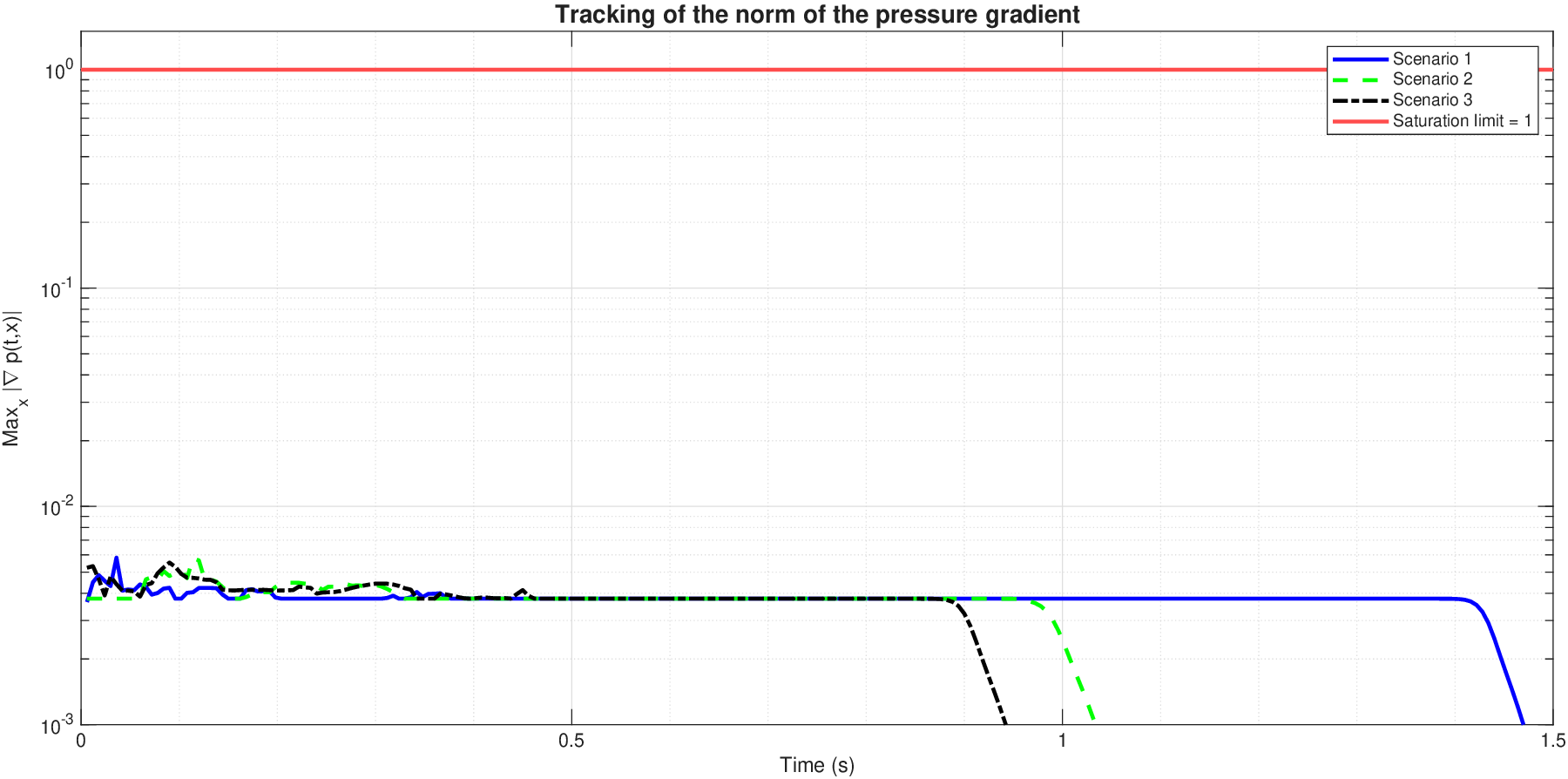}\\
\caption{Evolution of the maximum pressure gradient norm. The curves in blue, green, and black track the maximum spatial norm of the pressure $p$, plotted on a logarithmic scale. The red line represents the theoretical saturation limit, \ie $\vert \nabla p\vert =1$.}
	\label{fig:eg_pressure}
\end{figure}
In Scenario 1, the density is given by \eqref{eq:init-asym}, in Scenario 2 by \eqref{eq:rho0_scenario2}, while in Scenario 3 it is given by \eqref{eq:rho_annulus}. In all these examples, we make use of the classical Hughes' cost $\he_2(\rho) = \frac{1}{v(\rho)}$ with $v(\rho) = 1-\rho$. These numerical observations confirm what we stressed in \cref{sec:model}. The corrective flux $\Phi$ is never activated, confirming that the dynamic is driven by the standard  Hughes' anticipation cost.

In some sense, since the gradient $\nabla p$ may be interpreted as a measure of the interaction forces exerted between agents, the deviation mechanism induced by the Hughes-type strategy \eqref{VH1}-\eqref{VH2} naturally regulates these interactions and prevents such forces from reaching the critical threshold 
\[
|\nabla p|=1.
\]
 
As a consequence, at least from a numerical perspective, solving the Hughes' model supplemented with the corrective flux $\Phi$ appears to be essentially equivalent to solving the Hughes' model itself, since the corrective mechanism remains inactive whenever the congestion constraint is not saturated. We believe that the same phenomenon should persist when replacing \eqref{VH1} by the classical Hughes' choice 
\[
U[\rho]=-\rho v(\rho)^2\nabla\de.
\]
A rigorous proof of this property would provide a new route towards the well-posedness theory of the classical Hughes' model, a notoriously difficult problem that has remained largely open for more than two decades. 

\appendix
\section{Technical results}\label{app:tr}
This section is devoted to several technical lemmas and auxiliary results that were used throughout the paper.
\begin{lemma}\label{lem:app-lem-1}
The operator $\T: W^{1,s}_{D}(\Omega)\to \left(W^{1,s}_{D}(\Omega)\right)^{*}$ defined by
\begin{equation}\label{eq:op-T}
    \langle\T (u),\xi\rangle = \int_{\Omega}\beta(u)\xi \dd x + \int_{\Omega} \beta(u)\nabla\de\cdot\nabla \xi \dd x, \quad \mbox{for all } \xi\in W^{1,s}_{D}(\Omega),
\end{equation}
is pseudo-monotone\footnote{See, \textit{e.g.}, \cite[Definition 2.1]{LionsLivre}}.
\end{lemma}

\begin{proof}
Let $(u_n)_{n \in \mathbb{N}}$ be a sequence in $W^{1,s}_{D}(\Omega)$ such that $u_n \rightharpoonup u$ weakly in $W^{1,s}_{D}(\Omega)$ and $\limsup_{n \to \infty} \langle\T(u_n),u_n-u\rangle \leq 0$. It is clear that $\T$ is bounded. To establish pseudo-monotonicity, it remains to show that
\[
    \liminf_{n\to\infty} \langle\T(u_n),u_n-\xi\rangle \geq \langle\T(u),u-\xi\rangle, \quad \mbox{for all } \xi \in W^{1,s}_{D}(\Omega).
\]
We have, by definition of the operator,
\begin{equation}\label{eq:lem1-1}
	\langle\T(u_n),u_n-\xi\rangle = \int_{\Om}\beta(u_n) (u_n - \xi)\dd x + \int_{\Om}\beta(u_n)\nabla\de\cdot\nabla(u_n-\xi)\dd x.
\end{equation}
Since $s>N$, Rellich-Kondrachov's theorem ensures that the embedding $W^{1,s}_{D}(\Omega) \hookrightarrow C(\overline{\Omega})$ is compact. Consequently, the weak convergence of $(u_n)_n$ in $W^{1,s}_{D}(\Omega)$ implies its strong uniform convergence $u_n \to u$ in $C(\overline{\Omega})$. Since $\beta$ is continuous and bounded, it follows that $\beta(u_n) \to \beta(u)$ strongly in $C(\overline{\Omega})$, and thus in any $L^s(\Omega)$. 

For the first term of \eqref{eq:lem1-1}, the strong convergence of $\beta(u_n)$ and the weak convergence of $u_n - \xi$ allow us to pass to the limit. For the second term, the product $\beta(u_n)\nabla\de$ converges strongly to $\beta(u)\nabla\de$, while $\nabla(u_n-\xi)$ converges weakly in $L^s(\Omega)^N$. Consequently, passing to the limit as $n \to \infty$ in \eqref{eq:lem1-1}, we obtain
\begin{equation*}
\begin{aligned}
    \lim_{n \to \infty} \langle\T(u_n),u_n-\xi\rangle &= \lim_{n \to \infty} \left( \int_{\Om}\beta(u_n) (u_n - \xi)\dd x + \int_{\Om}\beta(u_n)\nabla\de\cdot\nabla(u_n-\xi)\dd x \right) \\
    &= \int_{\Om}\beta(u) (u - \xi)\dd x + \int_{\Om}\beta(u)\nabla\de\cdot\nabla(u-\xi)\dd x \\
    &= \langle\T(u),u-\xi\rangle.
\end{aligned}
\end{equation*}
Thus, the $\liminf$ condition is trivially satisfied. This completes the proof.
\end{proof}
The following result can be found, for example, in \cite[Theorem 8.1]{LionsLivre}.
\begin{theorem}\label{thm:app-thm1}
Let $\V$ be a separable reflexive Banach space, and let $\W$ be a non-empty, closed, convex, and bounded subset of $\V$. Let $\T:\W\to\V^*$ be a pseudo-monotone operator. Then, for any $f\in\V^{*}$, there exists $u\in\W$ such that 
\[
    \langle\T(u),u-v\rangle \leq \langle f,u-v\rangle \quad \text{for all } v\in \W.
\]
\end{theorem}

\begin{lemma}[Aubin-Lions-Simon \cite{LionsLivre}]\label{lem:ALS}
Let $X_0$, $X_1$, and $X_2$ be three Banach spaces. Assume that the embedding of $X_0$ into $X_1$ is compact and that the embedding of $X_1$ into $X_2$ is continuous. For any $T > 0$ and $1 \leq s, r \leq \infty$, define the space
\[
    W_{s,r} := \{u \in L^s(0,T; X_0) : \partial_t u \in L^r(0,T; X_2)\}.
\]
Then, the following compact embeddings hold:
\begin{itemize}
    \item If $s = \infty$ and $r > 1$, then the embedding of $W_{\infty,r}$ into $C([0,T]; X_1)$ is compact.
    \item If $s < \infty$ and $r \geq 1$, then the embedding of $W_{s,r}$ into $L^s(0,T; X_1)$ is compact.
\end{itemize}
\end{lemma}
\section{On the discrete operators}\label{section:discrete_op}
\setcounter{proposition}{0}
\renewcommand{\theproposition}{\Alph{section}\arabic{proposition}}

In this section, we recall some details concerning the discrete divergence and gradient operators used in \cref{section:4}. First, let us recall that the space $X = \R^{m\times n}$ is equipped with the inner product and its associated norm
\[
\langle u,v \rangle = h^2\sum_{i=1}^{m}\sum_{j=1}^{n} u_{i,j} v_{i,j} \quad \mbox{ and } \quad \Vert u\Vert = \sqrt{\langle u, u\rangle},
\] 
where $h$ is the spatial mesh size. 

The discrete divergence operator $\dive_h : Y \longrightarrow X$ is defined for a given vector field $\Phi = (\Phi^1, \Phi^2)$ by
\begin{equation}\label{eq:div1}
    (\dive_h\Phi)_{i, j} = \frac{\Phi^{1}_{i+\frac{1}{2},j} - \Phi^{1}_{i-\frac{1}{2},j}}{h}  +\frac{\Phi^{2}_{i,j+\frac{1}{2}} - \Phi^{2}_{i,j-\frac{1}{2}}}{h}.
\end{equation}

Accordingly, the discrete gradient $\nabla_h: X \longrightarrow Y=\R^{(m+1)\times n} \times \R^{m\times (n+1)}$ is given by $(\nabla_h u)_{i,j} = \Big((\nabla_h u)^{1}_{i,j},(\nabla_h u)^{2}_{i,j}\Big)$, where the components depend on the boundary conditions,
\begin{equation}
    \label{eq:grad_h2}
    \begin{aligned}
        (\nabla_{h}u)^1_{i,j} &= -\left((D^1_p)^\top u\right)_{i,j}, \quad \text{if } ((m+\textstyle{\frac{1}{2}})h, j h) \in \Gamma_D,\\
        (\nabla_{h}u)^1_{i,j} &= -\left((D^1_m)^\top u\right)_{i,j}, \quad \text{if } ((m+\textstyle{\frac{1}{2}})h, j h) \in \Gamma_N,\\
        (\nabla_{h}u)^2_{i,j} &= -\left((D^2)^\top u\right)_{i,j}.
    \end{aligned}
\end{equation}

Here, the 1D finite difference matrices $D^{1}_{p}, D^{1}_{m}$, and $D^{2}$ are given by
{\footnotesize
\[
D^{1}_{p}= \begin{pmatrix}
        0 & 1/h & 0 &\cdots &&  & 0 \\
        0 & - 1/h &  1/h & 0 &\cdots&   & 0 \\
        0 & 0 & - 1/h &  1/h & 0 &\cdots & 0 \\
        \vdots &&  \vdots & \ddots &&  \vdots \\
        0 & 0 & \cdots & & 0 & - 1/h &  1/h
\end{pmatrix}
\]
\[
D^1_m= \begin{pmatrix}
        0 & 1/h & 0 &\cdots &&  & 0 \\
        0 & - 1/h &  1/h & 0 &\cdots&   & 0 \\
        0 & 0 & - 1/h &  1/h & 0 &\cdots & 0 \\
        \vdots &&  \vdots & \ddots &&  \vdots \\
        0 & 0 & \cdots & & 0 & - 1/h &  0
\end{pmatrix}
\]
and
\[
D^2= \begin{pmatrix}
        0 & 1/h & 0 &\cdots &&  & 0 \\
        0 & - 1/h &  1/h & 0 &\cdots&   & 0 \\
        0 & 0 & - 1/h &  1/h & 0 &\cdots & 0 \\
        \vdots &&  \vdots & \ddots &&  \vdots \\
        0 & 0 & \cdots & & 0 & - 1/h &  0
\end{pmatrix}.
\]
} 

With these definitions in hand, one can readily verify that $-\dive_h$ and $\nabla_h$ are formal adjoints. We conclude by recalling the following classical result.

\begin{proposition}[\cite{chambolle2011first}]\label{prop:adj_norm}
    Under the above-mentioned definitions and inner products, we have:
    \begin{itemize}
        \item The adjoint operator of the discrete gradient is $\nabla^{*}_h = -\operatorname{div}_h$.
        \item The operator norm satisfies: $\Vert \nabla_h \Vert^2 = \Vert \operatorname{div}_h\Vert^2 \leq \frac{8}{h^2}$.
    \end{itemize}
\end{proposition}


\section{Application of the primal-dual algorithm}\label{appendix:pd}

\subsection{Computation of the velocity field $V$}

In model \eqref{evolution223}, the predicted density $\rho^{k+\frac{1}{2}}$ is calculated using equation \eqref{t2}. This requires knowing the value of the velocity field $V(t_k,\cdot)=-\nabla \de(t_k,\cdot)$ at the instant $t_k$. To compute the potential $\de(t_k,\cdot)$, we solve the Eikonal equation
\begin{equation}\label{eikp}
	\left\{ \begin{array}{ll}
		\vert \nabla \de(t_k,\cdot)\vert =\he(p(t_k,\cdot))\quad & \hbox{ in }\Omega,\\  
		\de(t_k,\cdot) =0 & \hbox{ on }\Gamma_D,
	\end{array}
	\right.
\end{equation}
where $p(t_k,\cdot)=p^k$ is the congestion pressure obtained from the dual problem \eqref{dual221}. 

Recall that the solution of \eqref{eikp} can be obtained by solving the maximization problem
\begin{equation}
	\label{vitesse2}
	\max_{z \in W^{1,\infty}_D(\Om)}\left\{  \int_\Omega  z \: \dd x  : \vert  \nabla z\vert \leq \he(p^k) \mbox{ a.e. in } \Omega  \right\},
\end{equation}
which can be recast as the minimization problem
\begin{equation}
	\inf_{z \in  W^{1,\infty}_{D}(\Om) } \mathcal{F}(z)+ \mathcal{G} (\nabla z),
\end{equation}
where $\mathcal{F}(z) = - \int_\Omega  z \: \dd x$, and $\mathcal{G}(\mathbf{q}) = \I_K(\mathbf{q})$ is the indicator function of the convex set 
\[
K = \big\{\mathbf{q} \in L^\infty(\Omega; \R^2) : \Vert \mathbf{q}(x)\Vert \leq \he(p^k(x)) \mbox{ a.e.} \big\}.
\]
Once again, the approximated solution is efficiently computed by applying the Chambolle-Pock (PD) algorithm to this primal-dual formulation (see \eg  \cite{EQE}).

\section*{Acknowledgments}

This publication is based upon work supported by King Abdullah University of Science and Technology (KAUST) under Award No. ORFS-CRG12-2024-6430. The work of H.E. was supported by the FMJH Program PGMO (grant no. P-2024-0019).

\bibliographystyle{plainurl}

\bibliography{cm2}

\end{document}